\pdfoutput=1

\documentclass[preprint,10pt]{elsarticle}
\usepackage[utf8]{inputenc}

\usepackage{fullpage}
\usepackage{verbatimbox}

\usepackage{amsmath, amsfonts, amssymb,mathdots}
\usepackage{amsthm} % remove for SISC
\usepackage{hyperref}
\hypersetup{colorlinks,citecolor=blue}

\usepackage{hypcap}
\usepackage{stmaryrd} % for jump

\usepackage{stackengine}
\stackMath
\newlength\matfield
\newlength\tmplength
\def\matscale{1.}
\newcommand\dimbox[3]{%
  \setlength\matfield{\matscale\baselineskip}%
  \setbox0=\hbox{\vphantom{X}\smash{#3}}%
  \setlength{\tmplength}{#1\matfield-\ht0-\dp0}%
  \fboxrule=1pt\fboxsep=-\fboxrule\relax%
  \fbox{\makebox[#2\matfield]{\addstackgap[.5\tmplength]{\box0}}}%
}
\newcommand\raiserows[2]{%
   \setlength\matfield{\matscale\baselineskip}%
   \raisebox{#1\matfield}{#2}%
}
\newcommand\matbox[5]{
  \stackunder{\dimbox{#1}{#2}{$#5$}}{\scriptstyle(#3\times #4)}%
}
\usepackage{graphicx, color, bm}
\usepackage{subfig}
\usepackage{booktabs}

\usepackage{tikz,pgfplots,pgfplotstable}
\definecolor{markercolor}{RGB}{124.9, 255, 160.65}

\newtheorem{theorem}{Theorem}[section]

\newtheorem{remark}{Remark}

\renewcommand{\hat}{\widehat}
\renewcommand{\tilde}{\widetilde}

\newcommand*\diff[1]{\mathop{}\!{\mathrm{d}#1}}
\newcommand{\diag}[1]{{\rm diag}\LRp{#1}}
\newcommand{\td}[2]{\frac{{\rm d}#1}{{\rm d}{ {#2}}}}
\newcommand{\pd}[2]{\frac{\partial#1}{\partial#2}}
\newcommand{\nor}[1]{\left\| #1 \right\|}
\newcommand{\LRp}[1]{\left( #1 \right)}
\newcommand{\LRs}[1]{\left[ #1 \right]}

\newcommand{\LRb}[1]{\left| #1 \right|}
\newcommand{\LRc}[1]{\left\{ #1 \right\}}

\newcommand{\LRl}[1]{\left. #1 \right|}

\newcommand{\jump}[1] {\ensuremath{\llbracket#1\rrbracket}}
\newcommand{\avg}[1] {\ensuremath{\LRc{\!\LRc{#1}\!}}}

\newcommand{\note}[1]{{\color{blue}{#1}}}
\newcommand{\bnote}[1]{#1}
\newcommand{\rnote}[1]{#1}

\newcommand{\fnt}[1]{\bm{\mathsf{ #1}}}

\newcommand{\eq}[1]{\begin{align*}#1\end{align*}}
\newcommand{\eqlab}[1]{\begin{align}#1\end{align}}
\newcommand{\bmat}[1]{\begin{bmatrix}#1\end{bmatrix}}

\graphicspath{./figs}

\date{}

\begin{document}

%\maketitle

\begin{frontmatter}
\title{Efficient computation of Jacobian matrices for entropy stable summation-by-parts schemes}
\author{Jesse Chan, Christina G.\ Taylor}

%\author[rice]{Jesse Chan\corref{cor1}, Christina Taylor}
\ead{jesse.chan@rice.edu,cgt@rice.edu}
\address[rice]{Department of Computational and Applied Mathematics, Rice University, 6100 Main St, Houston, TX, 77005}

\begin{abstract}
Entropy stable schemes replicate an entropy inequality at the semi-discrete level.  These schemes rely on an algebraic summation-by-parts (SBP) structure and a technique referred to as flux differencing.  We provide simple and efficient formulas for Jacobian matrices for the semi-discrete systems of ODEs produced by entropy stable discretizations.  These formulas are derived based on the structure of flux differencing and derivatives of flux functions, which can be computed using automatic differentiation (AD).  Numerical results demonstrate the efficiency and utility of these Jacobian formulas, which are then used in the context of two-derivative explicit time-stepping schemes and implicit time-stepping.   
\end{abstract}

\end{frontmatter}

\section{Introduction}

This paper is concerned with the numerical discretization of systems of nonlinear conservation laws. In particular, we focus on the computation of Jacobian matrices for nonlinear residuals associated with entropy conservative and entropy stable semi-discretizations. Such matrices are useful in the context of implicit time-stepping schemes \cite{persson2008newton}, as well as adjoint-based sensitivity computations and optimization \cite{ulbrich2002sensitivity, gunzburger2003perspectives}.

Entropy stable discretizations mimic a continuous dissipation of entropy for nonlinear conservation laws. Let $\Omega$ denote some domain with boundary $\partial \Omega$.  Nonlinear conservation laws are expressed as a system of nonlinear partial differential equations (PDEs) 
\begin{equation}
\pd{\bm{u}}{t}  + \sum_{i=1}^d\pd{\bm{f}_i(\bm{u})}{x_i} = 0, \qquad 
S(\bm{u}) \text{ convex}, \qquad
\bm{v}(\bm{u}) = \pd{S}{\bm{u}},
\label{eq:nonlineqs}
\end{equation}
where $\bm{u}\in \mathbb{R}^n$ are the conservative variables, $\bm{f}_i$ are nonlinear fluxes, and $\bm{v}(\bm{u})$ are the \textit{entropy variables} with respect to the entropy $S(\bm{u})$.  By multiplying (\ref{eq:nonlineqs}) by the entropy variables, vanishing viscosity solutions \cite{kruvzkov1970first} of many fluid systems \cite{hughes1986new, chen2017entropy} can be shown to satisfy the following entropy inequality
\begin{equation}
\int_{\Omega}\pd{S(\bm{u})}{t} + \sum_{i=1}^d \int_{\partial \Omega} \LRp{\bm{v}^T\bm{f}_i(\bm{u}) - \psi_i(\bm{u})}n_i \leq 0\label{eq:entropyineq},
\end{equation}
where $n_i$ denotes the $i$th component of the outward normal vector \rnote{and $\psi_i(\bm{u})$ denotes the entropy potential in the $i$th coordinate}.  The entropy inequality (\ref{eq:entropyineq}) is a statement of stability for nonlinear conservation laws \cite{mock1980systems, harten1983symmetric}.  

High order entropy stable schemes (see for example \cite{carpenter2014entropy, gassner2016split, chen2017entropy, crean2018entropy, chan2017discretely, parsani2016entropy, fernandez2019staggered}) reproduce this entropy inequality at the semi-discrete level.  The resulting methods display significantly improved robustness while retaining high order accuracy \cite{winters2018comparative, rojas2019robustness}.  These schemes are based on entropy conservative finite volume fluxes \cite{tadmor1987numerical}, which are extended to high order discretizations through a procedure referred to as flux differencing.  These methods have mainly been tested in the context of explicit time-stepping.  However, recent works have applied entropy stable methods to both the space-time and implicit settings \cite{friedrich2018entropy, hicken2020entropy}.  

Both space-time and implicit time discretizations require the solution of a system of nonlinear equations.  This can be done using Newton's method, which involves the Jacobian matrix of the nonlinear equations.  While it is possible to compute the solution to the nonlinear system without explicitly computing the Jacobian matrix using Jacobian-free Newton-Krylov methods \cite{knoll2004jacobian, birken2019subcell}, the Jacobian matrix is commonly used to construct preconditioners \cite{persson2008newton}.  

In this work, we present efficient formulas for Jacobian matrices of systems resulting from entropy stable formulations.  We also show that computing the Jacobian matrix is not significantly more expensive than evaluating the residual of the nonlinear system.  Finally, we apply the new Jacobian formulas to both explicit two-derivative and implicit time-stepping schemes.

%, which are defined as follows.  Let $\bm{u}_L,\bm{u}_R$ denote left and right solution states.  Entropy conservative fluxes in the sense of Tadmor \cite{tadmor1987numerical} are defined as a set of bivariate symmetric and consistent functions $\bm{f}_{i,S}(\bm{u}_L,\bm{u}_R)$ such that 
%\[
%\LRp{\bm{v}_L -\bm{v}_R}^T\bm{f}_{i,S}(\bm{u}_L,\bm{u}_R) = \psi_i(\bm{u}_L) - \psi_i(\bm{u}_R), \qquad i = 1,\ldots,d.
%\]
%\note{Todo: finish}

\subsection{On notation}

The notation in this paper is motivated by notation in \cite{crean2018entropy, fernandez2019entropy}. Unless otherwise specified, vector and matrices are denoted using lower and upper case bold font, respectively.  We denote spatially quantities \bnote{related to the spatial discretization (e.g., operators for differentiation, interpolation, or quadrature)} using a bold sans serif font. Finally, continuous functions \bnote{with vector arguments are interpreted as applying the continuous function to each entry of the vector.}

For example, if $\fnt{x}$ denotes a vector of point locations, i.e., $(\fnt{x})_i = \bm{x}_i$, then $u(\fnt{x})$ is interpreted as the vector 
\[	({u}(\fnt{x}))_i = {u}(\bm{x}_i).
\]
Similarly, if $\fnt{u} = {u}(\fnt{x})$, then ${f}(\fnt{u})$ corresponds to the vector
\[
	({f}(\fnt{u}))_i = {f}(u(\bm{x}_i)).
\]
Vector-valued functions are treated similarly. For example, given a vector-valued function $\bm{f}:\mathbb{R}^n\rightarrow \mathbb{R}^n$ and a vector \rnote{$\fnt{u}$ with vector-valued entries $\fnt{u}_i = \bm{u}_i \in \mathbb{R}^n$}, $\LRp{\bm{f}(\fnt{u})}_i = \bm{f}(\bm{u}_i)$.

\section{Jacobian matrix formulas for entropy conservative schemes}

For clarity of presentation, we consider first a scalar nonlinear conservation law in one spatial dimension
\eqlab{
\pd{u}{t} + \pd{f(u)}{x} = 0. \label{eq:ncl}
}
We assume periodic boundary conditions, which will simplify the presentation of the main results.  
\rnote{Non-periodic boundaries are treated in Section~\ref{sec:nonperiodic}}.  
The generalization to systems of nonlinear conservation laws is postponed until Section~\ref{sec:systems}.  

Let $f_S(x,y)$ denote a bivariate scalar flux function which is symmetric and consistent.  Suppose $\fnt{u}$ is a vector of nodal values of the solution.  Define the vector $\fnt{r} = \fnt{r}(\fnt{u})$ approximating the flux derivative $\pd{f(u)}{x}$ as 
\eqlab{
\fnt{r}(\fnt{u}) = 2\LRp{\fnt{Q}\circ \fnt{F}}\bm{1}, \qquad \fnt{F}_{ij} = f_S(\fnt{u}_i,\fnt{u}_j),
\label{eq:fu}
}
where $\fnt{Q}$ is a discretization matrix to be specified later \rnote{and $\circ$ denotes the matrix Hadamard product}.  The simplest entropy stable numerical schemes based on flux differencing discretize (\ref{eq:ncl}) via the system of ODEs
\[
\fnt{M}\td{\fnt{u}}{t} + \fnt{r}(\fnt{u}) = \fnt{0}.
\]
where $\fnt{M}$ is a diagonal mass (norm) matrix with positive entries.  If $f_S(x,y)$ is entropy conservative (in the sense of \cite{tadmor1987numerical}) and $\fnt{Q}$ is skew-symmetric, then the resulting scheme is also discretely entropy conservative.  An entropy stable scheme can be constructed from an entropy conservative scheme by adding appropriate terms which dissipate entropy \cite{chen2017entropy, upperman2019entropy, hicken2020entropy}.

We are interested in computing the Jacobian matrix $\pd{\fnt{r}}{\fnt{u}}$. Let $\diag{\fnt{x}}$ denote the diagonal matrix with the vector $\fnt{x}$ on the diagonal and let $\diag{\fnt{A}}$ denote the vector diagonal of $\fnt{A}$.  We then have the following theorem:
\begin{theorem}
\label{thm:explicitJ}
Suppose that $\fnt{Q} = \pm\fnt{Q}^T$.  Then, the Jacobian matrix of the entropy conservative scheme (\ref{eq:fu}) can be expressed as either
\eq{
\pd{\fnt{r}}{\fnt{u}} &= 2\LRp{\fnt{Q}\circ \fnt{F}_y} \pm \diag{\fnt{1}^T\LRp{2\fnt{Q}\circ \fnt{F}_y}}\\
\pd{\fnt{r}}{\fnt{u}} &= 2\LRp{\fnt{Q}\circ \fnt{F}_x^T} \pm \diag{\LRp{2\fnt{Q} \circ \fnt{F}_x}\fnt{1}}
}
where the matrices $\fnt{F}_x,\fnt{F}_y$ are
\[
\LRp{\fnt{F}_x}_{ij} = \LRl{\pd{f_S}{x}}_{\fnt{u}_i,\fnt{u}_j}, \qquad \LRp{\fnt{F}_y}_{ij} = \LRl{\pd{f_S}{y}}_{\fnt{u}_i,\fnt{u}_j}.
\]
\end{theorem}
\begin{proof}
We will prove the first formula involving $\fnt{F}_y$.  The second formula follows via symmetry and similar steps.  By the chain rule,
\eq{
\LRp{\pd{\fnt{r}}{\fnt{u}}}_{ij} &= \pd{\fnt{r}_i}{\fnt{u}_j} =  \sum_{k} 2\fnt{Q}_{ik} \pd{}{\fnt{u}_j}f_S\LRp{\fnt{u}_i,\fnt{u}_k}= \sum_{k} 2\fnt{Q}_{ik} \LRp{\LRl{ \pd{f_S}{x}}_{\fnt{u}_i,\fnt{u}_k}\pd{\fnt{u}_i}{\fnt{u}_j} + \LRl{\pd{f_S}{y}}_{\fnt{u}_i,\fnt{u}_k} \pd{\fnt{u}_k}{\fnt{u}_j}}
}
If $i\neq j$, then $\pd{\fnt{u}_i}{\fnt{u}_j} = \delta_{ij} = 0$. \rnote{Moreover}, most terms in the sum over $k$ vanish except for $k=j$. Since $\pd{\fnt{u}_k}{\fnt{u}_j} = 1$ \bnote{for $k=j$}, the formula reduces to 
\eq{
\pd{\fnt{r}_i}{\fnt{u}_j} &= 2\fnt{Q}_{ij} \LRl{\pd{f_S}{y}}_{\fnt{u}_i,\fnt{u}_j}.
}
When $i=j$, $\pd{\fnt{u}_i}{\fnt{u}_j} = \pd{\fnt{u}_i}{\fnt{u}_i} = 1$, and 
\eq{
\pd{\fnt{r}_i}{\fnt{u}_i} &= 
\LRp{\sum_{k} 2 \fnt{Q}_{ik} \LRl{\pd{f_S}{x}}_{\fnt{u}_i,\fnt{u}_k}} + 2\fnt{Q}_{ii}\LRl{\pd{f_S}{y}}_{\fnt{u}_i,\fnt{u}_i}.
}
The term $2\fnt{Q}_{ii}\LRl{\pd{f_S}{y}}_{\fnt{u}_i,\fnt{u}_i}$ is the diagonal of the matrix $2\LRp{\fnt{Q}\circ \fnt{F}_y}$, and we can simplify the first summation term.  By the symmetry of $f_S(x,y)$, we have that
\[
\LRl{\pd{f_S}{y}}_{x,y} = \LRl{\pd{f_S}{x}}_{y,x}
\]
Thus, by $\fnt{Q} = \pm\fnt{Q}^T$, 
\eq{
\sum_k 2\fnt{Q}_{ik} \LRl{\pd{f_S}{x}}_{\fnt{u}_i,\fnt{u}_k} &= \sum_k 2\fnt{Q}_{ik} \LRl{\pd{f_S}{y}}_{\fnt{u}_k,\fnt{u}_i}= \LRp{\LRp{2\fnt{Q}\circ \fnt{F}_y^T} \fnt{1}}_i = \LRp{\pm\fnt{1}^T\LRp{2\fnt{Q}\circ \fnt{F}_y}}_i.
}
\end{proof}
%Rearranging the steps of this proof implies that the Jacobian can also be expressed as 
While we consider only symmetric and skew-symmetric matrices $\fnt{Q}$ in this work, one can use this theorem to compute the Jacobian $\pd{\fnt{r}}{\fnt{u}}$ for arbitrary matrices $\fnt{Q}$ since any real matrix can be decomposed into symmetric and skew parts
\[
\fnt{Q} = \frac{1}{2} \LRp{\fnt{Q}+\fnt{Q}^T} +  \frac{1}{2} \LRp{\fnt{Q}-\fnt{Q}^T}.
\]
Two applications of Theorem~\ref{thm:explicitJ} then provide a formula for the Jacobian of (\ref{eq:fu}).  

\subsection{Computing derivatives of bivariate flux functions}

The aforementioned proofs require partial derivatives of flux functions $f_S(u_L,u_R)$ with respect to at least one argument.  This can be done by hand for simple fluxes. For example, for the Burgers' equation, the flux and its derivative are
\[
f_S(u_L,u_R) = \frac{1}{6}\LRp{u_L^2 + u_Lu_R + u_R^2}, \qquad \pd{f_S}{u_R} = \frac{1}{6}\LRp{u_L + 2u_R}.
\]
However, this procedure can become cumbersome for complex or piecewise-defined flux functions such as the logarithmic mean \cite{ismail2009affordable, winters2019entropy}.  This can be avoided by using Automatic Differentiation (AD) \cite{griewank2008evaluating}.  AD is distinct from both symbolic differentiation and finite difference approximations in that it does not return an explicit expression, but constructs a separate function which evaluates the derivative accurately up to machine precision.

In this work, we utilize the Julia implementation of forward-mode automatic differentiation provided by \verb+ForwardDiff.jl+ \cite{RevelsLubinPapamarkou2016}.  The procedure is remarkably simple: given some flux function \verb+f(x,y)+, \verb+ForwardDiff.jl+ returns the derivative with respect to either $x$ or $y$ as another function.  For example, defining the function $\LRl{\pd{f}{y}}_{x,y}$ is a one-line operation:
\begin{verbbox}
dfdy(x,y) = ForwardDiff.derivative(y->f(x,y),y)
\end{verbbox}
\capstartfalse
\begin{figure}[!h]
\centering
\theverbbox
\end{figure}
\capstarttrue
\\
\verb+ForwardDiff.jacobian+ is the analogous routine for computing Jacobians of vector-valued flux functions. This simple API utilizes the flexible Julia type system \cite{bezanson2017julia}.\footnote{In practice, derivative and Jacobian functions are initialized with information about the size and data type of the input to ensure type stability in Julia.}

Automatic differentiation can be directly applied to $\fnt{r}(\fnt{u})$ to compute the Jacobian matrix.  However, because AD scales with the number of inputs and outputs, the cost of applying AD directly to $\fnt{r}(\fnt{u})$ increases as the discretization resolution increases.  In contrast, using the approach in this paper, AD is applied only to the flux function, which has a small fixed number of inputs and outputs which are independent of the discretization resolution.  As a result, the cost of evaluating derivatives of the flux function is roughly the same as the cost of evaluating the flux function itself and entries of the Jacobian matrix can be computed for roughly the same cost as a single evaluation of the nonlinear term $\fnt{r}(\fnt{u})$.  Moreover, when computing the Jacobian matrix, the formula in Theorem~\ref{thm:explicitJ} makes it simpler to directly take advantage of sparsity in $\fnt{Q}$ without having to perform graph coloring \cite{coleman1998efficient}. 

\section{Examples of discretization matrices which appear in entropy conservative numerical schemes}

In this section, we give some examples of matrices $\fnt{Q}$ which appear in entropy stable numerical discretizations.  We assume periodicity, which corresponds to a skew-symmetric structure for $\fnt{Q}$. Non-periodic domains are treated later. 

\subsection{Finite volume methods}

The spatial discretization for \bnote{most second order} finite volume schemes can be reformulated in terms of (\ref{eq:fu}) \cite{chan2019entropy}.  Suppose that the 1D interval $[-1,1]$ is decomposed into $K$ non-overlapping elements of size $h$.  An entropy conservative finite volume scheme is given as
\eq{
\td{\fnt{u}_1}{t} &+ \frac{f_S(\fnt{u}_{2},\fnt{u}_1) - f_S(\fnt{u}_1,\fnt{u}_{K})}{h} = 0\\
\td{\fnt{u}_i}{t} &+ \frac{f_S(\fnt{u}_{i+1},\fnt{u}_i) - f_S(\fnt{u}_i,\fnt{u}_{i-1})}{h} = 0, \qquad i = 2,\ldots, K-1,\\
\td{\fnt{u}_K}{t} &+ \frac{f_S(\fnt{u}_{1},\fnt{u}_K) - f_S(\fnt{u}_K,\fnt{u}_{K-1})}{h} = 0,
}
where $\fnt{u}_i$ denotes the average value of the solution on each element and $f_S$ is an entropy conservative flux.  Let $\fnt{M} = h\fnt{I}$ and let $\fnt{Q}$ be the periodic second-order central difference matrix
\[
\fnt{Q} = \frac{1}{2}\begin{bmatrix}
0 & 1 & &\ldots & -1\\
-1 & 0 & 1 &&  \\
& -1 & 0 & 1 &  \\
 & & & \ddots &  \\
1 & &\ldots  & -1 & 0
\end{bmatrix}.
\]
An entropy conservative finite volume scheme is then equivalent to 
\[
\td{\fnt{u}}{t} + 2\LRp{\fnt{Q}\circ\fnt{F}}\fnt{1} = \fnt{0}, \qquad \fnt{F}_{ij} = f_S(u_i,u_j)
\]
where $\fnt{u} = \LRs{\fnt{u}_1,\ldots, \fnt{u}_K}^T$ is the vector of solution values.  

\subsection{Multi-block summation-by-parts finite differences and discontinuous Galerkin spectral element methods}

We consider next a multi-element summation-by-parts (SBP) finite element discretization \cite{kreiss1974finite, carpenter1999stable}.  Suppose again that a one-dimensional domain $\Omega$ is decomposed into $K$ non-overlapping elements $D^k$ of size $h$.  Let ${\fnt{M}}$ and ${\fnt{Q}} \in \mathbb{R}^{N_p\times N_p}$ denote diagonal mass (norm) and nodal differentiation matrices such that ${\fnt{M}}^{-1}{\fnt{Q}}$ approximates the first derivative on a reference interval and is exact for polynomials up to degree $N$.  The operators ${\fnt{M}}, {\fnt{Q}}$ satisfy an SBP property if
\eqlab{
{\fnt{Q}}+{\fnt{Q}}^T = {\fnt{B}}, \qquad 
{\fnt{B}} = \bmat{
-1 & & & \\
& 0 & & \\
& & \ddots & \\
& & & 1 
}.
\label{eq:fvQ}
}
We note that nodal discontinuous Galerkin spectral element (DG-SEM) discretizations \cite{kopriva2009implementing} also fall into a SBP framework \cite{gassner2013skew} and are thus also included in this framework. Since the finite volume methods described in the previous section can be interpreted as DG methods with polynomial degree $p = 0$, they also fall into this framework. 

These matrices can be used to construct entropy conservative high order discretizations.  Let $J = h/2$ be the Jacobian of the mapping from the reference element $[-1,1]$ to a physical interval of size $h$ and let $\fnt{F}^k_{ij} = f_S(\fnt{u}_{i,k},\fnt{u}_{j,k})$ denote the matrix of flux interactions between different nodes on the element $D^k$.  A local formulation on the element $D^k$ is given by
\begin{gather}
J_k {\fnt{M}} \td{\fnt{u}_k}{t} + 2\LRp{{\fnt{Q}} \circ \fnt{F}^k}\fnt{1} + \fnt{B}\LRp{\fnt{f}^*-f(\fnt{u}_k)}  = \fnt{0}, \nonumber\\
\fnt{f}^* = \bmat{
f_S(\fnt{u}_{1,k}^+,\fnt{u}_{1,k})\\
0\\
\vdots\\
0\\
f_S(\fnt{u}_{N_p,k}^+,\fnt{u}_{N_p,k})
}.
\label{eq:sbp1D}
\end{gather}
where $\fnt{u}_{1,k}^+, \fnt{u}_{N_p,k}^+$ denote the exterior values of $\fnt{u}_{1,k}, \fnt{u}_{N_p,k}$ on neighboring elements.  Assuming that the elements are ordered from left to right in ascending order, for interior element indices $1 < k < K$, these are given by
\[
\fnt{u}_{1,k}^+ = \fnt{u}_{N_p,k-1}, \qquad \fnt{u}_{N_p,k}^+ = \fnt{u}_{1,k+1}.
\]
In other words, the first node on $D^k$ is connected to the last node on the previous element, and the last node on $D^k$ is connected to the first node on the next element.

For periodic boundary conditions this local formulation can be understood as inducing a global skew-symmetric matrix.  To show this, we first use the SBP property to rewrite (\ref{eq:sbp1D}) in a skew-symmetric form \cite{chan2019skew}
\eq{
J_k {\fnt{M}} \td{\fnt{u}_k}{t} + \LRp{\LRp{\fnt{Q}-\fnt{Q}^T} \circ \fnt{F}^k}\fnt{1} + \fnt{B}\fnt{f}^* = \fnt{0}.
}
We now define a global vector $\fnt{u}_{\Omega} = \LRs{\fnt{u}_1, \fnt{u}_2, \ldots, \fnt{u}_K}^T$.  Let the global flux matrix be defined as
\[
\fnt{F} = \bmat{
\fnt{F}_{11} & \ldots & \fnt{F}_{1K}\\
\vdots & \ddots & \vdots\\
\fnt{F}_{K1} & \ldots & \fnt{F}_{KK}
}, \qquad \LRp{\fnt{F}_{k_1,k_2}}_{ij} = f_S(\fnt{u}_{i,k_1}, \fnt{u}_{j,k_2}).
\]
The blocks of the matrix $\fnt{F}$ capture flux interactions between solution values at different nodes and elements.  
The local formulations can now be concatenated into a single skew-symmetric matrix
\eqlab{
\fnt{M}_{\Omega}\td{\fnt{u}_{\Omega}}{t} + 2\LRp{\fnt{Q}_{\Omega}\circ \fnt{F}}\fnt{1} = \fnt{0},
\label{eq:sbpform}
}
where $\fnt{M}_{\Omega}$ is the block-diagonal matrix with blocks $J_k\fnt{M}$, and 
\eqlab{
\fnt{Q}_{\Omega} = \frac{1}{2}\bmat{
\fnt{S} &\fnt{B}_R & & -\fnt{B}_L\\
-\fnt{B}_L& \fnt{S} & \fnt{B}_R&\\
& -\fnt{B}_L & \ddots & \fnt{B}_R\\
\fnt{B}_R& & -\fnt{B}_L & \fnt{S}
}, \qquad \fnt{S} = \LRp{\fnt{Q}-\fnt{Q}^T},
\label{eq:sbpmat}
}
where the matrices $\fnt{B}_L,\fnt{B}_R$ are zeros except for a single entry
\eqlab{
\fnt{B}_L = \bmat{
 & &1\\
& \iddots & \\
0 & & 
}, \qquad \fnt{B}_R = \fnt{B}_L^T = \bmat{
 & & 0\\
& \iddots & \\
1 & & 
}
\label{eq:BLBR}
}
The matrix $\fnt{Q}_{\Omega}$ can be considered a high order generalization of the finite volume matrix (\ref{eq:fvQ}).  Similar ``global SBP operator'' approaches were used to construct simultaneous approximation (SBP-SAT) interface coupling terms in \cite{crean2018entropy, chan2018efficient, fernandez2019entropy}. 

\subsubsection{\rnote{Non-periodic boundary conditions}}
\label{sec:nonperiodic}

\rnote{For non-periodic domains, the structure of the global differentiation matrix $\fnt{Q}_{\Omega}$ changes. For finite volume, DG, and multi-block SBP methods, boundary conditions are typically imposed by specifying appropriate ``exterior'' values $\fnt{u}_{i,k}^+$ in flux expressions such as (\ref{eq:sbp1D}). The resulting formulation is a small modification of (\ref{eq:sbpform})
\eq{
\fnt{M}_{\Omega}\td{\fnt{u}_{\Omega}}{t} + 2\LRp{\fnt{Q}_{\Omega}\circ \fnt{F}}\fnt{1} + \fnt{B}_{\Omega}\fnt{f}_{\Omega}^* = \fnt{0},
}
where $\fnt{Q}_{\Omega}$, $\fnt{B}_{\Omega}$, and $\fnt{f}_{\Omega}^*$ are now given by
\begin{gather*}
\fnt{Q}_{\Omega} = \frac{1}{2}\bmat{
\fnt{S} &\fnt{B}_R & & \\
-\fnt{B}_L& \fnt{S} & \fnt{B}_R&\\
& -\fnt{B}_L & \ddots & \fnt{B}_R\\
& & -\fnt{B}_L & \fnt{S}
}, \qquad \fnt{B}_{\Omega} = \bmat{
-\fnt{B}_L &&&\\
&\fnt{0}&&\\
&&\ddots &\\
&&&\fnt{B}_R}, \qquad 
\fnt{f}_{\Omega}^* = \bmat{
f_S(\fnt{u}_{1,1},\fnt{u}_{1,1}^+)\\
0\\
\vdots\\
f_S(\fnt{u}_{N_p,K},\fnt{u}_{N_p,K}^+)
}.
\end{gather*}
Here, $\fnt{u}_{1,1}^+$ and $\fnt{u}_{N_p,K}^+$ denote the exterior values at the left and right endpoints, respectively.  Since $\fnt{Q}_{\Omega}$ is still skew-symmetric, we can reuse the formulas from Theorem~\ref{thm:explicitJ}. The term $\fnt{B}_{\Omega}\fnt{f}_{\Omega}^*$ can be differentiated efficiently using AD, since $\fnt{B}_{\Omega}$ is a sparse diagonal matrix and $\fnt{B}_{\Omega}\fnt{f}_{\Omega}^*$ is a vector whose few nonzero terms are straightforward scalings of flux evaluations. 
}

%\rnote{TODO: ADD} 

%For conciseness, we omit the assembly of multi-dimensional global differentiation matrices in this work, but note that the assembly of multi-dimensional DG matrices is discussed in detail elsewhere \cite{persson2008newton, hillewaert2011sharp, kempf2018automatic}.  

\section{Systems of conservation laws}
\label{sec:systems}

In this section, we extend the Jacobian formulas of Theorem~\ref{thm:explicitJ} from scalar nonlinear conservation laws to an $n\times n$ system of conservation laws.  Let $\bm{f}_{S}(\bm{u}_L,\bm{u}_R): \mathbb{R}^{n}\times \mathbb{R}^n \rightarrow \mathbb{R}^n$ denote an entropy conservative flux function for a 1D system of conservation laws.  We first formulate a system of ODEs by modifying the definition of the arrays and matrices in (\ref{eq:sbpmat}).

Let $\fnt{u}_{\Omega}$ denote a vector of vectors 
\eqlab{
\fnt{u}_{\Omega} = \bmat{\fnt{u}_1\\
\fnt{u}_2\\
\vdots\\
\fnt{u}_n}, \qquad \fnt{u}_i = \bmat{
\fnt{u}_{i,1}\\
\fnt{u}_{i,2}\\
\vdots\\
\fnt{u}_{i,K}
}, \qquad
\fnt{u}_{i,k} = \bmat{
\fnt{u}_{i,k,1}\\
\fnt{u}_{i,k,2}\\
\vdots\\
\fnt{u}_{i,k,N_p}
}
\label{eq:uordering}
}
Here, $\fnt{u}_{\ell,k,j}$ denotes the $j$th degree of freedom for the $\ell$th component of the solution on the $k$th element \bnote{(for $\ell = 1,\ldots, n$, $k = 1,\ldots, K$ and $j = 1,\ldots, N_p$)}.  
Let $(k_1, j_1)$ and $(k_2, j_2)$ be multi-indices which correspond to row and columns indices of a matrix, respectively.  
\bnote{We define the block-diagonal flux matrix $\fnt{F}$ consisting of $n$ diagonal blocks $\fnt{F}_i \in \mathbb{R}^{N_p K\times N_p K}$} as
\eqlab{
\fnt{F} = \bmat{
\fnt{F}_1 &&\\
& \ddots &\\
&& \fnt{F}_n
}, \qquad \LRp{\fnt{F}_\ell}_{(k_1, j_1),(k_2,j_2)} = \LRp{\bm{f}_{S}\LRp{\fnt{u}_{:, k_1,j_1},\fnt{u}_{:, k_2,j_2}}}_{\ell}.
\label{eq:fluxsys}
}
where $\fnt{u}_{:,k,j}$ denotes the vector containing all solution components at the $k$th element and $j$th node, and each entry of the block $\fnt{F}_\ell$ for $\ell = 1,\ldots, n$ corresponds to the $\ell$th component of the vector-valued flux evaluated at solution states $\fnt{u}_{:,k_1,j_1},\fnt{u}_{:,k_2,j_2}$.  

Let $\fnt{M}_{\Omega}, \fnt{Q}_{\Omega}$ denote the global mass and differentiation matrices in (\ref{eq:sbpmat}).  Then, an entropy conservative scheme is given by
\[
\LRp{\fnt{I}_n\otimes \fnt{M}_{\Omega}}\td{\fnt{u}_{\Omega}}{t} + 2\LRp{ \LRp{\fnt{I}_n\otimes \fnt{Q}_{\Omega} } \circ \fnt{F}}\fnt{1} = \fnt{0}.
\]
where $\fnt{I}_n\in \mathbb{R}^n$ is the $n\times n$ identity matrix.

We now provide Jacobian matrix formulas for systems of nonlinear conservation laws.  The proofs are straightforward extensions of the proof of Theorem~\ref{thm:explicitJ} to the vector-valued case, and we omit them for conciseness. The right hand side function $\fnt{r}(\fnt{u})$ for systems can be rewritten as
\[
\fnt{r}(\fnt{u}) = 2\LRp{ \LRp{\fnt{I}_n\otimes\fnt{Q}_{\Omega}} \circ \fnt{F}}\fnt{1} 
= 2\bmat{(\fnt{Q}_{\Omega}\circ \fnt{F}_1) \\ 
\vdots\\
(\fnt{Q}_{\Omega}\circ \fnt{F}_n)
}\fnt{1}.
\]
Then, the Jacobian matrix is 
\eqlab{
\pd{\fnt{r}}{\fnt{u}} &=
\bmat{
\fnt{\partial F}_{1,\fnt{u}_1} & \ldots & \fnt{\partial F}_{1,\fnt{u}_n}\\
\vdots & \ddots & \vdots\\
\fnt{\partial F}_{n,\fnt{u}_1} & \ldots & \fnt{\partial F}_{n,\fnt{u}_n}
}\label{eq:ecjac}
}
where each Jacobian block $\fnt{\partial F}_{i,\fnt{u}_j}$ is evaluated as in Theorem~\ref{thm:explicitJ} 
\eq{
\fnt{\partial F}_{i,\fnt{u}_j} &= 2 \LRp{\fnt{Q}_{\Omega} \circ \fnt{F}_{i,\fnt{u}_j}} \pm \diag{\fnt{1}^T\LRp{2\fnt{Q}_{\Omega}\circ \fnt{F}_{i,\fnt{u}_j}}}\\
\fnt{\partial F}_{i,\fnt{u}_j} &= 2 \LRp{\fnt{Q}_{\Omega} \circ \fnt{F}_{\fnt{u}_i,j}^T} \pm \diag{\LRp{2\fnt{Q}_{\Omega}\circ \fnt{F}_{\fnt{u}_i,j}}\fnt{1}}.
}
for $\fnt{Q}_{\Omega} = \pm \fnt{Q}_{\Omega}^T$.  
Here, the flux matrix $\fnt{F}_{i,\fnt{u}_j}$ is evaluated via one of two formulas
\eq{
\LRp{\fnt{F}_{i,\fnt{u}_j}}_{(j_1,k_1),(j_2,k_2)} = \LRl{\pd{\LRp{\bm{f}_S}_i}{\bm{u}_{R,j}}}_{\fnt{u}_{:,k_1,j_1},\fnt{u}_{:,k_2,j_2}}\\
\LRp{\fnt{F}_{\fnt{u}_i,j}}_{(j_1,k_1),(j_2,k_2)} = \LRl{\pd{\LRp{\bm{f}_S}_i}{\bm{u}_{L,j}}}_{\fnt{u}_{:,k_1,j_1},\fnt{u}_{:,k_2,j_2}},
}
where $\pd{\LRp{\bm{f}_S}_i}{\bm{u}_{L,j}}, \pd{\LRp{\bm{f}_S}_i}{\bm{u}_{R,j}}$ denote the derivatives of the $i$th component of the flux $\bm{f}_S\LRp{\bm{u}_L,\bm{u}_R}$ with respect to the $j$th solution component of $\fnt{u}_L, \fnt{u}_R$.  Thus, each entry of the block $\fnt{F}_{i,\fnt{u}_j}$ corresponds to an entry of the Jacobian (with respect to $\bm{u}_L$ or $\bm{u}_R$) of $\bm{f}_S(\bm{u}_L,\bm{u}_R)$ and an entry of the global differentiation matrix $\fnt{Q}_{\Omega}$.

\begin{remark}
\bnote{The ordering in this paper is chosen for notational convenience. In practice, other orderings typically yield more efficient solution procedures. 
For example, ordering the degrees of freedom by variable fields (as in \cite{crean2018entropy}) yields a Jacobian matrix with a more compact bandwidth, while also making it easier to use a block compressed storage format for sparse matrices and block-based preconditioners.}
\end{remark}

\section{Extension to entropy stable (dissipative) schemes}

We now consider entropy stable schemes, which include entropy dissipation terms to produce a semi-discrete dissipation (rather than conservation) of entropy. These can correspond either to physical or artificial viscosity mechanisms \cite{gassner2017br1,upperman2019entropy} or numerical interface dissipation \cite{winters2017uniquely}. Because Jacobian matrices for artificial viscosity mechanisms have been discussed in more detail in the time-implicit literature \cite{persson2006sub} we focus instead on numerical interface dissipation. 

Let $\bm{d}_S\LRp{\bm{u}_L,\bm{u}_R}$ be an entropy dissipative anti-symmetric flux such that
\[
\bm{d}_S\LRp{\bm{u}_L,\bm{u}_R} = -\bm{d}_S\LRp{\bm{u}_R, \bm{u}_L}, \qquad
\LRp{\bm{v}_L-\bm{v}_R}^T\bm{d}_S\LRp{\bm{u}_L,\bm{u}_R} \geq 0.
\]
Note that the anti-symmetry of $\bm{d}_S$ implies that $\bm{d}_S(\bm{u},\bm{u}) = \bm{0}$. Fluxes which fall into this category include the Lax-Friedrichs flux 
\[
\bm{d}_S(\bm{u}_L,\bm{u}_R) = \frac{\LRb{\lambda}}{2}(\bm{u}_L-\bm{u}_R), \qquad \lambda = \text{estimate of maximum wavespeed},
\]
as well as HLLC fluxes \cite{chen2017entropy} and matrix dissipation fluxes \cite{winters2017uniquely}. 

\subsection{Scalar dissipative fluxes}

We will begin by considering scalar dissipative fluxes $d_S(u_L,u_R)$ and dissipation terms of the form 
\[
\fnt{d}(\fnt{u}) = \LRp{\fnt{B}\circ\fnt{D}}\fnt{1}
\]
\note{where $\fnt{B}$ is a symmetric non-negative matrix} and the entries of $\fnt{D}_{ij} = d_S(\fnt{u}_i,\fnt{u}_j)$ correspond to evaluations of the dissipative flux. For the high order DG-SBP discretizations of periodic domains described in (\ref{eq:sbpform}), $\bm{B}_{\Omega}$ is the matrix
\eqlab{
\fnt{B} = \frac{1}{2}\bmat{
 &\fnt{B}_R & & \fnt{B}_L\\
\fnt{B}_L& & \fnt{B}_R&\\
& \fnt{B}_L & \ddots & \fnt{B}_R\\
\fnt{B}_R& & \fnt{B}_L &}
}
where $\fnt{B}_L,\fnt{B}_R$ are defined as in (\ref{eq:BLBR}).

To compute the Jacobian of this term, we can note that Theorem~\ref{thm:explicitJ} assumes that the discretization matrix is skew-symmetric (or symmetric), while the flux matrix is symmetric. Here, the orders are reversed --- the flux matrix $\bm{D}$ is now skew-symmetric, while the discretization matrix $\bm{B}_{\Omega}$ is symmetric. Thus, repeating the steps of the proof of Theorem~\ref{thm:explicitJ}, one can show that the Jacobians of the dissipative term can be computed using one of two formulas.

\begin{theorem} 
\label{eq:explicitJ_diss}
Let $\fnt{d}(\fnt{u}) = \LRp{\fnt{B}\circ\fnt{D}}\fnt{1}$, where $\fnt{B}$ is a symmetric matrix, $\fnt{D}_{ij} = d_S(\fnt{u}_i,\fnt{u}_j)$, and $d_S$ is an anti-symmetric bivariate function. Then, 
\eqlab{
\pd{\fnt{d}}{\fnt{u}} &= -(\fnt{B}\circ\fnt{D}_x^T) + \diag{\LRp{\fnt{B}\circ \fnt{D}_x^T}\fnt{1}},\label{eq:dissJac}\\
\pd{\fnt{d}}{\fnt{u}} &= (\fnt{B}\circ\fnt{D}_y) - \diag{\fnt{1}^T\LRp{\fnt{B}\circ \fnt{D}_y}}\nonumber
}
where the matrices $\fnt{D}_x,\fnt{D}_y$ are
\[
\LRp{\fnt{D}_x}_{ij} = \LRl{\pd{{d}_S}{{u}_L}}_{\fnt{u}_i,\fnt{u}_j}, \qquad \LRp{\fnt{D}_y}_{ij} = \LRl{\pd{{d}_S}{{u}_R}}_{\fnt{u}_i,\fnt{u}_j}.
\]
\end{theorem}
\begin{proof}
We will prove the second formula in (\ref{eq:dissJac}) involving $\fnt{D}_y$ using the same approach as the proof of Theorem~\ref{thm:explicitJ}. The proof of the first formula results from the fact that $\fnt{D}_y = -\fnt{D}_x^T$ by the anti-symmetry of $d_S(u_L,u_R)$. Applying the chain rule yields
\eq{
\pd{\fnt{d}_i}{\fnt{u}_j} &= \sum_{k} \fnt{B}_{ik} \LRp{\LRl{ \pd{d_S}{x}}_{\fnt{u}_i,\fnt{u}_k}\pd{\fnt{u}_i}{\fnt{u}_j} + \LRl{\pd{d_S}{y}}_{\fnt{u}_i,\fnt{u}_k} \pd{\fnt{u}_k}{\fnt{u}_j}}
}
If $i\neq j$, then $\pd{\fnt{u}_i}{\fnt{u}_j} = 0$ and the sum reduces to the single term $k = j$
\eq{
\pd{\fnt{d}_i}{\fnt{u}_j} &= \fnt{B}_{ij} \LRl{\pd{\bnote{d_S}}{y}}_{\fnt{u}_i,\fnt{u}_j}.
}
For $i=j$, $\pd{\fnt{u}_i}{\fnt{u}_j} = 1$.  Using the symmetry of $\fnt{B}$ and anti-symmetry of $d_S$ yields
\eq{
\pd{\fnt{d}_i}{\fnt{u}_i} &= 
\LRp{\sum_{k} \fnt{B}_{ik} \LRl{\pd{d_S}{x}}_{\fnt{u}_i,\fnt{u}_k}} + \fnt{B}_{ii}\LRl{\pd{d_S}{y}}_{\fnt{u}_i,\fnt{u}_i}
=\LRp{-\sum_{k} \fnt{B}_{ik} \LRl{\pd{d_S}{y}}_{\fnt{u}_k,\fnt{u}_i}} + \fnt{B}_{ii}\LRl{\pd{d_S}{y}}_{\fnt{u}_i,\fnt{u}_i}.
}
\end{proof}

\subsection{Vector-valued dissipative fluxes}

For a vector-valued dissipative flux, the dissipative contribution is 
\[
\fnt{d}(\fnt{u}) %= \LRp{ \LRp{\fnt{I}_n\otimes\fnt{B}} \circ \fnt{D}}\fnt{1} 
= \bmat{(\fnt{B}\circ \fnt{D}_1) \\ 
\vdots\\
(\fnt{B}\circ \fnt{D}_n)
}\fnt{1}
\]
where each matrix block $\LRp{\fnt{D}_\ell}_{(j_1,k_1),(j_2,k_2)} = \LRp{\bm{d}_{S}\LRp{\fnt{u}_{:, k_1,j_1},\fnt{u}_{:, k_2,j_2}}}_{\ell}$ corresponds to the $i$th component of the dissipative flux, where $\fnt{u}$ is ordered as in (\ref{eq:uordering}). Then, the Jacobian of $\fnt{d}(\fnt{u})$ yields the following block matrix
\eqlab{
\pd{\fnt{d}}{\fnt{u}} &=
\bmat{
\fnt{\partial D}_{1,\fnt{u}_1} & \ldots & \fnt{\partial D}_{1,\fnt{u}_n}\\
\vdots & \ddots & \vdots\\
\fnt{\partial D}_{n,\fnt{u}_1} & \ldots & \fnt{\partial D}_{n,\fnt{u}_n}
}
}
where each Jacobian block $\fnt{\partial D}_{i,\fnt{u}_j}$ is evaluated as in (\ref{eq:dissJac}) using one of two formulas
\eq{
\fnt{\partial D}_{i,\fnt{u}_j} &=  \LRp{\fnt{B} \circ \fnt{D}_{i,\fnt{u}_j}} - \diag{\fnt{1}^T\LRp{\fnt{B}\circ \fnt{D}_{i,\fnt{u}_j}}} \\
\fnt{\partial D}_{i,\fnt{u}_j} &= -\LRp{\fnt{B} \circ \fnt{D}_{\fnt{u}_i,j}} + \diag{\LRp{\fnt{B}\circ \fnt{D}_{\fnt{u}_i,j}}\fnt{1}}, 
}
where the dissipative flux matrices $\fnt{D}_{i,\fnt{u}_j}, \fnt{D}_{\fnt{u}_i,j}$ are defined in terms of entries of the Jacobian of $\bm{d}_S$
\eq{
\LRp{\fnt{D}_{i,\fnt{u}_j}}_{(j_1,k_1),(j_2,k_2)} &= \LRl{\pd{\LRp{\bm{d}_S}_i}{\bm{u}_{R,j}}}_{\fnt{u}_{:,k_1,j_1},\fnt{u}_{:,k_2,j_2}}\\
\LRp{\fnt{D}_{\fnt{u}_i,j}}_{(j_1,k_1),(j_2,k_2)} &= \LRl{\pd{\LRp{\bm{d}_S}_i}{\bm{u}_{L,j}}}_{\fnt{u}_{:,k_1,j_1},\fnt{u}_{:,k_2,j_2}},
}

\begin{remark}
If the derivative of $\bm{d}_S$ with respect to its second argument $\bm{u}_R$ is used to compute the dissipative flux matrices, then the structure of the dissipative Jacobian is identical to the structure of the entropy conservative Jacobian (\ref{eq:ecjac}). Thus, given discretization matrices $\bm{Q}_{\Omega}, \bm{B}_{\Omega}$ and functions which evaluate derivatives of flux functions $\bm{f}_S, \bm{d}_S$ with respect to their second arguments, the same routine can be used to compute both the entropy conservative and dissipative Jacobians.
\end{remark}
%\section{Generalizations to different numerical settings}

%In this section, we discuss how to generalize the explicit construction of Jacobians to different numerical settings, including non-collocation schemes and non-periodic boundary conditions.  

\section{Non-collocated schemes: hybridized SBP operators, entropy projection, over-integration}

Most entropy stable schemes rely on ``collocated'' SBP operators (where the mass matrix $\bm{M}_{\Omega}$ is diagonal) constructed using nodal sets which include boundary nodes \cite{chen2017entropy, crean2018entropy}.  However, in certain cases energy and entropy stable SBP schemes constructed using non-diagonal mass matrices \cite{chan2017discretely, chan2019entropy} and more general nodal sets \cite{fernandez2014review, ranocha2018generalised, crean2017high, chan2018efficient} achieve higher accuracy than SBP schemes built on nodal sets which include boundary nodes.  We discuss how to extend Jacobian formulas to ``modal'' formulations for entropy conservative schemes (the extension to entropy stable schemes is similar).

\subsection{``Modal'' entropy conservative schemes} % and hybridized SBP operators}

We now assume that the solution is represented using a ``modal'' expansion
\[
u(\bm{x},t) \approx \sum_{\bnote{i}=1}^{N_p} \hat{\fnt{u}}_{k,i}(t) \phi_i(\bm{x}),
\]
where $\hat{\fnt{u}}_{k,i}$ denotes the coefficients of the solution on an element $D^k$.  We assume two sets of quadrature points: volume quadrature points and weights, $\LRc{w_i, \fnt{x}_{q,i}}_{i=1}^{N_q}$, and surface quadrature points, $\LRc{w_{f,i}, \fnt{x}_{f,i}}_{i=1}^{N_f}$.  We assume both quadrature rules are exact for certain classes of integrands as detailed in \cite{chan2019skew, chan2019entropy}.  

Evaluating $\fnt{u}(\fnt{x},t)$ at quadrature points requires multiplication by an interpolation matrix $\fnt{V}$
\eq{
{\fnt{V}}_{ij} = \phi_j(\fnt{x}_i), \qquad i = 1,\ldots, N_q, \qquad j = 1,\ldots, N_p\\
\LRp{\fnt{V}_f}_{ij} = \phi_j(\fnt{x}_{f,i}), \qquad i = 1,\ldots, N_f, \qquad j = 1,\ldots, N_p.
}
We can similarly define mass and projection matrices $\fnt{M}, \fnt{P}$ 
\[
\fnt{M} = \fnt{V}^T\fnt{W}\fnt{V}, \qquad \fnt{P} = \fnt{M}^{-1}\fnt{V}^T\fnt{W},
\]
where $\fnt{W}$ is a diagonal matrix whose entries are the quadrature weights $w_i$.  We also define a face interpolation matrix
\[
\fnt{E} = \fnt{V}_f\fnt{P}
\]
which evaluates the solution at face quadrature points given values at volume quadrature points.  Finally, we define the matrix $\fnt{V}_h$ as the mapping between local coefficients $\hat{\fnt{u}}_k$ and the combined vector of volume and surface quadrature points
\[
\fnt{V}_h = \bmat{\fnt{V} \\ \fnt{V}_f}.
\]
These matrices are involved in the application of hybridized SBP operators (originally referred to as decoupled SBP operators) \cite{chan2017discretely, chenreview}.  We present the main ideas in a 1D setting and refer the reader to \cite{crean2018entropy, chan2017discretely, chan2019skew} for details on multi-dimensional settings.  

Given some modal weak differentiation matrix $\hat{\fnt{Q}}$ which acts on the basis coefficients $\hat{\fnt{u}}_k$, we define a nodal differentiation matrix $\fnt{Q} = \fnt{P}^T\hat{\fnt{Q}}\fnt{P}$.  Then we can define a hybridized SBP operator as 
\[
\fnt{Q}_h = \frac{1}{2}\bmat{\fnt{Q}-\fnt{Q}^T & \fnt{E}^T\fnt{B} \\
-\fnt{B}\fnt{E} & \fnt{B}}, \qquad \fnt{B} = \bmat{-1 & \\ & 1}.  
\]
The operator $\fnt{Q}_h$ can be used to approximate coefficients of the derivative in the basis $\phi_i(\fnt{x})$.  Let $f(u)$ denote some function of $u(\fnt{x})$, and let $\hat{\fnt{u}}$ denote the basis coefficients of $u(\fnt{x})$.  Then, 
\[
\pd{f(u)}{x}\approx \sum_j {\hat{\fnt{f}}}_j \phi_j(\fnt{x}), \qquad \hat{\fnt{f}} = \fnt{M}^{-1}\fnt{V}_h^T\fnt{Q}_h f\LRp{\fnt{V}_h\hat{\fnt{u}}}
\]

We now construct global matrices for the multi-element (periodic) case.  We begin by concatenating the local coefficients $\hat{\fnt{u}}_{k,i}$ into a global coefficient vector $\hat{\fnt{u}}_{\Omega}$.  We also introduce boundary matrices $\fnt{B}, \fnt{B}_L$, and $\fnt{B}_R$ which enforce coupling between different elements and are defined as 
\[
\fnt{B}_L = \bmat{ & 1\\0 &}, \qquad \fnt{B}_R = \bmat{ &0 \\1  &}, \qquad \fnt{B} = \bmat{-1 & \\ & 1}.
\]
In the multi-dimensional case, the entries of $\fnt{B}_L, \fnt{B}_R$ correspond instead to outward normals \bnote{scaled by surface quadrature weights and surface Jacobians (e.g., the area ratios between reference and physical surface elements)} \cite{crean2018entropy, chan2017discretely}.  

We can also adapt $\fnt{Q}_h$ to construct a globally skew-symmetric differentiation matrix (see also \cite{chan2018efficient}).  Define the matrix $\fnt{S} = \fnt{Q}-\fnt{Q}^T$ and define $\fnt{Q}_{\Omega}$ as the global block matrix
\[
\fnt{Q}_{\Omega} = \frac{1}{2}\LRs{\begin{array}{cc|cc|cc|cc}
\fnt{S} &  \fnt{E}^T\fnt{B} & &&& && \\
 - \fnt{B}\fnt{E} &  & & \fnt{B}_R && && -\fnt{B}_L\\ [.2em]
 \hline\\[-1.0em]
&& \fnt{S} &  \fnt{E}^T\fnt{B}&&&& \\
& -\fnt{B}_L &  - \fnt{B}\fnt{E} &&& \fnt{B}_R&&\\
 \hline\\[-1.0em]
&&&& \ddots & \ddots&  \\
&&&   -\fnt{B}_L & \ddots&\ddots&   \ddots&\fnt{B}_R \\
 \hline\\[-1.0em]
&&&&&\ddots& \fnt{S} &  \fnt{E}^T\fnt{B} \\
\fnt{B}_R &&& &&-\fnt{B}_L &  - \fnt{B}\fnt{E} &\\
\end{array}},
\]
We abuse notation and redefine $\fnt{V},\fnt{E},\fnt{P}$ and $\fnt{V}_h$ as \emph{global} interpolation, projection, and extrapolation matrices
\eq{
\fnt{V} &\longrightarrow \fnt{I}_K \otimes \fnt{V}, \qquad \fnt{P} \longrightarrow \fnt{I}_K \otimes \fnt{P} \\ 
\fnt{E} &\longrightarrow \fnt{I}_K \otimes \fnt{E}, \qquad \fnt{V}_h \longrightarrow \fnt{I}_K \otimes \fnt{V}_h
}
Finally, we assume that the global solution $\fnt{u}(\fnt{x},t) \in \mathbb{R}^n$ is vector-valued, and order the solution coefficients as in Section~\ref{sec:systems}.  

It was shown in \cite{parsani2016entropy, chan2017discretely} that when either the mass matrix is non-diagonal or the nodal set does not contain appropriate boundary points, it is necessary to perform an entropy projection (or extrapolation \cite{chenreview}) step to ensure discrete entropy stability.  Let $\bm{v}(\bm{u})$ denote the entropy variables as a function of the conservative variables, and let $\bm{u}(\bm{v})$ denote the inverse mapping.  We define the entropy projected variables $\tilde{\fnt{u}}_{\Omega}$ as
\eqlab{
\tilde{\fnt{u}}_{\Omega} = \bm{u}\LRp{\fnt{V}_h\fnt{P}\bm{v}\LRp{\fnt{V}\hat{\fnt{u}}_{\Omega}}}.
\label{eq:entropyprojection}
}
Let $\fnt{F}$ again denote the block-diagonal flux matrix in (\ref{eq:fluxsys}).  We evaluate each flux block $\fnt{F}_{\ell}$  using the entropy projected variables
\eqlab{
\LRp{\fnt{F}_\ell}_{(k_1, j_1),(k_2,j_2)} = \LRp{\bm{f}_{S}\LRp{\tilde{\fnt{u}}_{:, k_1,j_1},\tilde{\fnt{u}}_{:, k_2,j_2}}}_{\ell}.
\label{eq:fluxell}
}
Then, an entropy conservative method is given by 
\[
\LRp{\fnt{I}_n\otimes \fnt{M}_{\Omega}}\td{\fnt{u}_{\Omega}}{t} + 2\LRp{\fnt{I}_n\otimes \fnt{V}_h}^T\LRp{ \LRp{\fnt{1}_n\fnt{1}_n^T\otimes \fnt{Q}_{\Omega} } \circ \fnt{F}}\fnt{1} = \fnt{0}.
%\fnt{M}_{\Omega}\td{\fnt{u}_{\Omega}}{t} + 2\fnt{V}^T\LRp{ \fnt{Q}_{\Omega} \circ \fnt{F}}\fnt{1} = \fnt{0}.  
\]
where $\fnt{I}_n$ is the $n\times n$ identity matrix and $\fnt{1}_n$ denotes the length $n$ vector of all ones.

\subsection{Jacobian matrices for modal entropy stable schemes}

We redefine the nonlinear term as
\[
\fnt{r}(\hat{\fnt{u}}) = 2\LRp{\fnt{I}_n\otimes \fnt{V}_h}^T\LRp{ \LRp{\fnt{1}_n\fnt{1}_n^T\otimes \fnt{Q}_{\Omega}} \circ \fnt{F}}\fnt{1}.
\]
where the flux matrix $\fnt{F}$ is computed using the entropy projected conservative variables (\ref{eq:entropyprojection}) via (\ref{eq:fluxell}).
Let $\pd{\bm{u}}{\bm{v}}$ and $\pd{\bm{v}}{\bm{u}}$ denote Jacobians of the conservative variables with respect to the entropy variables and vice versa.  These have been explicitly derived for several equations (for example, the Jacobians for the compressible Navier-Stokes equations are given in \cite{hughes1986new}).  

We can compute the Jacobian of $\fnt{r}(\hat{\fnt{u}})$ via the chain rule.  We assume a scalar equation $n=1$ for simplicity, and motivate our approach by considering an entry $i\neq j$ of the Jacobian
\[
\LRp{\pd{\fnt{r}}{\hat{\fnt{u}}_{\Omega}}}_{ij} =  2 \fnt{V}_h^T \pd{}{\LRp{\hat{\fnt{u}}_{\Omega}}_j} \LRp{\LRp{\fnt{Q}_{\Omega}\circ\fnt{F}}\fnt{1}}_i.
\]
We focus on the latter term $\pd{}{\hat{\fnt{u}}_{\Omega}} \LRp{\fnt{Q}_{\Omega}\circ\fnt{F}}\fnt{1}$
\eq{
\LRp{\pd{}{\hat{\fnt{u}}_{\Omega}} \LRp{\fnt{Q}_{\Omega}\circ\fnt{F}}\fnt{1}}_{ij} &= \pd{}{\hat{\fnt{u}}_{\Omega,j}} \sum_{k} \LRp{\fnt{Q}_{\Omega}}_{ik} \bm{f}_S\LRp{\tilde{\fnt{u}}_i,\tilde{\fnt{u}}_k} =  \sum_{k} \LRp{\fnt{Q}_{\Omega}}_{ik} \LRl{\pd{\bm{f}_S}{y}}_{\tilde{\fnt{u}}_i,\tilde{\fnt{u}}_k} \pd{\tilde{\fnt{u}}_i}{\hat{\fnt{u}}_{\Omega,j}}
}
We observe that the term $\pd{\tilde{\fnt{u}}_i}{\hat{\fnt{u}}_{\Omega,j}}$ does not disappear as it did in the proof of Theorem~\ref{thm:explicitJ}.  We thus treat the Jacobian matrix in two parts.  First, we define the ``unassembled'' Jacobian matrix $\pd{\tilde{\fnt{r}}}{\tilde{\fnt{u}}}$ as
\eqlab{
\LRp{\pd{\tilde{\fnt{r}}}{\tilde{\fnt{u}}}}_{ij} = \LRp{\fnt{Q}_{\Omega}}_{ij} \LRl{\pd{\bm{f}_S}{y}}_{\tilde{\fnt{u}}_i,\tilde{\fnt{u}}_j} = \LRp{\fnt{Q}_{\Omega}\circ \fnt{F}_y}_{ij}.
\label{eq:dfdu_unassembled}
}
The construction of $\pd{\tilde{\fnt{r}}}{\tilde{\fnt{u}}}$ for systems ($n > 1$) is carried out using the procedure described in Section~\ref{sec:systems}.
Let $\tilde{\fnt{v}}$ denote the projected entropy variables evaluated at volume quadrature points
\[
\tilde{\fnt{v}} = \fnt{V}_h\fnt{P}\bm{v}\LRp{\fnt{V}\hat{\fnt{u}}_{\Omega}}.
\]
The vector $\pd{\tilde{\fnt{u}}}{\hat{\fnt{u}}_{\Omega}}$  can be further expanded as 
\[
\pd{\tilde{\fnt{u}}}{\hat{\fnt{u}}_{\Omega}} = \LRl{\pd{\bm{u}}{\bm{v}}}_{\tilde{\fnt{v}}} \fnt{V}_h \fnt{P} \LRl{\pd{\bm{v}}{\bm{u}}}_{\fnt{V}\hat{\fnt{u}}_{\Omega}} \bm{V}.
\]
where the Jacobian matrices for the maps between conservative and entropy variables are block diagonal matrices given by
\eqlab{
\LRl{\pd{\bm{u}}{\bm{v}}}_{\tilde{\bnote{\fnt{v}}}} = \bmat{
\LRl{\pd{\bm{u}}{\bm{v}}}_{\tilde{\bnote{\fnt{v}}}_1} &&\\
& \ddots &\\
&& \LRl{\pd{\bm{u}}{\bm{v}}}_{\tilde{\bnote{\fnt{v}}}_K}
}, \qquad \LRl{\pd{\bm{u}}{\bm{v}}}_{\tilde{\fnt{u}}_k} = \bmat{
\LRl{\pd{\bm{u}}{\bm{v}}}_{\tilde{\fnt{u}}_{1,k}} &&\\
 &\ddots &\\
&& \LRl{\pd{\bm{u}}{\bm{v}}}_{\tilde{\fnt{u}}_{N_p,k}}
},
\label{eq:changeofvars}
}
where the local block $\LRl{\pd{\bm{u}}{\bm{v}}}_{\tilde{\fnt{u}}_{j,k}}$ is the Jacobian matrix $\pd{\bm{u}}{\bm{v}}$ evaluated at the $j$th nodal solution value $\tilde{\fnt{u}}_{j,k}$ on the $k$th element.

Let $N_p, N_q$, and $N_f$ denote the number of total basis functions, quadrature points, and face quadrature points respectively, and define $N_{\rm total} = N_q + N_f$.  The structure and dimensions of matrices involved in constructing the ``assembled'' Jacobian matrix are illustrated as follows:
\eqlab{
\pd{\fnt{r}}{\hat{\fnt{u}}_{\Omega}} = \renewcommand\matscale{.55}
\raiserows{2}{\matbox{3}{7}{N_p}{N_{\rm total}}{\fnt{V}_h^T}}
\raiserows{0}{\matbox{7}{7}{N_{\rm total}}{N_{\rm total}}{\fnt{Q}_{\Omega}\circ\pd{\tilde{\fnt{r}}}{\tilde{\fnt{u}}}}}
\raiserows{0}{\matbox{7}{7}{N_{\rm total}}{N_{\rm total}}{\LRl{\pd{\bm{u}}{\bm{v}}}_{\tilde{\fnt{v}}}}}
\raiserows{0}{\matbox{7}{5}{N_{\rm total}}{N_q}{ \fnt{V}_h\fnt{P} }}
\raiserows{1}{\matbox{5}{5}{N_q}{N_q}{ \LRl{\pd{\bm{v}}{\bm{u}}}_{\fnt{V}\hat{\fnt{u}}_{\Omega}} }}
\raiserows{1}{\matbox{5}{3}{N_q}{N_p}{\fnt{V}}}
\label{eq:dfdu_assembled}
}

\begin{figure}
\centering
\subfloat[``Unassembled'' Jacobian matrix (\ref{eq:dfdu_unassembled})]{\includegraphics[height=.25\textheight]{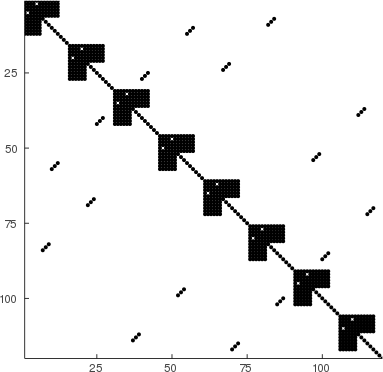}}
\hspace{3em}
\subfloat[``Assembled'' Jacobian matrix (\ref{eq:dfdu_assembled})]{\includegraphics[height=.25\textheight]{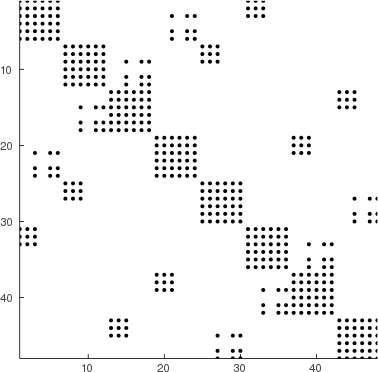}}
\caption{Spy plots of assembled and unassembled Jacobian matrices for Burgers' equation for $N=2$ on a $2\times 2$ uniform triangular mesh of $[-1,1]^2$.}
\label{fig:spyplots}
\end{figure}

``Unassembled'' and ``assembled'' Jacobian matrices (\ref{eq:dfdu_unassembled}) and (\ref{eq:dfdu_assembled}) for an entropy conservative discretization of a 2D Burgers' equation \cite{chan2017discretely} are shown in Figure~\ref{fig:spyplots}.  We note that the structure of these matrices becomes simplified under common assumptions for entropy stable discretizations.  The most common assumptions are either ``collocated volume nodes'' or ``collocated volume and surface nodes'' \cite{shadpey2019energy}.  When volume nodes are collocated, the solution is represented using a nodal Lagrange basis constructed using $N_q = N_p$ volume quadrature nodes.\footnote{This definition refers to discretizations which utilize an explicit basis. For entropy stable SBP discretizations, volume nodes are typically collocated by construction. This is possible because nodal degrees of freedom for SBP discretizations do not necessarily correspond to a nodal basis.}  When surface nodes are collocated, the surface quadrature points are also a subset of the volume quadrature nodes \cite{gassner2013skew, chen2017entropy}.  

If only volume nodes are collocated \cite{chan2018efficient}, then $\fnt{V} = \fnt{I}$.  If both volume and surface nodes are collocated, the system reduces to the simplified system described in Section~\ref{sec:systems} using the fact that $\pd{\bm{u}}{\bm{v}} = \LRp{\pd{\bm{v}}{\bm{u}}}^{-1}$.

%\subsection{Non-periodic boundary conditions}
%\label{eq:npbc}
%
%Finally, we consider the case of non-periodic boundary conditions.  \note{Impose BCs by specifying exterior value.}

\section{Numerical experiments}

In this section, we verify our theoretical results and compare the computational efficiency of the formulas derived in this paper with other methods for computing the Jacobian. Additional numerical experiments are included in Section~\ref{sec:tdrk} and \ref{sec:implicit}.

\subsection{\bnote{Verification of Jacobian formulas}}

\bnote{We begin by verifying the correctness of Theorem~\ref{thm:explicitJ} and its extension to systems of nonlinear conservation laws. We do so by comparing these formulas to Jacobians computed directly using automatic differentiation for $\fnt{r}(\fnt{u}) = (\fnt{Q}\circ\fnt{F})\fnt{1}$, where $\fnt{Q}$ is a randomly generated symmetric or skew-symmetric matrix and $\fnt{F}$ is the flux matrix defined in (\ref{eq:fu}) for scalar fluxes and (\ref{eq:fluxsys}) for systems. 

We compute Jacobians for three different fluxes. The first is the entropy conservative flux for Burgers' equation 
\[
f_S(u_L,u_R) = \frac{1}{6}\LRp{u_L^2 + u_Lu_R + u_R^2}.
\]

The second set of fluxes are entropy conservative fluxes for the two-dimensional shallow water equations with solution fields $h, hu, hv$ corresponding to water height and $x,y$ momentum \cite{fjordholm2011well, wintermeyer2017entropy}. Let the average be defined as $\avg{u} = \frac{u_L + u_R}{2}$. The two-dimensional shallow water fluxes $\bm{f}_{i,S}$ for each coordinate direction $i = 1,\ldots, d$ are given by 
\[
\bm{f}_{1,S} = \bmat{
\avg{hu}\\
\avg{hu}\avg{u} + \frac{g}{2} h_L h_R\\
\avg{hu}\avg{v}
}, \qquad 
\bm{f}_{2,S} = \bmat{
\avg{hv}\\
\avg{hv}\avg{u}\\
\avg{hv}\avg{v} + \frac{g}{2} h_L h_R
}.
\]

The third set of fluxes are kinetic energy preserving and entropy conservative fluxes for the 3D compressible Euler equations \cite{chandrashekar2013kinetic}. The solution fields are $\rho, \rho u, \rho v, \rho w, E$, corresponding to density, $x/y/z$ momentum, and total energy. Let $\avg{\cdot}^{\log}$ denote the logarithmic mean 
\[
\avg{u}^{\log} = \frac{u_R - u_L}{\log\LRp{u_R}-\log\LRp{u_L}},
\]
which we compute using the numerically stable expansion of \cite{winters2019entropy} with $\gamma = 1.4$. The fluxes $\bm{f}_{i,S}$ for each coordinate direction $i = 1,\ldots, d$ are then given by
\begin{gather*}
\bm{f}_{1,S} = \LRp{\begin{array}{c}
\avg{\rho}^{\log}\avg{u}\\
\avg{\rho}^{\log}\avg{u}^2 + p_{\rm avg}\\
\avg{\rho}^{\log}\avg{u}\avg{v}\\
\avg{\rho}^{\log}\avg{u}\avg{w}\\
(E_{\rm avg}+ p_{\rm avg})\avg{u}\\
\end{array}}, 
\qquad 
\bm{f}_{2,S} = \LRp{\begin{array}{c}
\avg{\rho}^{\log}\avg{v}\\
\avg{\rho}^{\log}\avg{u}\avg{v}\\
\avg{\rho}^{\log}\avg{v}^2 + p_{\rm avg}\\
\avg{\rho}^{\log}\avg{v}\avg{w}\\
(E_{\rm avg}+ p_{\rm avg})\avg{v}\\
\end{array}},\\
\bm{f}_{3,S} = \LRp{\begin{array}{c}
\avg{\rho}^{\log}\avg{w}\\
\avg{\rho}^{\log}\avg{u}\avg{w}\\
\avg{\rho}^{\log}\avg{v}\avg{w}\\
\avg{\rho}^{\log}\avg{w}^2 + p_{\rm avg}\\
(E_{\rm avg}+ p_{\rm avg})\avg{u_3}\\
\end{array}},\nonumber
\end{gather*}
where the auxiliary quantities are defined as
\begin{gather*}
\beta = \frac{\rho}{2p}, \qquad p_{\rm avg} = \frac{\avg{\rho}}{2\avg{\beta}}, \qquad E_{\rm avg} = \frac{\avg{\rho}^{\log}}{2(\gamma-1)\avg{\beta}^{\log}} + \frac{1}{2}\avg{\rho}^{\log}\note{{u}^2_{\rm avg}}, \\
{u^2_{\rm avg}} = u_{L}u_{R} + v_{L}v_{R} + w_{L}w_{R}.\nonumber
\end{gather*}

We also verify Theorem~\ref{eq:explicitJ_diss} for each system by computing the Jacobian matrix for the dissipative Lax-Friedrichs flux 
\[
\bm{d}_S(\bm{u}_L,\bm{u}_R) = \frac{\lambda_{\max}}{2}\LRp{\bm{u}_L - \bm{u}_R}
\]
where $\lambda_{\max}$ is an estimate of the maximum 1D wavespeed between $\bm{u}_L, \bm{u}_R$ along some unit vector $\bm{n}$ (e.g., the outward normal). In all cases, we take $\lambda_{\max} = \max\LRp{\lambda(\bm{u}_L),\lambda(\bm{u}_R)}$, where $\lambda(\bm{u})$ is an upper bound on the wavespeed. For both shallow water and Euler, $\lambda(\bm{u}) = |u_n| + c$, where $u_n$ is the normal component of velocity and $c$ is the speed of sound. For shallow water, $c = \sqrt{gh}$, while for Euler, $c= \sqrt{\gamma p/\rho}$. 
%For all tests, we take $\bm{n}$ to be a randomly generated unit vector. 

Table~\ref{tab:diffs} shows differences between Jacobian matrices computed using automatic differentiation and using formulas in Theorems~\ref{thm:explicitJ} and \ref{eq:explicitJ_diss}. Discretization matrices of size $25\times 25$ and corresponding solution vectors were generated randomly from a normal distribution. For the shallow water and Euler equations, positive solution values (e.g., water height for shallow water, density and pressure for Euler) were generated from a uniform distribution over $(0,1)$. For Jacobians of dissipative fluxes, the normal vector is taken to be a random unit vector. In all cases, the difference close to machine precision.
%For Jacobians of dissipative fluxes, the normal vector is taken to be one of the canonical basis vectors in $d$ dimensions, and the reported difference is the sum of the differences over all canonical basis vectors. For $\fnt{r}(\fnt{u}) = \LRp{\fnt{Q}\circ\fnt{F}}\fnt{1}$ with $\fnt{Q}$ skew-symmetric and $\bm{f}_S$ entropy conservative, we surprisingly observe zero difference up to sixteen digits between the two methods of computing Jacobian matrices. For dissipative Jacobian matrices computed from Theorem~\ref{eq:explicitJ_diss}, we observe differences close to machine precision. 
\begin{table}[!h]
\centering
\begin{tabular}{c|c|c|c|c|c}
Burgers' & Shallow water & Euler & LF (Burgers) & LF (SWE) & LF (Euler) \\
\hline
1.56616230e-15 & 9.17858305e-13 & 2.62285783e-14 & 2.03313333e-14  & 3.05403043e-12 & 5.04444613e-14 \\
\end{tabular}
\caption{Computed differences between Jacobian matrices of $\fnt{r}(\fnt{u})$ (measured in the Frobenius norm) when computed using AD and formulas from Theorems~\ref{thm:explicitJ} and \ref{eq:explicitJ_diss}. LF refers to the ``Lax-Friedrichs'' flux.}
\label{tab:diffs}
\end{table}

%A Julia code which reproduces these results is included in the Supplementary Materials. %This code also compares differences between Jacobian matrices computed through automatic differentiation and through formulas in Theorems~\ref{thm:explicitJ} and \ref{eq:explicitJ_diss} for several other cases (1D shallow water, 1D and 2D Euler, symmetric $\fnt{Q}$). In all cases, the results are identical up to machine precision. 
}

%\bnote{TODO: FINISH}

\subsection{Comparisons of computational cost}

We first compare the cost of computing the Jacobian matrix using the formulas in this paper to other approaches. All computations are performed on a 2019 Macbook Pro with a 2.3 GHz 8-Core Intel Core i9 processor using Julia version 1.4 and all timings are computed using the \verb+BenchmarkTools.jl+ package \cite{BenchmarkTools.jl-2016}. 

\bnote{The cost of forward-mode automatic differentiation is known to be minimal for functions with low-dimensional inputs and outputs \cite{griewank2008evaluating}. To give a sense for the efficiency of AD in Julia, we compare the cost of evaluating a flux function $f_S(u_L,u_R)$ to the cost of computing its derivative using} \verb+ForwardDiff.jl+ \bnote{for $10000$ random values of $u_L,u_R$. The evaluation of the entropy conservative Burgers' flux takes 7.087 microseconds, while the derivative takes 7.063 microseconds to evaluate. The logarithmic mean takes 129.254 microseconds to evaluate, while its derivative takes 161.322 microseconds to evaluate.} The cost of computing Jacobians using \verb+ForwardDiff.jl+ scales similarly. 

Next, we compare the cost of computing both the full Jacobian and a Jacobian-vector product using the formulas in Theorem~\ref{thm:explicitJ} and competing approaches. Let $f_S(u_L,u_R)$ denote the scalar flux Burgers' flux, and define 
\eqlab{
\fnt{r}(\fnt{u}) = \LRp{\fnt{Q}\circ\fnt{F}}\fnt{1}, \qquad \fnt{F}_{ij} = f_S(\fnt{u}_i,\fnt{u}_j),
\label{eq:f_comp}
}
where $\fnt{Q} \in \mathbb{R}^{N,N}$ is a dense randomly generated skew-symmetric matrix. We compute the Jacobian matrix using the formula from Theorem~\ref{thm:explicitJ} (referred to as ``Formula from Theorem~\ref{thm:explicitJ}'' in Table~\ref{tab:timings}), with $\pd{f_S}{u_R}$ computed using both the analytical formula and automatic differentiation, which are tagged as ``(analytic)'' and ``(AD)'' in Table~\ref{tab:timings}. We also compute the full Jacobian matrix directly using \verb+ForwardDiff.jl+ (referred to as ``Automatic differentiation'' in Table~\ref{tab:timings}). We also compute the Jacobian matrix using the \verb+FiniteDiff.jl+ toolkit within the \verb+DifferentialEquations.jl+ framework \cite{rackauckas2017differentialequations}, which computes the Jacobian matrix efficiently using cached in-place function evaluations and finite difference approximations (referred to as ``finite differences'' in Table~\ref{tab:timings}). Finally, we provide timings for evaluating $\fnt{r}(\fnt{u})$ for reference. Implementations of both $\fnt{r}(\fnt{u})$ and its Jacobian are optimized for performance in Julia.\footnote{In our implementations of the evaluation of $\fnt{r}(\fnt{u})$ and the Jacobian $\pd{\fnt{r}}{\fnt{u}}$ (computed using Theorem~\ref{thm:explicitJ}), we pre-allocate all output vectors and matrices for efficiency. For the implementation of $\pd{\fnt{r}}{\fnt{u}} $, we compute the sum $\LRp{\fnt{Q}\circ\fnt{F}}\fnt{1}$ by looping over rows of $\bm{Q}$ and accumulating contributions from $\fnt{Q}\circ\fnt{F}$ column-by-column. We access entries of $\fnt{Q}^T$ to take advantage of the column-major storage of matrices in Julia.} %We have included code to compute timings in the Supplementary Materials. 

%Finally, we compare the computation of the Jacobian-vector product using the formulas in Theorem~\ref{thm:explicitJ} and using a finite difference approximation \cite{knoll2004jacobian}
%\[
%\LRl{\pd{\fnt{r}}{\fnt{u}}}_{\fnt{u}_h}\fnt{v} \approx \frac{\fnt{f}(\fnt{u}_h+\epsilon \fnt{v}) - \fnt{f}(\fnt{u}_h)}{\epsilon}.
%\]

\begin{table}[!h]
\centering
%\subfloat[Full Jacobian]{
\begin{tabular}{cccc}
\toprule
& N = 10 & N = 25 & N = 50 \\
\midrule
Automatic differentiation & 3.160 & 26.386 & 166.689 \\
Finite differences  & 1.536 & 17.397 & 129.510 \\
Formula from Theorem~\ref{thm:explicitJ} (analytic) & .125 & .628 &  2.357 \\
Formula from Theorem~\ref{thm:explicitJ} (AD) & .128 & .628 & 2.530 \\
Evaluation of $\fnt{r}(\fnt{u})$ (for reference) & .129 & .623 & 2.517 \\
\bottomrule
\end{tabular}
%}\\
%\subfloat[Jacobian-vector product]{
%\begin{tabular}{cccc}
%\toprule
%& N = 10 & N = 25 & N = 50 \\
%\midrule
%Finite differences (full Jacobian) \\ 
%Explicit formula \\
%\bottomrule
%\end{tabular}}
\caption{Timings for the computation of $\fnt{r}(\fnt{u})$ in (\ref{eq:f_comp}) and various methods of computing the full Jacobian $\td{\fnt{r}}{\fnt{u}}$  using the scalar Burgers' flux $f_S(u_L,u_R) = (u_L^2 + u_Lu_R + u_R^2)/6$ (times in microseconds).}
\label{tab:timings}
\end{table}
We observe that the cost of evaluating the full Jacobian matrix using the formula of Theorem~\ref{thm:explicitJ} is 1-2 orders of magnitude less expensive than automatic differentiation or finite differences applied directly to the nonlinear term $\fnt{r}(\fnt{u})$. These results highlight the fact that \bnote{Theorem~\ref{thm:explicitJ} allows one to take advantage of the Hadamard product and symmetry/skew-symmetry, which is difficult to do when directly applying automatic differentiation}. 

Because the number of flux evaluations required to evaluate $\fnt{r}(\fnt{u})$ is comparable to the number of AD function evaluations required to evaluate the Jacobian matrix, the cost of evaluating the full Jacobian matrix is proportional to the cost of directly evaluating $\fnt{r}(\fnt{u})$. Here, the constant of proportionality is roughly equal to the ratio of the cost of evaluating the flux function derivative (or Jacobian) and the cost of directly evaluating the flux function. This ratio of this cost will vary depending on the specific flux and the implementation.  For the entropy conservative fluxes for the two-dimensional compressible Euler equations \cite{chandrashekar2013kinetic}, the cost of evaluating the flux Jacobian matrix is only $1.625$ times more expensive than directly evaluating the flux (both the Jacobian matrix and the flux were evaluated only for a single coordinate direction). 

Finally, we note that if the Jacobian matrix is not explicitly required, \bnote{Jacobian-vector products can be evaluated in a matrix-free fashion using either forward mode AD \cite{griewank2008evaluating} or finite difference approximations \cite{knoll2004jacobian} at much lower computational cost. The formulas in Theorem~\ref{thm:explicitJ} can still be applied in a matrix-free fashion, but it is unclear if there are computational advantages over AD for computing Jacobian-vector products.}

\section{Conclusion}

In this work, we derive efficient formulas for Jacobian matrices resulting from entropy conservative and entropy stable schemes based on flux differencing and summation-by-parts operators. These formulas are given in terms of summation-by-parts matrices and derivatives of flux functions, the latter of which can be computed efficiently using automatic differentiation. The computation of Jacobians using these formulas is significantly faster than directly computing Jacobian matrices using automatic differentiation, especially for dense operators.  
%Formulas for both finite difference summation-by-parts and modal DG formulations are derived and the resulting matrices are applied to both two-derivative explicit Runge Kutta methods and implicit time-stepping schemes based on high order DG formulations of the 1D Burgers', 1D shallow water, and 2D compressible Euler equations. 
Future work will investigate the application of such formulas towards preconditioners and sensitivity analysis.

\section{Acknowledgments}

The authors gratefully acknowledge support from the National Science Foundation under award DMS-CAREER-1943186.  Christina Taylor also acknowledges support from the Ken Kennedy Institute 2019-2020 BP Graduate Fellowship.

\bibliographystyle{unsrt}
\bibliography{refs}

\appendix

\section{\rnote{Higher-dimensional domains and curved elements}}

\rnote{
The generalization to higher dimensional domains and curved geometric mappings is straightforward, but notationally more complicated.  The construction of skew-symmetric SBP matrices on curved meshes follows from approaches detailed in \cite{carpenter2014entropy, crean2018entropy, chan2017discretely, chan2018discretely, chan2018efficient, chan2019skew, hicken2020entropy}, which are summarized here. 

Let $\bm{x}, \hat{\bm{x}} \in \mathbb{R}^d$ denote $d$-dimensional physical and reference coordinates, respectively. Let $\hat{\fnt{Q}}_j$ denote the reference SBP operator corresponding to differentiation with respect to $\hat{x}_j$ which satisfies the SBP property (\ref{eq:fvQ}). This operator can be constructed any number of ways: using the tensor product of 1D SBP operators \cite{carpenter2014entropy}, multi-dimensional SBP operators \cite{hicken2016multidimensional}, or hybridized SBP operators \cite{chan2017discretely, chan2019skew}.  Consider now a curved element $D^k$ which is the image of a reference element under some differentiable mapping such that for $\bm{x}\in D^k$, $\bm{x} = \bm{\Phi}(\hat{\bm{x}})$. Then, derivatives with respect to physical coordinates can be computed via the chain rule $\pd{u}{x_i} = \sum_{j=1}^d \pd{u}{\hat{x}_j}\pd{\hat{x}_j}{x_i}$. For geometric terms which satisfy a discrete geometric conservation law (GCL) \cite{thomas1979geometric, crean2018entropy, chan2018discretely}, we can further manipulate the chain rule to show that
\[
\pd{u}{x_i} = \frac{1}{2}\sum_{j=1}^d \LRp{\pd{u}{\hat{x}_j}\pd{\hat{x}_j}{x_i} + \pd{}{\hat{x}_j}\LRp{u\pd{\hat{x}_j}{x_i}}}.
\]
We will construct a physical SBP operator by mimicking this form of the chain rule. Let $J$ denote the determinant of the Jacobian of $\bm{\Phi}$. Define the scaled geometric terms $\bm{g}_{ij} = J \pd{\hat{x}_j}{x_i}$, and let $\fnt{g}_{ij}$ denote the vector containing values of $\bm{g}_{ij}$ evaluated at nodal points. Define the physical SBP operator $\fnt{Q}_i$ as
\[
\fnt{Q}_i = \frac{1}{2} \sum_{j=1}^d \LRp{\diag{\fnt{g}_{ij}}\hat{\fnt{Q}}_j + \hat{\fnt{Q}}_j \diag{\fnt{g}_{ij}}}. 
\]
Then, one can show (using relationships between geometric terms $\bm{g}_{ij}$ and reference/physical normals) that $\fnt{Q}_i$ satisfies a physical SBP property $\fnt{Q}_i + \fnt{Q}_i^T = \fnt{B}_i$, where $\fnt{B}_i$ is a diagonal matrix whose entries consist of values (at face nodes) of the $i$th component of the outward normal scaled by the surface Jacobian and surface quadrature weights \cite{crean2018entropy, chan2018discretely}. Given connectivity maps between face nodes on different elements, the physical SBP operators $\fnt{Q}_i$ and $\fnt{B}_i$ can then be used to construct global SBP operators analogous to (\ref{eq:sbpmat}) in two and three dimensions. 
%We can now construct a global SBP operator using physical SBP operators, which we illustrate using the simple example of a mesh with four triangles as shown in Figure~\ref{fig:mesh}. Let $\fnt{Q}_{i,k}$ denote the SBP operator with respect to the $i$th physical coordinate on the $k$th element. We will construct a global SBP operator $\fnt{Q}_{1, \Omega}$ corresponding to differentiation with respect to the $x$-coordinate. 
}

\section{Two-derivative time-stepping methods}
\label{sec:tdrk}
Consider a general system of ODEs
\[
\td{\fnt{u}}{t} = \fnt{f}(\fnt{u}).
\]
Two-derivative explicit time-stepping methods are constructed based on the assumption that second derivatives of $\fnt{u}$ in time are available \cite{chan2010explicit, christlieb2016explicit}.  The resulting schemes can achieve higher order accuracy with fewer stages and function evaluations compared to standard Runge-Kutta methods.  

Let $\fnt{g}(\fnt{u})$ denote the second derivative of $\fnt{u}$ in time
\[
\fnt{g}(\fnt{u}) = \frac{{\rm d}^2 \fnt{u}}{{\rm dt}^2} = \td{}{t} \fnt{f}(\fnt{u}) = \pd{\fnt{f}}{\fnt{u}}\td{\fnt{u}}{t} = \pd{\fnt{f}}{\fnt{u}}\fnt{f}(\fnt{u}),
\]
where we have used the chain rule in the final step.  The simplest two-derivative Runge-Kutta method is the one-stage second order scheme \cite{chan2010explicit}
\[
\fnt{u}^{k+1} = \fnt{u}^k + \Delta t \fnt{f}(\fnt{u}^k) + \frac{\Delta t^2}{2} \fnt{g}(\fnt{u}^k),
\]
where $\fnt{u}^k$ denotes the solution at the $k$th time-step.  We examine the one-stage, two-stage, and three-stage two-derivative Runge Kutta given in \cite{chan2010explicit}, which we refer to as TDRK-1, TDRK-2, TDRK-3.\footnote{Five different three-stage schemes are presented in \cite{chan2010explicit}. We use the scheme corresponding to free parameter $c_3 = 2/3$, which the authors report as the best performing three-stage scheme.} These schemes are second, fourth, and fifth order accurate, respectively. We also provide reference results using a low-storage $4$th order 5-stage Runge-Kutta method (RK-45). 

%For the Burgers' equation, the resulting source terms have the form
%\eq{
%f(x,t) &= k \cos(k t) \sin(\pi x) + \pi \sin^2(k t) \cos(\pi x) \sin(\pi x), \\
%\pd{f}{t} &= -k^2 \sin(k t) \sin(\pi x) + 2k \pi \sin(k t) \cos(\pi x) \sin(\pi x).
%}
%For the shallow water equation, the source terms 

\begin{figure}
\centering
\subfloat[Burgers' equation]{
\begin{tikzpicture}
\begin{loglogaxis}[
    width=.49\textwidth,
    xlabel={Time-step $dt$},
    ylabel={$L^2$ errors}, 
    ymin=1e-21, ymax=.1,
    legend pos=south east, legend cell align=left, legend style={font=\tiny},	
    xmajorgrids=true, ymajorgrids=true, grid style=dashed, 
    legend entries={TDRK-1,TDRK-2,TDRK-3,RK-45}
]
\pgfplotsset{
cycle list={{blue, mark=*}, {red, dashed, mark=square*},{magenta, dashdotted ,mark=triangle*},{black,dashed,mark=x}}
}
% N = 1
\addplot+[semithick, mark options={solid, fill=markercolor, scale=1.5}]
coordinates{(0.00223214,0.0136524)(0.00111607,0.00342134)(0.000558036,0.00085604)(0.000279018,0.000214076)(0.000139509,5.35e-05)};
\addplot+[semithick, mark options={solid, fill=markercolor, scale=1.5}]
coordinates{(0.00223214,2.15e-05)(0.00111607,1.35e-06)(0.000558036,8.43e-08)(0.000279018,5.27e-09)(0.000139509,3.29e-10)};
\addplot+[semithick, mark options={solid, fill=markercolor, scale=1.5}]
coordinates{(0.00223214,5.05e-08)(0.00111607,7.86e-10)(0.000558036,1.23e-11)(0.000279018,2.04e-13)(0.000139509,7.13e-15)};
\addplot+[semithick, mark options={solid, fill=markercolor, scale=1.5}]
coordinates{(0.00223214,2.33e-06)(0.00111607,1.46e-07)(0.000558036,9.14e-09)(0.000279,5.71e-10)(0.00014,3.57e-11)};
\node at (axis cs:2.4,2.8) {$N = 1$};
\end{loglogaxis}
\end{tikzpicture}
}
\subfloat[Shallow water equations]{
\begin{tikzpicture}
\begin{loglogaxis}[
    width=.49\textwidth,
    xlabel={Time-step $dt$},
    ylabel={$L^2$ errors}, 
    ymin=1e-21, ymax=.1,
    legend pos=south east, legend cell align=left, legend style={font=\tiny},	
    xmajorgrids=true, ymajorgrids=true, grid style=dashed, 
    legend entries={TDRK-1,TDRK-2,TDRK-3,RK-45}
]
\pgfplotsset{
cycle list={{blue, mark=*}, {red, dashed, mark=square*},{magenta, dashdotted ,mark=triangle*},{black,dashed,mark=x}}
}
% N = 1
\addplot+[semithick, mark options={solid, fill=markercolor, scale=1.5}]
coordinates{(0.000447227,0.000487391)(0.000223614,7.93e-05)(0.000111807,1.98e-05)(5.59e-05,4.94e-06)(2.8e-05,1.23e-06)};
\addplot+[semithick, mark options={solid, fill=markercolor, scale=1.5}]
coordinates{(0.000447227,1.16e-08)(0.000223614,7.23e-10)(0.000111807,4.52e-11)(5.59e-05,2.84e-12)(2.8e-05,3.04e-13)};
\addplot+[semithick, mark options={solid, fill=markercolor, scale=1.5}]
coordinates{(0.000447227,1.52e-11)(0.000223614,4.76e-13)(0.000111807,3.52e-14)(5.59e-05,1.61e-13)};
\addplot+[semithick, mark options={solid, fill=markercolor, scale=1.5}]
coordinates{(0.000447227,1.34e-09)(0.000223614,8.39e-11)(0.000111807,5.27e-12)(5.59e-05,4.67e-13)(2.8e-05,6.19e-13)};
\node at (axis cs:2.4,2.8) {$N = 1$};
\end{loglogaxis}
\end{tikzpicture}
}
\caption{$L^2$ errors for manufactured solutions of the Burgers and shallow water equations for three TDRK schemes under various time-step sizes. Errors for RK-45 are also included for reference.}
\label{fig:tdrk_err}
\end{figure}
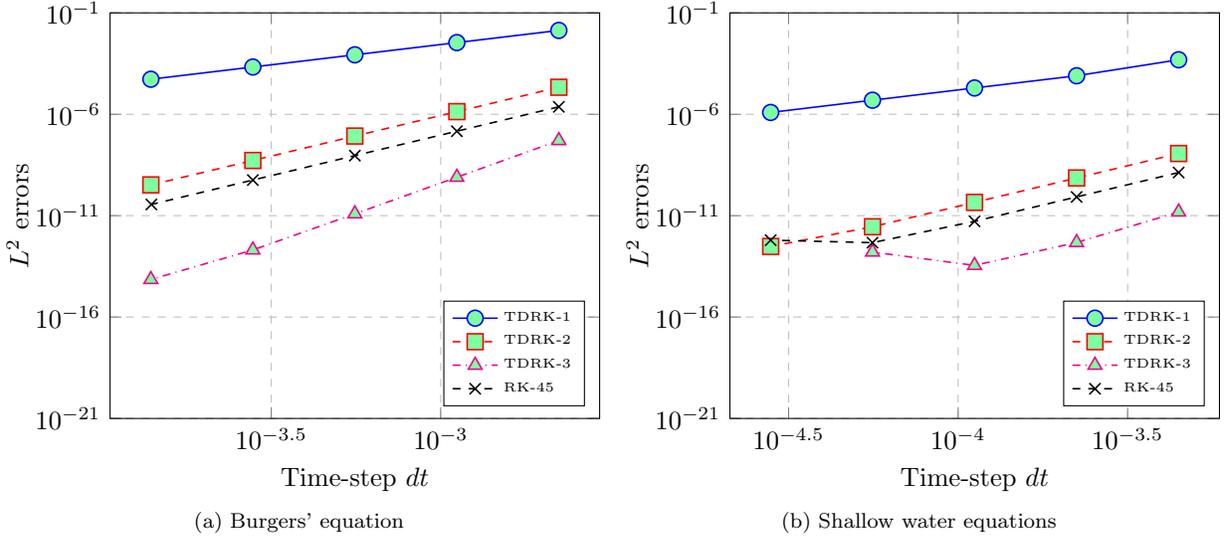
We examine the performance of two-derivative time-stepping methods for the one-dimensional Burgers' and shallow water equations using an entropy conservative and entropy stable spectral (Lobatto) collocation method of degree $N=40$ on a single periodic domain $[-1,1]$. For the entropy stable scheme, we apply a local Lax-Friedrichs penalty at the boundaries to produce entropy dissipation. We compute $L^2$ errors for a manufactured solution where all solution components have the form $\sin(k t) \sin(\pi x)$ with $k = 100$. \bnote{TDRK methods also require derivatives of source terms $f(x,t)$ associated with manufactured solutions, which we compute analytically. }

Figure~\ref{fig:tdrk_err} plots $L^2$ errors (computed using a higher accuracy Gaussian quadrature rule at final time $T=5$) against the time-step size, while Table~\ref{tab:tdrk} shows computed rates of convergence for each TDRK scheme. We observe that all except one TDRK scheme achieves the expected rate of convergence up until the point at which errors are affected by numerical roundoff. The outlier is the TDRK-3 scheme, which converges at the expected rate of $O(dt^5)$ for the shallow water equations, but achieves a higher $O(dt^6)$ rate of convergence for the Burgers' equation. 
\begin{table}
\centering
\subfloat[Burgers equation]{
\begin{tabular}{c|cccc}
$dt/dt_0$ & 1/2 & 1/4 & 1/8 & 1/16\\
\hline
TDRK-1 & 1.997 & 1.999& 2.000 & 2.000\\
TDRK-2 & 3.999&  4.000& 4.000 & 4.000 \\
TDRK-3 & 6.006 & {5.993} & \it{5.916} & \it{4.842}
\end{tabular}}
\qquad 
\subfloat[Shallow water equation]{
\begin{tabular}{c|cccc}
$dt/dt_0$ & 1/2 & 1/4 & 1/8 & 1/16\\
\hline
TDRK-1 & 2.620& 2.003& 2.001 & 2.001\\
TDRK-2 & 4.002 & 4.001&4.001 &4.000\\
TDRK-3 & 4.998 & \it{3.757}& \it{-2.193} &
\end{tabular}}
\caption{Computed rates of convergence with respect to $dt$ for different TDRK schemes ($dt_0$ denotes the initial time-step). Italicized numbers denote rates which are likely affected by numerical roundoff.}
\label{tab:tdrk}
\end{table}
We also observe that the 4th order RK-45 scheme is slightly more accurate than the 4th order TDRK-2 scheme. As noted in \cite{chan2010explicit}, the 2-stage TDRK-2 scheme requires only one evaluation of $\fnt{f}(\fnt{u})$ and two evaluations of $\fnt{g}(\fnt{u})$. However, when $\fnt{g}(\fnt{u})$ is computed using a Jacobian-vector product, this corresponds to two evaluations of $\fnt{f}(\fnt{u})$ and two Jacobian-vector products. Because evaluating Jacobian-vector products are at least as expensive as evaluating $\fnt{f}(\fnt{u})$, it is not clear that the TDRK-2 scheme would be more efficient than either RK-45 or the standard 4-stage 4th order Runge-Kutta method in practice. 

\begin{figure}
\centering
\subfloat[Burgers' equation]{
\begin{tikzpicture}
\begin{axis}[
    width=.49\textwidth,
    xlabel={Time },
    ylabel={$\mathsf{S}(t)-\mathsf{S}(0)$}, 
    ymin=-.015,
    legend pos=south west, legend cell align=left, legend style={font=\tiny},	
    xmajorgrids=true, ymajorgrids=true, grid style=dashed, 
    legend entries={TDRK-2 (EC),RK-45 (EC),TDRK-2 (ES),RK-45 (ES)}
    ]
\pgfplotsset{
cycle list={{blue}, {red, dashed},{magenta, dashdotted},{black,dashed}}
}
\addplot+[semithick, mark options={solid, fill=markercolor, scale=1.5},line width=1.5pt]
coordinates{(0,0)(0.0223614,0)(0.0447227,0)(0.0670841,0)(0.0894454,0)(0.111807,0)(0.134168,0)(0.15653,0)(0.178891,0)(0.201252,0)(0.223614,0)(0.245975,0)(0.268336,0)(0.290698,0)(0.313059,0)(0.33542,0)(0.357782,0)(0.380143,0)(0.402504,0)(0.424866,0)(0.447227,0)(0.469589,0)(0.49195,0)(0.514311,0)(0.536673,0)(0.559034,0)(0.581395,0)(0.603757,0)(0.626118,0)(0.648479,0)(0.670841,0)(0.693202,0)(0.715564,0)(0.737925,0)(0.760286,0)(0.782648,0)(0.805009,0)(0.82737,0)(0.849732,0)(0.872093,0)(0.894454,0)(0.916816,0)(0.939177,0)(0.961538,0)(0.9839,0)(1.00626,0)(1.02862,0)(1.05098,0)(1.07335,0)(1.09571,0)(1.11807,0)(1.14043,0)(1.16279,0)(1.18515,0)(1.20751,0)(1.22987,0)(1.25224,0)(1.2746,0)(1.29696,0)(1.31932,0)(1.34168,0)(1.36404,0)(1.3864,0)(1.40877,0)(1.43113,0)(1.45349,0)(1.47585,0)(1.49821,0)(1.52057,0)(1.54293,-1e-11)(1.5653,-1e-11)(1.58766,-2e-11)(1.61002,-3e-11)(1.63238,-3e-11)(1.65474,-3e-11)(1.6771,-3e-11)(1.69946,-3e-11)(1.72182,-3e-11)(1.74419,-4e-11)(1.76655,-7e-11)(1.78891,-1.3e-10)(1.81127,-2.4e-10)(1.83363,-3.5e-10)(1.85599,-4e-10)(1.87835,-3.3e-10)(1.90072,-3.8e-10)(1.92308,-4.4e-10)(1.94544,-4.3e-10)(1.9678,-4.6e-10)(1.99016,-5.2e-10)(2.01252,-4.4e-10)(2.03488,-2.6e-10)(2.05725,-1.4e-10)(2.07961,-1.1e-10)(2.10197,-1e-10)(2.12433,-8e-11)(2.14669,-8e-11)(2.16905,-9e-11)(2.19141,-7e-11)(2.21377,-1.04e-09)(2.23614,-1.44e-09)(2.2585,-1.21e-09)(2.28086,-1.21e-09)(2.30322,-1.48e-09)(2.32558,-1.55e-09)(2.34794,-2.16e-09)(2.3703,-2.16e-09)(2.39267,-9.9e-10)(2.41503,-2.8e-10)(2.43739,-1.5e-10)(2.45975,-1.2e-10)(2.48211,-1.1e-10)(2.50447,-1.1e-10)(2.52683,-1e-10)(2.54919,-8e-11)(2.57156,-8e-11)(2.59392,-8e-11)(2.61628,-8e-11)(2.63864,-8e-11)(2.661,-4.5e-10)(2.68336,-8.9e-10)(2.70572,-9.5e-10)(2.72809,-7.5e-10)(2.75045,-7.6e-10)(2.77281,-8.3e-10)(2.79517,-1.03e-09)(2.81753,-2.52e-09)(2.83989,-3.62e-09)(2.86225,-2.6e-09)(2.88462,-1.53e-09)(2.90698,-1.26e-09)(2.92934,-8.8e-10)(2.9517,-7.4e-10)(2.97406,-6.8e-10)(2.99642,-5.7e-10)(3.01878,-4.7e-10)(3.04114,-4.2e-10)(3.06351,-3.9e-10)(3.08587,-3.7e-10)(3.10823,-3.5e-10)(3.13059,-3.5e-10)(3.15295,-3.4e-10)(3.17531,-3.4e-10)(3.19767,-3.4e-10)(3.22004,-3.4e-10)(3.2424,-3.4e-10)(3.26476,-3.5e-10)(3.28712,-3.7e-10)(3.30948,-3.4e-10)(3.33184,-2.9e-10)(3.3542,-6.9e-10)(3.37657,-5.25e-09)(3.39893,-7.56e-09)(3.42129,-5.57e-09)(3.44365,-4.49e-09)(3.46601,-4.39e-09)(3.48837,-4.51e-09)(3.51073,-7.64e-09)(3.53309,-8.11e-09)(3.55546,-1.94e-09)(3.57782,-1.926e-08)(3.60018,-5.8e-10)(3.62254,4.3e-10)(3.6449,4.7e-10)(3.66726,5.6e-10)(3.68962,-2.1e-10)(3.71199,9.8e-10)(3.73435,1.41e-09)(3.75671,1.33e-09)(3.77907,1.3e-09)(3.80143,1.06e-09)(3.82379,1.2e-09)(3.84615,1.29e-09)(3.86852,1.34e-09)(3.89088,1.38e-09)(3.91324,1.33e-09)(3.9356,1.24e-09)(3.95796,1.49e-09)(3.98032,1.6e-09)(4.00268,1.57e-09)(4.02504,1.56e-09)(4.04741,1.53e-09)(4.06977,1.46e-09)(4.09213,1.42e-09)(4.11449,1.39e-09)(4.13685,1.4e-09)(4.15921,1.46e-09)(4.18157,1.47e-09)(4.20394,1.46e-09)(4.2263,1.48e-09)(4.24866,1.47e-09)(4.27102,1.39e-09)(4.29338,1.28e-09)(4.31574,1.25e-09)(4.3381,1.33e-09)(4.36047,1.4e-09)(4.38283,1.41e-09)(4.40519,1.43e-09)(4.42755,1.46e-09)(4.44991,1.45e-09)(4.47227,1.39e-09)(4.49463,1.37e-09)(4.51699,1.45e-09)(4.53936,1.53e-09)(4.56172,1.56e-09)(4.58408,1.57e-09)(4.60644,1.58e-09)(4.6288,1.59e-09)(4.65116,1.59e-09)(4.67352,1.58e-09)(4.69589,1.37e-09)(4.71825,1.45e-09)(4.74061,1.53e-09)(4.76297,1.55e-09)(4.78533,1.56e-09)(4.80769,1.56e-09)(4.83005,1.58e-09)(4.85242,1.62e-09)(4.87478,1.63e-09)(4.89714,1.63e-09)(4.9195,1.63e-09)(4.94186,1.63e-09)(4.96422,1.63e-09)(4.98658,1.62e-09)};
\addplot+[semithick, mark options={solid, fill=markercolor, scale=1.5},line width=1.5pt]
coordinates{(0,0)(0.0223614,0)(0.0447227,0)(0.0670841,0)(0.0894454,0)(0.111807,0)(0.134168,0)(0.15653,0)(0.178891,0)(0.201252,0)(0.223614,0)(0.245975,0)(0.268336,0)(0.290698,0)(0.313059,0)(0.33542,0)(0.357782,0)(0.380143,0)(0.402504,0)(0.424866,0)(0.447227,0)(0.469589,0)(0.49195,0)(0.514311,0)(0.536673,0)(0.559034,0)(0.581395,0)(0.603757,0)(0.626118,0)(0.648479,0)(0.670841,0)(0.693202,0)(0.715564,0)(0.737925,0)(0.760286,0)(0.782648,0)(0.805009,0)(0.82737,0)(0.849732,0)(0.872093,0)(0.894454,0)(0.916816,0)(0.939177,0)(0.961538,0)(0.9839,0)(1.00626,0)(1.02862,0)(1.05098,0)(1.07335,0)(1.09571,0)(1.11807,0)(1.14043,0)(1.16279,0)(1.18515,0)(1.20751,0)(1.22987,0)(1.25224,0)(1.2746,0)(1.29696,0)(1.31932,0)(1.34168,0)(1.36404,0)(1.3864,0)(1.40877,0)(1.43113,0)(1.45349,0)(1.47585,0)(1.49821,0)(1.52057,0)(1.54293,0)(1.5653,0)(1.58766,-1e-11)(1.61002,-1e-11)(1.63238,-1e-11)(1.65474,-1e-11)(1.6771,-1e-11)(1.69946,-1e-11)(1.72182,-1e-11)(1.74419,-1e-11)(1.76655,-2e-11)(1.78891,-3e-11)(1.81127,-6e-11)(1.83363,-9e-11)(1.85599,-1e-10)(1.87835,-1e-10)(1.90072,-1e-10)(1.92308,-1.1e-10)(1.94544,-1.2e-10)(1.9678,-1.2e-10)(1.99016,-1.3e-10)(2.01252,-1.2e-10)(2.03488,-8e-11)(2.05725,-5e-11)(2.07961,-4e-11)(2.10197,-3e-11)(2.12433,-3e-11)(2.14669,-3e-11)(2.16905,-4e-11)(2.19141,-4e-11)(2.21377,-2.4e-10)(2.23614,-2.8e-10)(2.2585,-2.6e-10)(2.28086,-2.7e-10)(2.30322,-3.1e-10)(2.32558,-3.6e-10)(2.34794,-5.1e-10)(2.3703,-5e-10)(2.39267,-2.5e-10)(2.41503,-9e-11)(2.43739,-5e-11)(2.45975,-4e-11)(2.48211,-3e-11)(2.50447,-3e-11)(2.52683,-3e-11)(2.54919,-3e-11)(2.57156,-3e-11)(2.59392,-3e-11)(2.61628,-3e-11)(2.63864,-4e-11)(2.661,-1.3e-10)(2.68336,-2.2e-10)(2.70572,-2.4e-10)(2.72809,-2e-10)(2.75045,-2e-10)(2.77281,-2.2e-10)(2.79517,-3e-10)(2.81753,-6.8e-10)(2.83989,-9.3e-10)(2.86225,-7.4e-10)(2.88462,-4.9e-10)(2.90698,-3.9e-10)(2.92934,-3e-10)(2.9517,-2.7e-10)(2.97406,-2.4e-10)(2.99642,-2.2e-10)(3.01878,-1.9e-10)(3.04114,-1.8e-10)(3.06351,-1.7e-10)(3.08587,-1.7e-10)(3.10823,-1.6e-10)(3.13059,-1.6e-10)(3.15295,-1.6e-10)(3.17531,-1.6e-10)(3.19767,-1.6e-10)(3.22004,-1.6e-10)(3.2424,-1.6e-10)(3.26476,-1.6e-10)(3.28712,-1.7e-10)(3.30948,-1.7e-10)(3.33184,-1.8e-10)(3.3542,-3.3e-10)(3.37657,-1.4e-09)(3.39893,-2.03e-09)(3.42129,-1.68e-09)(3.44365,-1.36e-09)(3.46601,-1.32e-09)(3.48837,-1.43e-09)(3.51073,-2.17e-09)(3.53309,-2.29e-09)(3.55546,-1.37e-09)(3.57782,-5.52e-09)(3.60018,-1.88e-09)(3.62254,-1.29e-09)(3.6449,-1.26e-09)(3.66726,-1.27e-09)(3.68962,-1.44e-09)(3.71199,-1.19e-09)(3.73435,-1.13e-09)(3.75671,-1.15e-09)(3.77907,-1.15e-09)(3.80143,-1.19e-09)(3.82379,-1.15e-09)(3.84615,-1.14e-09)(3.86852,-1.13e-09)(3.89088,-1.12e-09)(3.91324,-1.14e-09)(3.9356,-1.15e-09)(3.95796,-1.1e-09)(3.98032,-1.09e-09)(4.00268,-1.1e-09)(4.02504,-1.1e-09)(4.04741,-1.11e-09)(4.06977,-1.12e-09)(4.09213,-1.13e-09)(4.11449,-1.14e-09)(4.13685,-1.14e-09)(4.15921,-1.13e-09)(4.18157,-1.12e-09)(4.20394,-1.12e-09)(4.2263,-1.12e-09)(4.24866,-1.13e-09)(4.27102,-1.15e-09)(4.29338,-1.17e-09)(4.31574,-1.18e-09)(4.3381,-1.16e-09)(4.36047,-1.15e-09)(4.38283,-1.14e-09)(4.40519,-1.13e-09)(4.42755,-1.13e-09)(4.44991,-1.13e-09)(4.47227,-1.14e-09)(4.49463,-1.15e-09)(4.51699,-1.13e-09)(4.53936,-1.12e-09)(4.56172,-1.11e-09)(4.58408,-1.1e-09)(4.60644,-1.1e-09)(4.6288,-1.1e-09)(4.65116,-1.1e-09)(4.67352,-1.11e-09)(4.69589,-1.15e-09)(4.71825,-1.14e-09)(4.74061,-1.12e-09)(4.76297,-1.11e-09)(4.78533,-1.11e-09)(4.80769,-1.11e-09)(4.83005,-1.1e-09)(4.85242,-1.09e-09)(4.87478,-1.09e-09)(4.89714,-1.09e-09)(4.9195,-1.09e-09)(4.94186,-1.09e-09)(4.96422,-1.1e-09)(4.98658,-1.1e-09)};
\addplot+[semithick, mark options={solid, fill=markercolor, scale=1.5},line width=1.5pt]
coordinates{(0,0)(0.0223614,0)(0.0447227,0)(0.0670841,0)(0.0894454,0)(0.111807,0)(0.134168,0)(0.15653,0)(0.178891,0)(0.201252,0)(0.223614,0)(0.245975,0)(0.268336,0)(0.290698,0)(0.313059,0)(0.33542,0)(0.357782,0)(0.380143,0)(0.402504,0)(0.424866,0)(0.447227,0)(0.469589,0)(0.49195,0)(0.514311,0)(0.536673,0)(0.559034,0)(0.581395,0)(0.603757,0)(0.626118,0)(0.648479,0)(0.670841,0)(0.693202,0)(0.715564,-1e-11)(0.737925,-2e-11)(0.760286,-6e-11)(0.782648,-1.4e-10)(0.805009,-3.3e-10)(0.82737,-7.6e-10)(0.849732,-1.67e-09)(0.872093,-3.49e-09)(0.894454,-7e-09)(0.916816,-1.341e-08)(0.939177,-2.454e-08)(0.961538,-4.293e-08)(0.9839,-7.193e-08)(1.00626,-1.1573e-07)(1.02862,-1.7935e-07)(1.05098,-2.6866e-07)(1.07335,-3.9052e-07)(1.09571,-5.5295e-07)(1.11807,-7.6558e-07)(1.14043,-1.04021e-06)(1.16279,-1.39182e-06)(1.18515,-1.83997e-06)(1.20751,-2.41102e-06)(1.22987,-3.1416e-06)(1.25224,-4.08397e-06)(1.2746,-5.31466e-06)(1.29696,-6.9483e-06)(1.31932,-9.15987e-06)(1.34168,-1.22207e-05)(1.36404,-1.65553e-05)(1.3864,-2.28293e-05)(1.40877,-3.20739e-05)(1.43113,-4.58364e-05)(1.45349,-6.62888e-05)(1.47585,-9.61003e-05)(1.49821,-0.000137706)(1.52057,-0.000191615)(1.54293,-0.000254175)(1.5653,-0.000316883)(1.58766,-0.000369751)(1.61002,-0.000407211)(1.63238,-0.000430734)(1.65474,-0.000445867)(1.6771,-0.000458355)(1.69946,-0.000473298)(1.72182,-0.000496606)(1.74419,-0.000537611)(1.76655,-0.000612934)(1.78891,-0.000747119)(1.81127,-0.000954191)(1.83363,-0.00119761)(1.85599,-0.00140056)(1.87835,-0.00153394)(1.90072,-0.0016329)(1.92308,-0.0017422)(1.94544,-0.00189255)(1.9678,-0.00208925)(1.99016,-0.00229255)(2.01252,-0.00243918)(2.03488,-0.00251096)(2.05725,-0.00253841)(2.07961,-0.00255108)(2.10197,-0.0025626)(2.12433,-0.00257974)(2.14669,-0.0026085)(2.16905,-0.00265862)(2.19141,-0.00274844)(2.21377,-0.00290529)(2.23614,-0.00314736)(2.2585,-0.00344181)(2.28086,-0.00370678)(2.30322,-0.00390306)(2.32558,-0.00406854)(2.34794,-0.00426148)(2.3703,-0.00451741)(2.39267,-0.00480984)(2.41503,-0.00504408)(2.43739,-0.00516249)(2.45975,-0.00520041)(2.48211,-0.00520861)(2.50447,-0.00520987)(2.52683,-0.00521001)(2.54919,-0.00521009)(2.57156,-0.00521066)(2.59392,-0.00521384)(2.61628,-0.00522256)(2.63864,-0.0052399)(2.661,-0.00527253)(2.68336,-0.00533586)(2.70572,-0.00545769)(2.72809,-0.00566468)(2.75045,-0.00593171)(2.77281,-0.00615877)(2.79517,-0.00628063)(2.81753,-0.00633455)(2.83989,-0.00637725)(2.86225,-0.00644893)(2.88462,-0.00659191)(2.90698,-0.00685945)(2.92934,-0.00724497)(2.9517,-0.00758817)(2.97406,-0.00776219)(2.99642,-0.00782911)(3.01878,-0.00787119)(3.04114,-0.00791976)(3.06351,-0.00797755)(3.08587,-0.00803439)(3.10823,-0.00807721)(3.13059,-0.00810023)(3.15295,-0.00810841)(3.17531,-0.00811013)(3.19767,-0.00811034)(3.22004,-0.00811039)(3.2424,-0.00811066)(3.26476,-0.00811189)(3.28712,-0.00811441)(3.30948,-0.00811775)(3.33184,-0.00812163)(3.3542,-0.00812617)(3.37657,-0.00813198)(3.39893,-0.00814022)(3.42129,-0.0081529)(3.44365,-0.00817348)(3.46601,-0.00820797)(3.48837,-0.00826693)(3.51073,-0.00836755)(3.53309,-0.00853047)(3.55546,-0.00875971)(3.57782,-0.00901146)(3.60018,-0.00921461)(3.62254,-0.00934942)(3.6449,-0.00945643)(3.66726,-0.00959007)(3.68962,-0.00980368)(3.71199,-0.0101254)(3.73435,-0.0104817)(3.75671,-0.010728)(3.77907,-0.0108359)(3.80143,-0.0108738)(3.82379,-0.0108877)(3.84615,-0.0108919)(3.86852,-0.0108922)(3.89088,-0.0108963)(3.91324,-0.0109139)(3.9356,-0.010934)(3.95796,-0.0109465)(3.98032,-0.0109501)(4.00268,-0.0109502)(4.02504,-0.0109554)(4.04741,-0.0109665)(4.06977,-0.010976)(4.09213,-0.0109828)(4.11449,-0.0109876)(4.13685,-0.010991)(4.15921,-0.0109937)(4.18157,-0.0109958)(4.20394,-0.0109976)(4.2263,-0.0109993)(4.24866,-0.0110011)(4.27102,-0.0110033)(4.29338,-0.0110067)(4.31574,-0.0110119)(4.3381,-0.0110202)(4.36047,-0.0110331)(4.38283,-0.011053)(4.40519,-0.011084)(4.42755,-0.0111324)(4.44991,-0.0112065)(4.47227,-0.0113118)(4.49463,-0.0114412)(4.51699,-0.0115681)(4.53936,-0.0116607)(4.56172,-0.0117102)(4.58408,-0.011732)(4.60644,-0.0117437)(4.6288,-0.0117565)(4.65116,-0.0117789)(4.67352,-0.0118212)(4.69589,-0.0119004)(4.71825,-0.0120395)(4.74061,-0.0122437)(4.76297,-0.0124631)(4.78533,-0.0126217)(4.80769,-0.0127024)(4.83005,-0.0127415)(4.85242,-0.0127711)(4.87478,-0.012806)(4.89714,-0.0128505)(4.9195,-0.0129012)(4.94186,-0.0129493)(4.96422,-0.0129844)(4.98658,-0.0130019)};
\addplot+[semithick, mark options={solid, fill=markercolor, scale=1.5},line width=1.5pt]
coordinates{(0,0)(0.0223614,0)(0.0447227,0)(0.0670841,0)(0.0894454,0)(0.111807,0)(0.134168,0)(0.15653,0)(0.178891,0)(0.201252,0)(0.223614,0)(0.245975,0)(0.268336,0)(0.290698,0)(0.313059,0)(0.33542,0)(0.357782,0)(0.380143,0)(0.402504,0)(0.424866,0)(0.447227,0)(0.469589,0)(0.49195,0)(0.514311,0)(0.536673,0)(0.559034,0)(0.581395,0)(0.603757,0)(0.626118,0)(0.648479,0)(0.670841,0)(0.693202,0)(0.715564,-1e-11)(0.737925,-2e-11)(0.760286,-6e-11)(0.782648,-1.4e-10)(0.805009,-3.3e-10)(0.82737,-7.6e-10)(0.849732,-1.67e-09)(0.872093,-3.49e-09)(0.894454,-7e-09)(0.916816,-1.341e-08)(0.939177,-2.454e-08)(0.961538,-4.293e-08)(0.9839,-7.193e-08)(1.00626,-1.1573e-07)(1.02862,-1.7935e-07)(1.05098,-2.6866e-07)(1.07335,-3.9052e-07)(1.09571,-5.5295e-07)(1.11807,-7.6558e-07)(1.14043,-1.04021e-06)(1.16279,-1.39182e-06)(1.18515,-1.83997e-06)(1.20751,-2.41102e-06)(1.22987,-3.1416e-06)(1.25224,-4.08397e-06)(1.2746,-5.31466e-06)(1.29696,-6.9483e-06)(1.31932,-9.15987e-06)(1.34168,-1.22207e-05)(1.36404,-1.65553e-05)(1.3864,-2.28293e-05)(1.40877,-3.20739e-05)(1.43113,-4.58364e-05)(1.45349,-6.62888e-05)(1.47585,-9.61003e-05)(1.49821,-0.000137706)(1.52057,-0.000191615)(1.54293,-0.000254175)(1.5653,-0.000316883)(1.58766,-0.000369751)(1.61002,-0.000407211)(1.63238,-0.000430734)(1.65474,-0.000445867)(1.6771,-0.000458355)(1.69946,-0.000473298)(1.72182,-0.000496606)(1.74419,-0.000537611)(1.76655,-0.000612934)(1.78891,-0.000747119)(1.81127,-0.000954191)(1.83363,-0.00119761)(1.85599,-0.00140056)(1.87835,-0.00153394)(1.90072,-0.0016329)(1.92308,-0.0017422)(1.94544,-0.00189255)(1.9678,-0.00208925)(1.99016,-0.00229255)(2.01252,-0.00243918)(2.03488,-0.00251096)(2.05725,-0.00253841)(2.07961,-0.00255108)(2.10197,-0.0025626)(2.12433,-0.00257974)(2.14669,-0.0026085)(2.16905,-0.00265862)(2.19141,-0.00274844)(2.21377,-0.00290529)(2.23614,-0.00314736)(2.2585,-0.00344181)(2.28086,-0.00370677)(2.30322,-0.00390306)(2.32558,-0.00406854)(2.34794,-0.00426148)(2.3703,-0.00451741)(2.39267,-0.00480984)(2.41503,-0.00504408)(2.43739,-0.00516249)(2.45975,-0.00520041)(2.48211,-0.00520861)(2.50447,-0.00520987)(2.52683,-0.00521001)(2.54919,-0.00521009)(2.57156,-0.00521066)(2.59392,-0.00521384)(2.61628,-0.00522256)(2.63864,-0.0052399)(2.661,-0.00527253)(2.68336,-0.00533586)(2.70572,-0.00545769)(2.72809,-0.00566468)(2.75045,-0.00593171)(2.77281,-0.00615877)(2.79517,-0.00628063)(2.81753,-0.00633455)(2.83989,-0.00637725)(2.86225,-0.00644893)(2.88462,-0.00659191)(2.90698,-0.00685945)(2.92934,-0.00724497)(2.9517,-0.00758817)(2.97406,-0.00776219)(2.99642,-0.00782911)(3.01878,-0.00787119)(3.04114,-0.00791976)(3.06351,-0.00797755)(3.08587,-0.00803439)(3.10823,-0.00807721)(3.13059,-0.00810023)(3.15295,-0.00810841)(3.17531,-0.00811013)(3.19767,-0.00811034)(3.22004,-0.00811039)(3.2424,-0.00811066)(3.26476,-0.00811189)(3.28712,-0.00811441)(3.30948,-0.00811775)(3.33184,-0.00812163)(3.3542,-0.00812617)(3.37657,-0.00813198)(3.39893,-0.00814022)(3.42129,-0.0081529)(3.44365,-0.00817348)(3.46601,-0.00820797)(3.48837,-0.00826693)(3.51073,-0.00836755)(3.53309,-0.00853047)(3.55546,-0.00875971)(3.57782,-0.00901146)(3.60018,-0.00921461)(3.62254,-0.00934942)(3.6449,-0.00945643)(3.66726,-0.00959007)(3.68962,-0.00980368)(3.71199,-0.0101254)(3.73435,-0.0104817)(3.75671,-0.010728)(3.77907,-0.0108359)(3.80143,-0.0108738)(3.82379,-0.0108877)(3.84615,-0.0108919)(3.86852,-0.0108922)(3.89088,-0.0108963)(3.91324,-0.0109139)(3.9356,-0.010934)(3.95796,-0.0109465)(3.98032,-0.0109501)(4.00268,-0.0109502)(4.02504,-0.0109554)(4.04741,-0.0109665)(4.06977,-0.010976)(4.09213,-0.0109828)(4.11449,-0.0109876)(4.13685,-0.010991)(4.15921,-0.0109937)(4.18157,-0.0109958)(4.20394,-0.0109976)(4.2263,-0.0109993)(4.24866,-0.0110011)(4.27102,-0.0110033)(4.29338,-0.0110067)(4.31574,-0.0110119)(4.3381,-0.0110202)(4.36047,-0.0110331)(4.38283,-0.011053)(4.40519,-0.011084)(4.42755,-0.0111324)(4.44991,-0.0112065)(4.47227,-0.0113118)(4.49463,-0.0114412)(4.51699,-0.0115681)(4.53936,-0.0116607)(4.56172,-0.0117102)(4.58408,-0.011732)(4.60644,-0.0117437)(4.6288,-0.0117565)(4.65116,-0.0117789)(4.67352,-0.0118212)(4.69589,-0.0119004)(4.71825,-0.0120395)(4.74061,-0.0122437)(4.76297,-0.0124631)(4.78533,-0.0126217)(4.80769,-0.0127024)(4.83005,-0.0127415)(4.85242,-0.0127711)(4.87478,-0.012806)(4.89714,-0.0128505)(4.9195,-0.0129012)(4.94186,-0.0129493)(4.96422,-0.0129844)(4.98658,-0.0130019)};
\end{axis}
\end{tikzpicture}
}
\subfloat[Shallow water equations]{
\begin{tikzpicture}
\begin{axis}[
    width=.49\textwidth,
    xlabel={Time},
    ylabel={$\mathsf{S}(t)-\mathsf{S}(0)$}, 
    ymin= -.015,
    legend pos=south west, legend cell align=left, legend style={font=\tiny},	
    xmajorgrids=true, ymajorgrids=true, grid style=dashed, 
    legend entries={TDRK-2 (EC),RK-45 (EC),TDRK-2 (ES),RK-45 (ES)}
]
\pgfplotsset{
cycle list={{blue}, {red, dashed},{magenta, dashdotted},{black,dashed}}
}
\addplot+[semithick, mark options={solid, fill=markercolor, scale=1.5},line width=1.5pt]
coordinates{(0,0)(0.0223614,-8e-09)(0.0447227,2e-09)(0.0670841,-1.7e-08)(0.0894454,6e-09)(0.111807,-3.4e-08)(0.134168,2e-08)(0.15653,-5.2e-08)(0.178891,4e-08)(0.201252,-6.3e-08)(0.223614,6.8e-08)(0.245975,-5.7e-08)(0.268336,1.02e-07)(0.290698,-1.4e-08)(0.313059,1.16e-07)(0.33542,9e-08)(0.357782,8.7e-08)(0.380143,2.02e-07)(0.402504,1.21e-07)(0.424866,1.52e-07)(0.447227,2.11e-07)(0.469589,1.96e-07)(0.49195,1.89e-07)(0.514311,1.95e-07)(0.536673,1.95e-07)(0.559034,1.93e-07)(0.581395,1.88e-07)(0.603757,2.76e-07)(0.626118,7.53e-07)(0.648479,1.714e-06)(0.670841,2.46e-06)(0.693202,1.655e-06)(0.715564,1.546e-06)(0.737925,4.758e-06)(0.760286,2.172e-06)(0.782648,-1.826e-06)(0.805009,1.928e-06)(0.82737,-6.033e-06)(0.849732,-2.217e-06)(0.872093,-8.79e-06)(0.894454,-9.204e-06)(0.916816,-1.427e-05)(0.939177,-1.9142e-05)(0.961538,-2.2505e-05)(0.9839,-2.3567e-05)(1.00626,-2.9132e-05)(1.02862,-2.3406e-05)(1.05098,-3.2304e-05)(1.07335,-2.0934e-05)(1.09571,-3.515e-05)(1.11807,-2.1569e-05)(1.14043,-3.7229e-05)(1.16279,-2.3704e-05)(1.18515,-3.3152e-05)(1.20751,-2.8323e-05)(1.22987,-3.2935e-05)(1.25224,-5.5064e-05)(1.2746,-5.7733e-05)(1.29696,-8.6374e-05)(1.31932,-8.6695e-05)(1.34168,-0.000111086)(1.36404,-0.000119314)(1.3864,-0.000125441)(1.40877,-0.000144365)(1.43113,-0.000137519)(1.45349,-0.000159809)(1.47585,-0.000146474)(1.49821,-0.000164114)(1.52057,-0.000147002)(1.54293,-0.000158494)(1.5653,-0.000148021)(1.58766,-0.000158139)(1.61002,-0.000154195)(1.63238,-0.000156897)(1.65474,-0.00015816)(1.6771,-0.000156876)(1.69946,-0.000155602)(1.72182,-0.000151395)(1.74419,-0.000151001)(1.76655,-0.000149996)(1.78891,-0.000169143)(1.81127,-0.000174237)(1.83363,-0.000176236)(1.85599,-0.000177842)(1.87835,-0.000171036)(1.90072,-0.000171649)(1.92308,-0.000169155)(1.94544,-0.000178184)(1.9678,-0.00018441)(1.99016,-0.000160115)(2.01252,-0.000121795)(2.03488,-0.000158926)(2.05725,-0.000169263)(2.07961,-0.000185092)(2.10197,-0.000202044)(2.12433,-0.000195102)(2.14669,-0.000204682)(2.16905,-0.000213601)(2.19141,-0.000214263)(2.21377,-0.00025125)(2.23614,-0.000245288)(2.2585,-0.000253612)(2.28086,-0.000253093)(2.30322,-0.000242775)(2.32558,-0.000254005)(2.34794,-0.000230563)(2.3703,-0.000244484)(2.39267,-0.000247068)(2.41503,-0.000249152)(2.43739,-0.000239164)(2.45975,-0.000253941)(2.48211,-0.0002379)(2.50447,-0.000245776)(2.52683,-0.000236692)(2.54919,-0.000233701)(2.57156,-0.000243972)(2.59392,-0.000243405)(2.61628,-0.000274677)(2.63864,-0.000271022)(2.661,-0.000292022)(2.68336,-0.000293571)(2.70572,-0.000297108)(2.72809,-0.000327113)(2.75045,-0.000317273)(2.77281,-0.000324802)(2.79517,-0.000319957)(2.81753,-0.000326761)(2.83989,-0.000316469)(2.86225,-0.000305064)(2.88462,-0.000318659)(2.90698,-0.000324667)(2.92934,-0.000312555)(2.9517,-0.00029165)(2.97406,-0.000296045)(2.99642,-0.000288633)(3.01878,-0.0003017)(3.04114,-0.000276123)(3.06351,-0.000276935)(3.08587,-0.000276925)(3.10823,-0.000288634)(3.13059,-0.000278875)(3.15295,-0.000277779)(3.17531,-0.000285056)(3.19767,-0.000270422)(3.22004,-0.000252195)(3.2424,-0.000253577)(3.26476,-0.000243523)(3.28712,-0.000261526)(3.30948,-0.000283864)(3.33184,-0.000216907)(3.3542,-0.000265819)(3.37657,-0.00025638)(3.39893,-0.000256632)(3.42129,-0.000264224)(3.44365,-0.000268724)(3.46601,-0.000318632)(3.48837,-0.000292355)(3.51073,-0.000312855)(3.53309,-0.000342756)(3.55546,-0.000322336)(3.57782,-0.000305784)(3.60018,-0.000334086)(3.62254,-0.000361118)(3.6449,-0.000378564)(3.66726,-0.000361957)(3.68962,-0.00042614)(3.71199,-0.000462584)(3.73435,-0.000449318)(3.75671,-0.000427551)(3.77907,-0.000433221)(3.80143,-0.000468489)(3.82379,-0.000480974)(3.84615,-0.000434238)(3.86852,-0.000471146)(3.89088,-0.000438868)(3.91324,-0.00040141)(3.9356,-0.000426577)(3.95796,-0.000434876)(3.98032,-0.000462101)(4.00268,-0.000404734)(4.02504,-0.00037378)(4.04741,-0.000430047)(4.06977,-0.000446577)(4.09213,-0.00045167)(4.11449,-0.000467654)(4.13685,-0.000485521)(4.15921,-0.000565913)(4.18157,-0.000546586)(4.20394,-0.000484243)(4.2263,-0.000473885)(4.24866,-0.000498455)(4.27102,-0.000567464)(4.29338,-0.000547559)(4.31574,-0.000550901)(4.3381,-0.00054449)(4.36047,-0.000497286)(4.38283,-0.000533809)(4.40519,-0.000558904)(4.42755,-0.000573058)(4.44991,-0.000562084)(4.47227,-0.000538326)(4.49463,-0.000570921)(4.51699,-0.000582618)(4.53936,-0.000583566)(4.56172,-0.000578355)(4.58408,-0.000487312)(4.60644,-0.000559696)(4.6288,-0.000572919)(4.65116,-0.000585925)(4.67352,-0.000572761)(4.69589,-0.000494538)(4.71825,-0.000558576)(4.74061,-0.00054772)(4.76297,-0.000596692)(4.78533,-0.000606225)(4.80769,-0.000625103)(4.83005,-0.000640866)(4.85242,-0.000650975)(4.87478,-0.000659941)(4.89714,-0.000677729)(4.9195,-0.000598236)(4.94186,-0.000601043)(4.96422,-0.000552798)(4.98658,-0.000570677)};
\addplot+[semithick, mark options={solid, fill=markercolor, scale=1.5},line width=1.5pt]
coordinates{(0,0)(0.0223614,0)(0.0447227,0)(0.0670841,0)(0.0894454,0)(0.111807,0)(0.134168,0)(0.15653,0)(0.178891,0)(0.201252,0)(0.223614,0)(0.245975,0)(0.268336,0)(0.290698,0)(0.313059,0)(0.33542,0)(0.357782,0)(0.380143,0)(0.402504,0)(0.424866,-1e-09)(0.447227,-1e-09)(0.469589,-1e-09)(0.49195,-1e-09)(0.514311,-2e-09)(0.536673,-2e-09)(0.559034,-3e-09)(0.581395,-4e-09)(0.603757,-6e-09)(0.626118,-8e-09)(0.648479,-1.1e-08)(0.670841,-1.4e-08)(0.693202,-1.6e-08)(0.715564,-1.8e-08)(0.737925,-2.1e-08)(0.760286,-2.3e-08)(0.782648,-2.6e-08)(0.805009,-2.8e-08)(0.82737,-2.9e-08)(0.849732,-3.2e-08)(0.872093,-3.4e-08)(0.894454,-3.6e-08)(0.916816,-3.8e-08)(0.939177,-4.1e-08)(0.961538,-4.3e-08)(0.9839,-4.5e-08)(1.00626,-4.7e-08)(1.02862,-5e-08)(1.05098,-5.2e-08)(1.07335,-5.3e-08)(1.09571,-5.6e-08)(1.11807,-5.8e-08)(1.14043,-6e-08)(1.16279,-6.3e-08)(1.18515,-6.5e-08)(1.20751,-6.8e-08)(1.22987,-7e-08)(1.25224,-7.3e-08)(1.2746,-7.5e-08)(1.29696,-7.8e-08)(1.31932,-8e-08)(1.34168,-8.3e-08)(1.36404,-8.6e-08)(1.3864,-8.9e-08)(1.40877,-9.2e-08)(1.43113,-9.6e-08)(1.45349,-1.03e-07)(1.47585,-1.09e-07)(1.49821,-1.16e-07)(1.52057,-1.23e-07)(1.54293,-1.31e-07)(1.5653,-1.38e-07)(1.58766,-1.44e-07)(1.61002,-1.5e-07)(1.63238,-1.55e-07)(1.65474,-1.59e-07)(1.6771,-1.64e-07)(1.69946,-1.69e-07)(1.72182,-1.74e-07)(1.74419,-1.8e-07)(1.76655,-1.85e-07)(1.78891,-1.91e-07)(1.81127,-2e-07)(1.83363,-2.07e-07)(1.85599,-2.17e-07)(1.87835,-2.35e-07)(1.90072,-2.81e-07)(1.92308,-3.36e-07)(1.94544,-3.8e-07)(1.9678,-4.24e-07)(1.99016,-4.84e-07)(2.01252,-5.88e-07)(2.03488,-6.51e-07)(2.05725,-7.28e-07)(2.07961,-7.69e-07)(2.10197,-8.38e-07)(2.12433,-9.02e-07)(2.14669,-9.39e-07)(2.16905,-1.007e-06)(2.19141,-1.075e-06)(2.21377,-1.127e-06)(2.23614,-1.198e-06)(2.2585,-1.271e-06)(2.28086,-1.357e-06)(2.30322,-1.445e-06)(2.32558,-1.557e-06)(2.34794,-1.624e-06)(2.3703,-1.682e-06)(2.39267,-1.753e-06)(2.41503,-1.824e-06)(2.43739,-1.863e-06)(2.45975,-1.934e-06)(2.48211,-1.986e-06)(2.50447,-2.048e-06)(2.52683,-2.127e-06)(2.54919,-2.198e-06)(2.57156,-2.302e-06)(2.59392,-2.417e-06)(2.61628,-2.528e-06)(2.63864,-2.62e-06)(2.661,-2.713e-06)(2.68336,-2.837e-06)(2.70572,-2.918e-06)(2.72809,-2.978e-06)(2.75045,-3.184e-06)(2.77281,-3.413e-06)(2.79517,-3.61e-06)(2.81753,-3.8e-06)(2.83989,-4.076e-06)(2.86225,-4.294e-06)(2.88462,-4.515e-06)(2.90698,-4.659e-06)(2.92934,-4.802e-06)(2.9517,-4.891e-06)(2.97406,-5.005e-06)(2.99642,-5.1e-06)(3.01878,-5.182e-06)(3.04114,-5.307e-06)(3.06351,-5.355e-06)(3.08587,-5.428e-06)(3.10823,-5.524e-06)(3.13059,-5.615e-06)(3.15295,-5.671e-06)(3.17531,-5.732e-06)(3.19767,-5.808e-06)(3.22004,-5.906e-06)(3.2424,-6.007e-06)(3.26476,-6.1e-06)(3.28712,-6.161e-06)(3.30948,-6.321e-06)(3.33184,-6.64e-06)(3.3542,-6.899e-06)(3.37657,-7.171e-06)(3.39893,-7.336e-06)(3.42129,-7.516e-06)(3.44365,-7.686e-06)(3.46601,-7.804e-06)(3.48837,-7.933e-06)(3.51073,-8.032e-06)(3.53309,-8.232e-06)(3.55546,-8.458e-06)(3.57782,-8.635e-06)(3.60018,-8.833e-06)(3.62254,-9.031e-06)(3.6449,-9.225e-06)(3.66726,-9.392e-06)(3.68962,-9.621e-06)(3.71199,-9.894e-06)(3.73435,-1.0339e-05)(3.75671,-1.0642e-05)(3.77907,-1.0952e-05)(3.80143,-1.1019e-05)(3.82379,-1.1277e-05)(3.84615,-1.1522e-05)(3.86852,-1.1738e-05)(3.89088,-1.2122e-05)(3.91324,-1.2628e-05)(3.9356,-1.3079e-05)(3.95796,-1.3318e-05)(3.98032,-1.344e-05)(4.00268,-1.3641e-05)(4.02504,-1.3711e-05)(4.04741,-1.3904e-05)(4.06977,-1.4261e-05)(4.09213,-1.4532e-05)(4.11449,-1.4826e-05)(4.13685,-1.4972e-05)(4.15921,-1.5211e-05)(4.18157,-1.5556e-05)(4.20394,-1.5964e-05)(4.2263,-1.6395e-05)(4.24866,-1.6923e-05)(4.27102,-1.7291e-05)(4.29338,-1.7603e-05)(4.31574,-1.7719e-05)(4.3381,-1.7932e-05)(4.36047,-1.8075e-05)(4.38283,-1.8212e-05)(4.40519,-1.8555e-05)(4.42755,-1.885e-05)(4.44991,-1.915e-05)(4.47227,-1.9262e-05)(4.49463,-1.9332e-05)(4.51699,-1.95e-05)(4.53936,-1.9672e-05)(4.56172,-1.9904e-05)(4.58408,-2.0371e-05)(4.60644,-2.0888e-05)(4.6288,-2.1297e-05)(4.65116,-2.1718e-05)(4.67352,-2.2316e-05)(4.69589,-2.3218e-05)(4.71825,-2.3868e-05)(4.74061,-2.479e-05)(4.76297,-2.5232e-05)(4.78533,-2.6145e-05)(4.80769,-2.7087e-05)(4.83005,-2.7516e-05)(4.85242,-2.8621e-05)(4.87478,-2.9129e-05)(4.89714,-2.9955e-05)(4.9195,-3.0167e-05)(4.94186,-3.0453e-05)(4.96422,-3.0736e-05)(4.98658,-3.089e-05)};
\addplot+[semithick, mark options={solid, fill=markercolor, scale=1.5},line width=1.5pt]
coordinates{(0,0)(0.0223614,-1.501e-06)(0.0447227,-1.492e-06)(0.0670841,-1.512e-06)(0.0894454,-1.492e-06)(0.111807,-1.541e-06)(0.134168,-1.493e-06)(0.15653,-1.57e-06)(0.178891,-1.491e-06)(0.201252,-1.593e-06)(0.223614,-1.483e-06)(0.245975,-1.603e-06)(0.268336,-1.476e-06)(0.290698,-1.58e-06)(0.313059,-1.493e-06)(0.33542,-1.497e-06)(0.357782,-1.548e-06)(0.380143,-1.407e-06)(0.402504,-1.51e-06)(0.424866,-1.483e-06)(0.447227,-1.414e-06)(0.469589,-1.435e-06)(0.49195,-1.451e-06)(0.514311,-1.45e-06)(0.536673,-1.449e-06)(0.559034,-1.452e-06)(0.581395,-1.471e-06)(0.603757,-2.082e-06)(0.626118,-7.24e-06)(0.648479,-2.1429e-05)(0.670841,-4.1874e-05)(0.693202,-7.5441e-05)(0.715564,-0.00013616)(0.737925,-0.000183908)(0.760286,-0.000213063)(0.782648,-0.000291545)(0.805009,-0.000328774)(0.82737,-0.000392884)(0.849732,-0.000469711)(0.872093,-0.000503877)(0.894454,-0.000584673)(0.916816,-0.000602869)(0.939177,-0.000674569)(0.961538,-0.000693318)(0.9839,-0.000765078)(1.00626,-0.000796677)(1.02862,-0.000867978)(1.05098,-0.000919367)(1.07335,-0.000972257)(1.09571,-0.00102889)(1.11807,-0.00104448)(1.14043,-0.00109937)(1.16279,-0.00109125)(1.18515,-0.00113466)(1.20751,-0.00114192)(1.22987,-0.00117337)(1.25224,-0.00124019)(1.2746,-0.00125484)(1.29696,-0.00133974)(1.31932,-0.00135206)(1.34168,-0.00140303)(1.36404,-0.00141955)(1.3864,-0.00143357)(1.40877,-0.00144649)(1.43113,-0.00144389)(1.45349,-0.00145043)(1.47585,-0.00144102)(1.49821,-0.00144761)(1.52057,-0.00143612)(1.54293,-0.00144488)(1.5653,-0.00143432)(1.58766,-0.00144341)(1.61002,-0.00143545)(1.63238,-0.00143975)(1.65474,-0.00144261)(1.6771,-0.00144029)(1.69946,-0.00145464)(1.72182,-0.00145101)(1.74419,-0.00145515)(1.76655,-0.00146115)(1.78891,-0.00146845)(1.81127,-0.0014697)(1.83363,-0.00146979)(1.85599,-0.00148226)(1.87835,-0.00148716)(1.90072,-0.00154187)(1.92308,-0.00170577)(1.94544,-0.00203219)(1.9678,-0.002297)(1.99016,-0.0025478)(2.01252,-0.00291738)(2.03488,-0.00326841)(2.05725,-0.00392507)(2.07961,-0.00419599)(2.10197,-0.00459885)(2.12433,-0.00492182)(2.14669,-0.00508815)(2.16905,-0.00532457)(2.19141,-0.00537853)(2.21377,-0.00549491)(2.23614,-0.00553089)(2.2585,-0.00564566)(2.28086,-0.00569223)(2.30322,-0.00583025)(2.32558,-0.00585823)(2.34794,-0.00597893)(2.3703,-0.00600043)(2.39267,-0.00607114)(2.41503,-0.00608117)(2.43739,-0.00610465)(2.45975,-0.00612003)(2.48211,-0.00611426)(2.50447,-0.0061462)(2.52683,-0.00614774)(2.54919,-0.00618505)(2.57156,-0.00623471)(2.59392,-0.00626422)(2.61628,-0.00635589)(2.63864,-0.0063722)(2.661,-0.00645853)(2.68336,-0.00647227)(2.70572,-0.00651735)(2.72809,-0.00652694)(2.75045,-0.00654297)(2.77281,-0.00654684)(2.79517,-0.00654983)(2.81753,-0.00655171)(2.83989,-0.00655124)(2.86225,-0.00655374)(2.88462,-0.00655276)(2.90698,-0.0065555)(2.92934,-0.00655546)(2.9517,-0.0065567)(2.97406,-0.00655888)(2.99642,-0.00655711)(3.01878,-0.00656117)(3.04114,-0.00656106)(3.06351,-0.00656422)(3.08587,-0.00656497)(3.10823,-0.00656458)(3.13059,-0.00656609)(3.15295,-0.00656774)(3.17531,-0.00657456)(3.19767,-0.00657827)(3.22004,-0.00658299)(3.2424,-0.00660141)(3.26476,-0.00668435)(3.28712,-0.00676185)(3.30948,-0.00680482)(3.33184,-0.00685311)(3.3542,-0.00694306)(3.37657,-0.00711731)(3.39893,-0.00718522)(3.42129,-0.0074407)(3.44365,-0.00754597)(3.46601,-0.0078209)(3.48837,-0.00795125)(3.51073,-0.00816502)(3.53309,-0.00824821)(3.55546,-0.0084143)(3.57782,-0.00847104)(3.60018,-0.00864086)(3.62254,-0.00868927)(3.6449,-0.00880746)(3.66726,-0.00883093)(3.68962,-0.0088981)(3.71199,-0.00890345)(3.73435,-0.0089308)(3.75671,-0.00893524)(3.77907,-0.00894396)(3.80143,-0.0089607)(3.82379,-0.00896453)(3.84615,-0.00898898)(3.86852,-0.00899178)(3.89088,-0.00902498)(3.91324,-0.00903543)(3.9356,-0.00905689)(3.95796,-0.00907194)(3.98032,-0.00908387)(4.00268,-0.00909807)(4.02504,-0.00910247)(4.04741,-0.00910643)(4.06977,-0.00910885)(4.09213,-0.0091101)(4.11449,-0.00911756)(4.13685,-0.00911936)(4.15921,-0.00913139)(4.18157,-0.00913261)(4.20394,-0.00914881)(4.2263,-0.00915143)(4.24866,-0.00916626)(4.27102,-0.0091709)(4.29338,-0.00917985)(4.31574,-0.00918889)(4.3381,-0.00919325)(4.36047,-0.00920346)(4.38283,-0.00920557)(4.40519,-0.0092157)(4.42755,-0.0092184)(4.44991,-0.00922455)(4.47227,-0.00922956)(4.49463,-0.00923368)(4.51699,-0.00923985)(4.53936,-0.00924184)(4.56172,-0.00924747)(4.58408,-0.00924777)(4.60644,-0.00925573)(4.6288,-0.00927067)(4.65116,-0.00929596)(4.67352,-0.00931453)(4.69589,-0.00933221)(4.71825,-0.00934949)(4.74061,-0.00937827)(4.76297,-0.00943242)(4.78533,-0.00945936)(4.80769,-0.00955077)(4.83005,-0.00957644)(4.85242,-0.00967105)(4.87478,-0.00968897)(4.89714,-0.00976906)(4.9195,-0.0097836)(4.94186,-0.00986214)(4.96422,-0.00987648)(4.98658,-0.00993778)};
\addplot+[semithick, mark options={solid, fill=markercolor, scale=1.5},line width=1.5pt]
coordinates{(0,0)(0.0223614,0)(0.0447227,0)(0.0670841,-3e-09)(0.0894454,-5e-09)(0.111807,-1.8e-08)(0.134168,-2.2e-08)(0.15653,-3.7e-08)(0.178891,-4.5e-08)(0.201252,-5.8e-08)(0.223614,-7.2e-08)(0.245975,-8.3e-08)(0.268336,-1.01e-07)(0.290698,-1.11e-07)(0.313059,-1.27e-07)(0.33542,-1.36e-07)(0.357782,-1.47e-07)(0.380143,-1.56e-07)(0.402504,-1.61e-07)(0.424866,-1.69e-07)(0.447227,-1.77e-07)(0.469589,-1.83e-07)(0.49195,-1.92e-07)(0.514311,-1.98e-07)(0.536673,-1.98e-07)(0.559034,-2e-07)(0.581395,-2.15e-07)(0.603757,-8.93e-07)(0.626118,-6.38e-06)(0.648479,-2.1011e-05)(0.670841,-4.1804e-05)(0.693202,-7.4459e-05)(0.715564,-0.000133992)(0.737925,-0.000184665)(0.760286,-0.000213533)(0.782648,-0.000290121)(0.805009,-0.000334947)(0.82737,-0.000394335)(0.849732,-0.000475479)(0.872093,-0.00050793)(0.894454,-0.00058887)(0.916816,-0.000609115)(0.939177,-0.000682087)(0.961538,-0.000704675)(0.9839,-0.000779898)(1.00626,-0.000813505)(1.02862,-0.000889959)(1.05098,-0.000938878)(1.07335,-0.000999118)(1.09571,-0.00104954)(1.11807,-0.00107701)(1.14043,-0.00112488)(1.16279,-0.00113111)(1.18515,-0.0011707)(1.20751,-0.0011878)(1.22987,-0.00121699)(1.25224,-0.00127076)(1.2746,-0.00128268)(1.29696,-0.00134862)(1.31932,-0.00136418)(1.34168,-0.00139919)(1.36404,-0.0014194)(1.3864,-0.00142712)(1.40877,-0.00143979)(1.43113,-0.00144087)(1.45349,-0.00144276)(1.47585,-0.00144302)(1.49821,-0.00144306)(1.52057,-0.00144312)(1.54293,-0.00144527)(1.5653,-0.00144624)(1.58766,-0.00144956)(1.61002,-0.00145097)(1.63238,-0.00145267)(1.65474,-0.00145729)(1.6771,-0.00145865)(1.69946,-0.00146205)(1.72182,-0.00146303)(1.74419,-0.00146883)(1.76655,-0.00147568)(1.78891,-0.0014842)(1.81127,-0.00148562)(1.83363,-0.00148596)(1.85599,-0.00149905)(1.87835,-0.00150321)(1.90072,-0.00156195)(1.92308,-0.00171915)(1.94544,-0.00204483)(1.9678,-0.00230588)(1.99016,-0.00257678)(2.01252,-0.00298132)(2.03488,-0.00331174)(2.05725,-0.0039456)(2.07961,-0.00423198)(2.10197,-0.00462054)(2.12433,-0.00494758)(2.14669,-0.005108)(2.16905,-0.00533454)(2.19141,-0.00539251)(2.21377,-0.00550667)(2.23614,-0.00554989)(2.2585,-0.00566401)(2.28086,-0.00571638)(2.30322,-0.00584925)(2.32558,-0.0058825)(2.34794,-0.00599676)(2.3703,-0.00602221)(2.39267,-0.00608863)(2.41503,-0.00610101)(2.43739,-0.00612743)(2.45975,-0.00614212)(2.48211,-0.00614501)(2.50447,-0.00617395)(2.52683,-0.0061843)(2.54919,-0.0062148)(2.57156,-0.00626235)(2.59392,-0.00628681)(2.61628,-0.00637004)(2.63864,-0.00638473)(2.661,-0.00646162)(2.68336,-0.00647908)(2.70572,-0.00651849)(2.72809,-0.00653161)(2.75045,-0.00654587)(2.77281,-0.00655098)(2.79517,-0.00655531)(2.81753,-0.00655642)(2.83989,-0.00655867)(2.86225,-0.00655952)(2.88462,-0.00656139)(2.90698,-0.00656208)(2.92934,-0.00656392)(2.9517,-0.00656457)(2.97406,-0.00656665)(2.99642,-0.00656685)(3.01878,-0.00656742)(3.04114,-0.00656846)(3.06351,-0.006571)(3.08587,-0.00657272)(3.10823,-0.00657397)(3.13059,-0.00657436)(3.15295,-0.00657501)(3.17531,-0.00658187)(3.19767,-0.00658543)(3.22004,-0.00659029)(3.2424,-0.00661016)(3.26476,-0.00669335)(3.28712,-0.00677013)(3.30948,-0.00681624)(3.33184,-0.00686486)(3.3542,-0.00695202)(3.37657,-0.00713071)(3.39893,-0.00719487)(3.42129,-0.00745076)(3.44365,-0.00756191)(3.46601,-0.00782767)(3.48837,-0.00796337)(3.51073,-0.00816842)(3.53309,-0.00825828)(3.55546,-0.00842167)(3.57782,-0.00848655)(3.60018,-0.00865259)(3.62254,-0.00870645)(3.6449,-0.00881793)(3.66726,-0.00884436)(3.68962,-0.00890632)(3.71199,-0.00891317)(3.73435,-0.00893858)(3.75671,-0.00894418)(3.77907,-0.00895351)(3.80143,-0.00897042)(3.82379,-0.00897606)(3.84615,-0.00899934)(3.86852,-0.00900415)(3.89088,-0.00903514)(3.91324,-0.00904501)(3.9356,-0.00906472)(3.95796,-0.00907762)(3.98032,-0.0090889)(4.00268,-0.0091006)(4.02504,-0.00910549)(4.04741,-0.00910817)(4.06977,-0.00911098)(4.09213,-0.00911235)(4.11449,-0.00911932)(4.13685,-0.00912184)(4.15921,-0.00913283)(4.18157,-0.00913519)(4.20394,-0.00915044)(4.2263,-0.00915424)(4.24866,-0.00916871)(4.27102,-0.00917461)(4.29338,-0.00918353)(4.31574,-0.00919328)(4.3381,-0.00919727)(4.36047,-0.00920711)(4.38283,-0.00920864)(4.40519,-0.00921796)(4.42755,-0.00922077)(4.44991,-0.00922713)(4.47227,-0.00923191)(4.49463,-0.00923572)(4.51699,-0.00924155)(4.53936,-0.00924357)(4.56172,-0.00924896)(4.58408,-0.00924942)(4.60644,-0.00925725)(4.6288,-0.0092719)(4.65116,-0.00929619)(4.67352,-0.00931433)(4.69589,-0.0093343)(4.71825,-0.00935257)(4.74061,-0.00937917)(4.76297,-0.00943585)(4.78533,-0.00946079)(4.80769,-0.00955233)(4.83005,-0.00957809)(4.85242,-0.00966961)(4.87478,-0.00968864)(4.89714,-0.00976727)(4.9195,-0.00978364)(4.94186,-0.00986232)(4.96422,-0.00987798)(4.98658,-0.00993924)};
\end{axis}
\end{tikzpicture}
}
\caption{Evolution of entropy over time for TDRK-2 and RK-45 schemes using both entropy conservative (EC) and entropy stable (ES) spectral collocation formulations of the Burgers' and shallow water equations. }
\label{fig:tdrkentropy}
\end{figure}
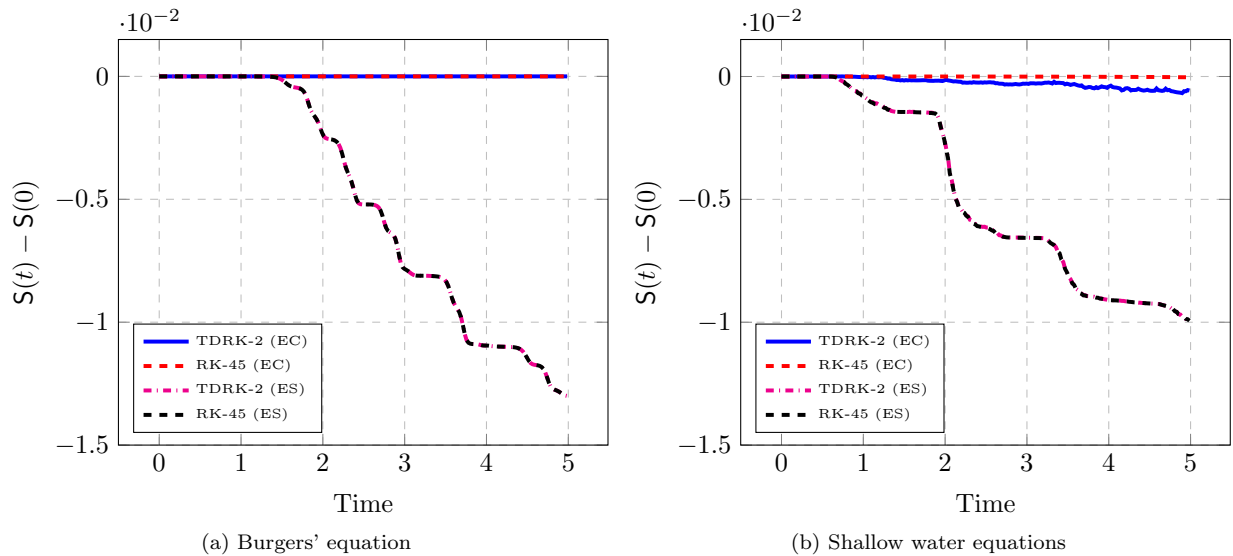
Finally, we plot the integrated entropy $\mathsf{S}(t) = \int_{\Omega} S(\bm{u})\diff{x}$ over time in Figure~\ref{fig:tdrkentropy}. For Burgers' equation, we do not observe significant differences in the entropy dissipation for TDRK-2 and RK-45 schemes. For the entropy conservative formulation of the shallow water equations, the TDRK-2 scheme produces slightly more entropy dissipation than RK-45; however, both two schemes produce similar entropy dissipation for the entropy stable formulation. 

\section{Time-implicit discretizations on triangular meshes}
\label{sec:implicit}

Jacobian matrices also appear in time-implicit discretizations of nonlinear ODEs.  Consider the implicit midpoint rule 
\[
\fnt{u}^{k+1} = \fnt{u}^k - \Delta t \fnt{r}\LRp{\frac{\fnt{u}^{k+1} + \fnt{u}^k}{2}}.
\]
This can be rewritten in the following form where $\fnt{u}^{k+1/2} = \frac{\fnt{u}^{k+1}+\fnt{u}^k}{2}$
\eq{
\fnt{u}^{k+1/2} &= \fnt{u}^k - \frac{\Delta t}{2}\fnt{r}\LRp{\fnt{u}^{k+1/2}}\\
\fnt{u}^{k+1} &= 2\fnt{u}^{k+1/2} - \fnt{u}^k.
}
Solving for $\fnt{u}^{k+1/2}$ is a nonlinear equation and can be done via Newton's method
\eq{
%&\fnt{u}^{k+1/2} + \frac{\Delta t}{2}\fnt{f}\LRp{\fnt{u}^{k+1/2}} - \fnt{u}^k = \fnt{0}\\
\fnt{u}^{k+1/2,\ell+1} = \fnt{u}^{k+1/2,\ell} - \LRp{\fnt{I} + \frac{\Delta t}{2}\LRl{\pd{\fnt{r}}{\fnt{u}}}_{\fnt{u}^{k+1/2,\ell}}}^{-1}\LRp{\fnt{u}^{k+1/2,\ell} + \frac{\Delta t}{2}\fnt{r}\LRp{\fnt{u}^{k+1/2,\ell}} - \fnt{u}^k}.
}
\bnote{where $\ell$ denotes the Newton iteration index.} 

\bnote{All linear systems are solved using Julia's sparse direct solver.} 
We utilize a relative tolerance of $1e-11$ for the Newton iteration, and determine the time-step $\Delta t$ using the following estimate 
\[
\Delta t = {\rm CFL} \times \frac{ h_{\min}} { C_N}, \qquad C_N = \frac{(N+1)(N+2)}{2},
\]
where ${\rm CFL}$ is the CFL constant, $h_{\min}$ is the size of the smallest element in the mesh, and $C_N$ is the $N$-dependent trace constant for a degree $N$ polynomial space on the reference triangle \cite{warburton2003constants}.\footnote{Trace constants $C_N$ for other element types are derived in \cite{chan2015gpu}.}

\subsubsection{2D Burgers' equation}

We consider energy conservative and energy stable discretizations of 2D Burgers' equation
\[
\pd{u}{t} + \frac{1}{2}\pd{u^2}{x} = 0
\]
with periodic boundary conditions on the domain $[-1,1]^2$. For the initial condition $u(\bm{x},0) = -\sin(\pi x)$, the solution forms a shock around $T = 1/2$. 

We discretize the Burgers' equation using an energy conservative (or stable) scheme in space \cite{chan2017discretely, chan2019skew} and an implicit midpoint discretization in time.  The spatial discretization utilizes a degree $N$ polynomial space, degree $2N$ volume quadrature, and an $(N+1)$-point Gauss quadrature for faces.  An energy stable scheme is constructed by adding a local Lax-Friedrichs penalization term, $ - \frac{\lambda}{2}\jump{u}$, to the energy conservative flux contribution, where $\lambda = \max\LRp{\LRb{u^+},\LRb{u}}$ is the maximum wavespeed at an interface.  We utilize both uniform and ``squeezed'' triangular meshes with very small elements (see Figure~\ref{fig:meshburg}). \bnote{This mesh is constructed by taking the $x$-coordinate $x_i$ of vertices in a uniform triangular mesh (constructed by bisecting each element in a uniform mesh quadrilateral mesh) and transforming them via $x_i - .3\sin(\pi x_i)$ to produce a new mesh.}

\begin{figure}
\centering
\subfloat[Uniform mesh]{\includegraphics[width=.4\textwidth]{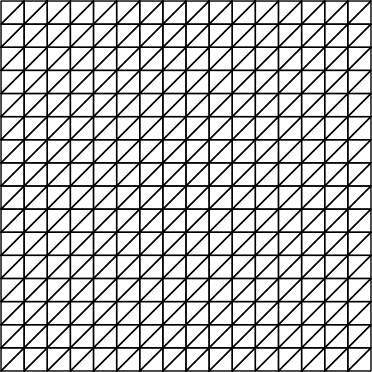}}
\hspace{2em}
\subfloat[Solution on a uniform mesh]{\includegraphics[width=.4\textwidth]{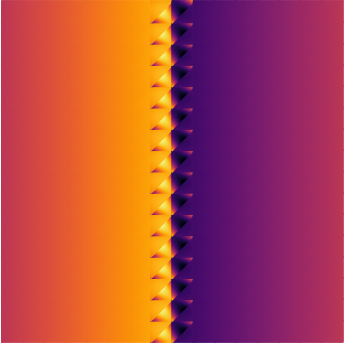}}\\
\subfloat[Anisotropic mesh]{\includegraphics[width=.4\textwidth]{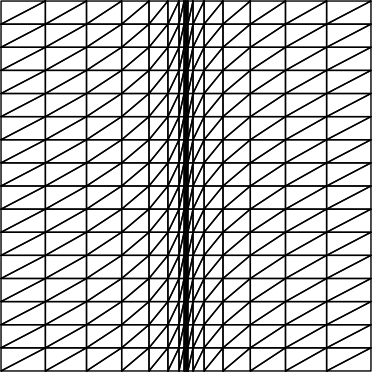}}
\hspace{2em}
\subfloat[Solution on an anisotropic mesh]{\includegraphics[width=.4\textwidth]{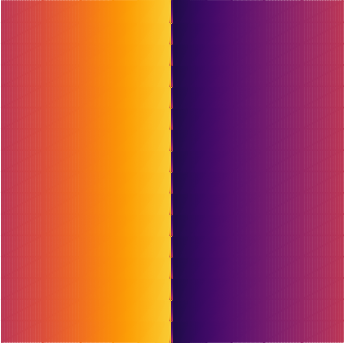}}
\caption{$N=2$ solutions of Burgers' equation at $T=1$ on uniform and anisotropic ``squeezed'' $16\times 16$ meshes.}
\label{fig:meshburg}
\end{figure}
Since the implicit midpoint rule is a symplectic integrator, we expect energy to be conserved up to machine precision for an energy conservative scheme. We set the initial condition randomly, remove the Lax-Friedrichs penalization, and run until time $T=1$ using a CFL of 10 on both uniform and squeezed  $8\times 8$ meshes with $N=2$.  For the uniform mesh, the total change in energy was $-2.665e-15$. The squeezed mesh behaved similarly, with a total change in energy of $3.553e-15$. 

Next, we add local Lax-Friedrichs dissipation and run with the initial condition $-\sin(\pi x)$ until time $T=1$ using a CFL of 250. Figure~\ref{fig:meshburg} shows solutions for both cases. In each case, oscillations appear in a one-element vicinity around the shock. For both meshes, the Newton iteration converges in between 4 and 7 iterations. We note that for an entropy conservative scheme with a randomly generated initial condition, increasing the CFL further resulted in non-convergence of the Newton iteration. However, either switching to the initial condition $u(x,y,0) = -\sin(\pi x)$ or adding local Lax-Friedrichs dissipation avoids stalling of the Newton iteration.

\subsubsection{2D compressible Euler equations}
\label{sec:euler}
Finally, we consider a time-implicit discretization of the 2D compressible Euler equations
\begin{align*}
\pd{\rho}{t} + \pd{\LRp{\rho u}}{x_1} + \pd{\LRp{\rho v}}{x_2} &= 0,\\
\pd{\rho u}{t} + \pd{\LRp{\rho u^2 + p }}{x_1} + \pd{\LRp{\rho uv}}{x_2} &= 0,\nonumber\\
\pd{\rho v}{t} + \pd{\LRp{\rho uv}}{x_1} + \pd{\LRp{\rho v^2 + p }}{x_2} &= 0,\nonumber\\
\pd{E}{t} + \pd{\LRp{u(E+p)}}{x_1} + \pd{\LRp{v(E+p)}}{x_2}&= 0,\nonumber
\end{align*}
Here, $\gamma = 1.4$, and $p = (\gamma-1)\rho e$ is the pressure, where $\rho e = E - \frac{1}{2}\rho (u^2+v^2)$ is the specific internal energy. We construct a scheme which is stable with respect to the unique entropy for the compressible Navier-Stokes equations \cite{hughes1986new}
\begin{equation*}
S(\bm{u}) = -\frac{\rho s}{\gamma-1}, \qquad \bm{u} = \LRs{\rho,\rho u,\rho v, E}^T,
\label{eq:entropy2d}
\end{equation*}
where $s = \log\LRp{\frac{p}{\rho^\gamma}}$ denotes the specific entropy. 
Mappings between conservative variables $\bm{u}$ and entropy variables $\bm{v} = \LRc{v_1,v_2,v_3,v_4}$ in two dimensions are given by
\begin{align*}
v_1 &= \frac{\rho e (\gamma + 1 - s) - E}{\rho e}, \qquad v_2 = \frac{\rho u}{\rho e}, \qquad v_3 = \frac{\rho v}{\rho e}, \qquad v_4 = -\frac{\rho}{\rho e}\\
\rho &= -(\rho e) v_4, \qquad \rho u = (\rho e) v_2, \qquad \rho v = (\rho e) v_3, \qquad E = (\rho e)\LRp{1 - \frac{{v_2^2+v_3^2}}{2 v_4}},
\end{align*}
where $\rho e$ and $s$ can be expressed in terms of the entropy variables as
\begin{equation*}
\rho e = \LRp{\frac{(\gamma-1)}{\LRp{-v_4}^{\gamma}}}^{1/(\gamma-1)}e^{\frac{-s}{\gamma-1}}, \qquad s = \gamma - v_1 + \frac{{v_2^2+v_3^2}}{2v_4}.
\end{equation*}
We utilize the entropy conservative and kinetic energy preserving finite volume fluxes derived in \cite{chandrashekar2013kinetic}, and apply entropy dissipation by adding a local Lax-Friedrichs penalization term, $-\frac{\lambda}{2}\jump{\bm{u}}$ \cite{chen2017entropy, chan2018discretely}. We compute at each point on an interface the local wavespeed $a = \LRb{\bm{u}\cdot\bm{n}} + c$, where $c = \sqrt{\gamma\rho/p}$ is the sound speed and $\bm{u}\cdot\bm{n}$ is the normal velocity. The local Lax-Friedrichs parameter is then computed via $\lambda = \sqrt{\frac{(a^+)^2+a^2}{2}}$, where $a^+,a$ are computed using the interior and exterior solution states, respectively. 

We employ an entropy stable modal DG formulation from \cite{chan2018discretely} on triangles using total degree $N$ polynomials. The surface quadrature is constructed using 1D $(N+1)$ point Gauss quadrature rules on each face and we use a volume quadrature \cite{xiao2010quadrature} which is exact for degree $2N$ polynomials. Since this is a non-collocated formulation, we need the change of variables matrices $\pd{\bm{u}}{\bm{v}}, \pd{\bm{v}}{\bm{u}}$ to evaluate (\ref{eq:changeofvars}). These matrices can be computed using automatic differentiation or using the explicit formulas \cite{barth1999numerical}
\eq{
\pd{\bm{u}}{\bm{v}} = \bmat{
\rho    	 & \rho u            & \rho v   		&   E\\
\rho u  	& \rho u^2 + p	& \rho uv 	 	&   \rho uH \\
\rho v	& \rho uv	  	& \rho v^2+p 	&   \rho vH \\
E 	  	& \rho u H   	&\rho v H      	& \rho H^2-c^2\frac{p}{\gamma-1}
}, \\
\pd{\bm{v}}{\bm{u}} = -\frac{1}{\rho e v_4}\bmat{
\gamma+k^2	& kv_2     		& kv_3   		&   (k+1)v_4\\
k v_2  		& v_2^2-v_4	& v_2v_3 	 	&   v_2v_4 \\
k v_3		& v_2v_3	  	& v_3^2-v_4 	&   v_3v_4 \\
(k+1)v_4 		& v_2v_4  	& v_3v_4      	&  v_4^2
}
%γ+k^2     k*V[2]          k*V[3]         V[4]*(k+1);
% k*V[2]     V[2]^2-V[4]     V[2]*V[3]     V[2]*V[4];
% k*V[3]     V[2]*V[3]      V[3]^2-V[4]    V[3]*V[4]
% V[4]*(k+1) V[2]*V[4]       V[3]*V[4]      V[4]^2
%-dVdU/(rhoe*V[4])                    
}
where $c$ is the sound speed, $H = c^2/(\gamma-1) + \frac{1}{2}(u^2+v^2)$ is the enthalpy, and $k = \frac{1}{2}(v_2^2+v_3^2)/v_4$.

There exist several choices for entropy conservative fluxes \cite{ismail2009affordable, ranocha2018comparison, chandrashekar2013kinetic}.  We utilize the
the entropy conservative numerical fluxes given by Chandrashekar in \cite{chandrashekar2013kinetic}
\begin{align*}
&f^1_{1,S}(\bm{u}_L,\bm{u}_R) = \avg{\rho}^{\log} \avg{u},& &f^1_{2,S}(\bm{u}_L,\bm{u}_R) = \avg{\rho}^{\log} \avg{v},&\\
&f^2_{1,S}(\bm{u}_L,\bm{u}_R) = f^1_{1,S} \avg{u} + p_{\rm avg},&  &f^2_{2,S}(\bm{u}_L,\bm{u}_R) = f^1_{2,S} \avg{u},&\nonumber\\
&f^3_{1,S}(\bm{u}_L,\bm{u}_R) = f^2_{2,S},& &f^3_{2,S}(\bm{u}_L,\bm{u}_R) = f^1_{2,S} \avg{v} + p_{\rm avg},&\nonumber\\
&f^4_{1,S}(\bm{u}_L,\bm{u}_R) = \LRp{E_{\rm avg} + p_{\rm avg}}\avg{u},& &f^4_{2,S}(\bm{u}_L,\bm{u}_R) = \LRp{E_{\rm avg} + p_{\rm avg} }\avg{v},& \nonumber
\end{align*}
where the quantities $p_{\rm avg}, E_{\rm avg},  \nor{\bm{u}}^2_{\rm avg}$ are defined as
\begin{gather*}
p_{\rm avg} = \frac{\avg{\rho}}{2\avg{\beta}}, \qquad E_{\rm avg} = \frac{\avg{\rho}^{\log}}{2\avg{\beta}^{\log}\LRp{\gamma -1}}   + \frac{\nor{\bm{u}}^2_{\rm avg}}{2}, \qquad  \beta = \frac{\rho}{2p},\\
 \nor{\bm{u}}^2_{\rm avg} = 2(\avg{u}^2 + \avg{v}^2) - \LRp{\avg{u^2} +\avg{v^2}} = u^+u + v^+v \nonumber,
\end{gather*}
where $\avg{u} = \frac{1}{2}\LRp{u^+ + u}$, where $u^+,u$ denotes the exterior and interior states across the interface of an element $D^k$.

\begin{figure}
\centering
\subfloat[$N=2$]{
\begin{tikzpicture}
\begin{axis}[
    width=.49\textwidth,
    xlabel={Time },
    ylabel={$\mathsf{S}(t)-\mathsf{S}(0)$}, 
    ymin=-1.2e-7, ymax=1.2e-7,
    legend pos=north east, legend cell align=left, legend style={font=\tiny},	
    xmajorgrids=true, ymajorgrids=true, grid style=dashed, 
    legend entries={$\text{CFL} = \frac{1}{4}$, $\text{CFL} = \frac{1}{8}$}
]
\pgfplotsset{
cycle list={{blue}, {red, dashed}}
}
\addplot+[semithick, mark options={solid, fill=markercolor, scale=1.5},line width=1.5pt]
coordinates{(0.0104112,0)(0.0624675,-5.19164e-09)(0.114524,3.61871e-09)(0.16658,1.41953e-08)(0.218636,1.698e-08)(0.270692,2.17401e-08)(0.322749,2.20701e-08)(0.374805,2.58682e-08)(0.426861,3.62348e-08)(0.478917,4.08486e-08)(0.530973,4.10812e-08)(0.58303,3.93001e-08)(0.635086,4.10676e-08)(0.687142,4.25268e-08)(0.739198,4.04793e-08)(0.791255,3.84745e-08)(0.843311,3.90362e-08)(0.895367,4.95432e-08)(0.947423,5.48887e-08)(0.999479,5.29454e-08)(1.05154,4.38309e-08)(1.10359,3.2088e-08)(1.15565,3.19724e-08)(1.2077,3.03918e-08)(1.25976,2.69956e-08)(1.31182,2.10818e-08)(1.36387,1.88147e-08)(1.41593,2.08275e-08)(1.46799,2.47401e-08)(1.52004,2.57103e-08)(1.5721,2.29511e-08)(1.62415,2.33854e-08)(1.67621,2.44561e-08)(1.72827,2.46443e-08)(1.78032,2.66939e-08)(1.83238,3.08554e-08)(1.88444,3.7393e-08)(1.93649,4.12304e-08)(1.98855,3.7691e-08)(2.0406,3.94884e-08)(2.09266,4.0872e-08)(2.14472,4.02386e-08)(2.19677,3.57267e-08)(2.24883,3.15917e-08)(2.30088,2.84344e-08)(2.35294,2.87772e-08)(2.405,2.91958e-08)(2.45705,2.25966e-08)(2.50911,1.61358e-08)(2.56117,1.34573e-08)(2.61322,1.41786e-08)(2.66528,1.19803e-08)(2.71733,8.1798e-09)(2.76939,-1.30194e-10)(2.82145,-4.11189e-09)(2.8735,-5.83913e-09)(2.92556,-9.35513e-09)(2.97762,-8.59194e-09)(3.02967,-1.54769e-08)(3.08173,-2.41233e-08)(3.13378,-2.92763e-08)(3.18584,-3.29543e-08)(3.2379,-3.84601e-08)(3.28995,-4.10895e-08)(3.34201,-4.48672e-08)(3.39407,-4.35315e-08)(3.44612,-3.95163e-08)(3.49818,-3.71962e-08)(3.55023,-4.00019e-08)(3.60229,-4.57234e-08)(3.65435,-4.19679e-08)(3.7064,-3.4054e-08)(3.75846,-2.95295e-08)(3.81052,-2.92591e-08)(3.86257,-2.87019e-08)(3.91463,-2.51387e-08)(3.96668,-2.27356e-08)(4.01874,-2.15263e-08)(4.0708,-2.06954e-08)(4.12285,-1.71694e-08)(4.17491,-1.78114e-08)(4.22697,-1.37899e-08)(4.27902,-1.07959e-08)(4.33108,-8.6262e-09)(4.38313,-1.04675e-08)(4.43519,-1.65911e-08)(4.48725,-2.19518e-08)(4.5393,-2.24303e-08)(4.59136,-2.37826e-08)(4.64341,-3.03371e-08)(4.69547,-3.60098e-08)(4.74753,-3.96228e-08)(4.79958,-4.1519e-08)(4.85164,-4.53207e-08)(4.9037,-5.35949e-08)(4.95575,-5.9411e-08)(5.00781,-5.81562e-08)(5.05986,-5.38951e-08)(5.11192,-4.8889e-08)(5.16398,-4.68009e-08)(5.21603,-4.81725e-08)(5.26809,-5.07196e-08)(5.32015,-5.36805e-08)(5.3722,-5.12641e-08)(5.42426,-5.03726e-08)(5.47631,-5.05688e-08)(5.52837,-5.18869e-08)(5.58043,-5.21662e-08)(5.63248,-4.8108e-08)(5.68454,-4.52193e-08)(5.7366,-4.41298e-08)(5.78865,-4.65294e-08)(5.84071,-4.41288e-08)(5.89276,-3.92736e-08)(5.94482,-3.15278e-08)(5.99688,-2.80072e-08)(6.04893,-2.21175e-08)(6.10099,-1.51606e-08)(6.15305,-7.59983e-09)(6.2051,-2.52116e-09)(6.25716,-1.7969e-09)(6.30921,-1.41675e-09)(6.36127,7.23879e-10)(6.41333,6.97564e-09)(6.46538,1.43514e-08)(6.51744,1.95517e-08)(6.5695,2.05547e-08)(6.62155,1.62876e-08)(6.67361,1.50253e-08)(6.72566,1.88767e-08)(6.77772,2.3911e-08)(6.82978,2.10577e-08)(6.88183,1.7588e-08)(6.93389,1.29279e-08)(6.98594,1.44284e-08)(7.038,1.93594e-08)(7.09006,1.80011e-08)(7.14211,1.42037e-08)(7.19417,9.54518e-09)(7.24623,8.65161e-09)(7.29828,1.3258e-08)(7.35034,1.49231e-08)(7.40239,9.00924e-09)(7.45445,2.17367e-09)(7.50651,4.43473e-09)(7.55856,9.19695e-09)(7.61062,1.06003e-08)(7.66268,9.86686e-09)(7.71473,9.31936e-09)(7.76679,1.23324e-08)(7.81884,1.39361e-08)(7.8709,1.74317e-08)(7.92296,2.33593e-08)(7.97501,2.64835e-08)(8.02707,2.68992e-08)(8.07913,3.02841e-08)(8.13118,3.68101e-08)(8.18324,3.87252e-08)(8.23529,3.25036e-08)(8.28735,2.93548e-08)(8.33941,3.03268e-08)(8.39146,3.3216e-08)(8.44352,3.47554e-08)(8.49558,3.17857e-08)(8.54763,2.85479e-08)(8.59969,2.95511e-08)(8.65174,3.44376e-08)(8.7038,3.69378e-08)(8.75586,3.3702e-08)(8.80791,2.84671e-08)(8.85997,2.53638e-08)(8.91202,3.10969e-08)(8.96408,4.03325e-08)(9.01614,4.02743e-08)(9.06819,3.50539e-08)(9.12025,3.0969e-08)(9.17231,3.64115e-08)(9.22436,4.13008e-08)(9.27642,3.27548e-08)(9.32847,2.41588e-08)(9.38053,1.73819e-08)(9.43259,1.68999e-08)(9.48464,2.01661e-08)(9.5367,1.53921e-08)(9.58876,4.87699e-09)(9.64081,-4.86388e-10)(9.69287,-3.00433e-09)(9.74492,3.32753e-10)(9.79698,1.09248e-09)(9.84904,-9.12398e-09)(9.90109,-1.90228e-08)(9.95315,-2.69945e-08)(10.0052,-2.66758e-08)};
\addplot+[semithick, mark options={solid, fill=markercolor, scale=1.5},line width=1.5pt]
coordinates{(0.00520562,0)(0.0572618,-1.60674e-09)(0.109318,8.79859e-10)(0.161374,3.25566e-09)(0.213431,3.97119e-09)(0.265487,5.30783e-09)(0.317543,5.26604e-09)(0.369599,6.14867e-09)(0.421655,8.66441e-09)(0.473712,9.97581e-09)(0.525768,1.0055e-08)(0.577824,9.54419e-09)(0.62988,9.97375e-09)(0.681936,1.02958e-08)(0.733993,9.81252e-09)(0.786049,9.21027e-09)(0.838105,9.29312e-09)(0.890161,1.19234e-08)(0.942218,1.33778e-08)(0.994274,1.29208e-08)(1.04633,1.06562e-08)(1.09839,7.7939e-09)(1.15044,7.80227e-09)(1.2025,7.47556e-09)(1.25455,6.65028e-09)(1.30661,5.29901e-09)(1.35867,4.77395e-09)(1.41072,5.45272e-09)(1.46278,6.42948e-09)(1.51484,6.42984e-09)(1.56689,5.83453e-09)(1.61895,5.95679e-09)(1.671,6.05766e-09)(1.72306,6.0783e-09)(1.77512,6.32215e-09)(1.82717,7.43796e-09)(1.87923,9.09443e-09)(1.93129,9.54042e-09)(1.98334,8.70534e-09)(2.0354,9.30744e-09)(2.08745,9.52954e-09)(2.13951,9.22753e-09)(2.19157,7.90258e-09)(2.24362,6.72646e-09)(2.29568,5.76371e-09)(2.34774,5.92525e-09)(2.39979,5.70354e-09)(2.45185,3.46117e-09)(2.5039,1.92715e-09)(2.55596,1.35099e-09)(2.60802,1.52576e-09)(2.66007,6.54508e-10)(2.71213,-2.49502e-10)(2.76419,-2.27452e-09)(2.81624,-3.75411e-09)(2.8683,-3.92036e-09)(2.92035,-4.56144e-09)(2.97241,-4.30454e-09)(3.02447,-6.05385e-09)(3.07652,-7.71546e-09)(3.12858,-9.01051e-09)(3.18064,-1.00082e-08)(3.23269,-1.11996e-08)(3.28475,-1.18653e-08)(3.3368,-1.28352e-08)(3.38886,-1.22179e-08)(3.44092,-1.08891e-08)(3.49297,-1.07257e-08)(3.54503,-1.19735e-08)(3.59708,-1.33932e-08)(3.64914,-1.16866e-08)(3.7012,-9.90163e-09)(3.75325,-8.76115e-09)(3.80531,-8.6281e-09)(3.85737,-8.94673e-09)(3.90942,-7.64758e-09)(3.96148,-7.12986e-09)(4.01353,-6.70285e-09)(4.06559,-5.77575e-09)(4.11765,-5.3336e-09)(4.1697,-5.86831e-09)(4.22176,-4.87756e-09)(4.27382,-3.51522e-09)(4.32587,-3.36719e-09)(4.37793,-4.16078e-09)(4.42998,-6.35132e-09)(4.48204,-7.32061e-09)(4.5341,-7.84401e-09)(4.58615,-9.16995e-09)(4.63821,-1.14058e-08)(4.69027,-1.27026e-08)(4.74232,-1.32301e-08)(4.79438,-1.40614e-08)(4.84643,-1.598e-08)(4.89849,-1.90408e-08)(4.95055,-2.03769e-08)(5.0026,-1.98956e-08)(5.05466,-1.90596e-08)(5.10672,-1.83453e-08)(5.15877,-1.8227e-08)(5.21083,-1.93541e-08)(5.26288,-2.00899e-08)(5.31494,-1.97675e-08)(5.367,-1.93316e-08)(5.41905,-1.86691e-08)(5.47111,-1.93815e-08)(5.52317,-1.97817e-08)(5.57522,-1.91056e-08)(5.62728,-1.7962e-08)(5.67933,-1.67808e-08)(5.73139,-1.7523e-08)(5.78345,-1.80819e-08)(5.8355,-1.65824e-08)(5.88756,-1.55378e-08)(5.93961,-1.4274e-08)(5.99167,-1.32658e-08)(6.04373,-1.21607e-08)(6.09578,-1.01244e-08)(6.14784,-8.65908e-09)(6.1999,-7.42969e-09)(6.25195,-7.04894e-09)(6.30401,-5.66904e-09)(6.35606,-5.13275e-09)(6.40812,-3.35704e-09)(6.46018,-8.23584e-10)(6.51223,9.20185e-10)(6.56429,1.40148e-09)(6.61635,4.29664e-10)(6.6684,1.0762e-09)(6.72046,2.32563e-09)(6.77251,3.1122e-09)(6.82457,2.05529e-09)(6.87663,9.50817e-11)(6.92868,2.19642e-11)(6.98074,1.6255e-09)(7.0328,2.7949e-09)(7.08485,2.25724e-09)(7.13691,9.48317e-10)(7.18896,-2.84382e-10)(7.24102,2.45548e-10)(7.29308,7.58481e-10)(7.34513,-7.42834e-10)(7.39719,-3.14974e-09)(7.44925,-4.39074e-09)(7.5013,-3.86145e-09)(7.55336,-3.41849e-09)(7.60541,-3.36798e-09)(7.65747,-4.01767e-09)(7.70953,-3.13136e-09)(7.76158,-3.24233e-09)(7.81364,-1.76029e-09)(7.86569,-2.02668e-10)(7.91775,7.95537e-10)(7.96981,3.07193e-10)(8.02186,1.62854e-09)(8.07392,3.68162e-09)(8.12598,5.06758e-09)(8.17803,4.34648e-09)(8.23009,3.11461e-09)(8.28214,3.10681e-09)(8.3342,3.37384e-09)(8.38626,3.94866e-09)(8.43831,3.42741e-09)(8.49037,1.83003e-09)(8.54243,1.0272e-09)(8.59448,1.64016e-09)(8.64654,2.39089e-09)(8.69859,1.99442e-09)(8.75065,-1.56531e-10)(8.80271,-2.01276e-09)(8.85476,-1.82721e-09)(8.90682,-7.369e-10)(8.95888,-8.12405e-10)(9.01093,-2.7693e-09)(9.06299,-5.55232e-09)(9.11504,-5.78742e-09)(9.1671,-5.50798e-09)(9.21916,-7.23382e-09)(9.27121,-9.1654e-09)(9.32327,-1.13515e-08)(9.37533,-1.21362e-08)(9.42738,-1.17282e-08)(9.47944,-1.20132e-08)(9.53149,-1.29612e-08)(9.58355,-1.43223e-08)(9.63561,-1.62033e-08)(9.68766,-1.57219e-08)(9.73972,-1.52722e-08)(9.79178,-1.60385e-08)(9.84383,-1.83368e-08)(9.89589,-2.07927e-08)(9.94794,-2.13182e-08)(10,-2.08628e-08)};
\end{axis}
\end{tikzpicture}
}
\subfloat[$N=3$]{
\begin{tikzpicture}
\begin{axis}[
    width=.49\textwidth,
    xlabel={Time },
    ylabel={$\mathsf{S}(t)-\mathsf{S}(0)$}, 
    ymin=-1.2e-7, ymax=1.2e-7,
    legend pos=north east, legend cell align=left, legend style={font=\tiny},	
    xmajorgrids=true, ymajorgrids=true, grid style=dashed, 
    legend entries={$\text{CFL} = \frac{1}{4}$, $\text{CFL} = \frac{1}{8}$}
    ]
\pgfplotsset{
cycle list={{blue}, {red, dashed}}
}
\addplot+[semithick, mark options={solid, fill=markercolor, scale=1.5},line width=1.5pt]
coordinates{(0.00624805,0)(0.0499844,-1.00578e-09)(0.0937207,4.27791e-09)(0.137457,5.69024e-09)(0.181193,1.01132e-08)(0.22493,1.51474e-08)(0.268666,1.92586e-08)(0.312402,2.53154e-08)(0.356139,2.74302e-08)(0.399875,2.88243e-08)(0.443611,2.90612e-08)(0.487348,3.14072e-08)(0.531084,3.34135e-08)(0.57482,3.19189e-08)(0.618557,3.52928e-08)(0.662293,3.8433e-08)(0.706029,4.14384e-08)(0.749766,4.78918e-08)(0.793502,5.06168e-08)(0.837238,5.69694e-08)(0.880975,6.46637e-08)(0.924711,6.88565e-08)(0.968447,6.99206e-08)(1.01218,7.01085e-08)(1.05592,7.39151e-08)(1.09966,8.07977e-08)(1.14339,8.1123e-08)(1.18713,7.92769e-08)(1.23087,7.94767e-08)(1.2746,7.66804e-08)(1.31834,7.35782e-08)(1.36207,7.0274e-08)(1.40581,7.03478e-08)(1.44955,7.14908e-08)(1.49328,7.00786e-08)(1.53702,7.12516e-08)(1.58076,7.24142e-08)(1.62449,7.07992e-08)(1.66823,7.25048e-08)(1.71197,7.11029e-08)(1.7557,6.67772e-08)(1.79944,6.69249e-08)(1.84317,6.50965e-08)(1.88691,6.11087e-08)(1.93065,5.74553e-08)(1.97438,5.34544e-08)(2.01812,5.35652e-08)(2.06186,5.02835e-08)(2.10559,4.96117e-08)(2.14933,5.06362e-08)(2.19306,4.78819e-08)(2.2368,4.6152e-08)(2.28054,4.61492e-08)(2.32427,4.45295e-08)(2.36801,4.29969e-08)(2.41175,4.52257e-08)(2.45548,4.80921e-08)(2.49922,4.838e-08)(2.54296,4.96141e-08)(2.58669,5.08383e-08)(2.63043,4.6747e-08)(2.67416,4.50343e-08)(2.7179,4.57811e-08)(2.76164,4.49465e-08)(2.80537,4.39347e-08)(2.84911,4.03653e-08)(2.89285,3.76173e-08)(2.93658,3.69779e-08)(2.98032,3.4891e-08)(3.02405,3.45774e-08)(3.06779,3.29115e-08)(3.11153,3.10074e-08)(3.15526,2.97435e-08)(3.199,2.87588e-08)(3.24274,2.84067e-08)(3.28647,2.81302e-08)(3.33021,2.58228e-08)(3.37395,2.48137e-08)(3.41768,2.40986e-08)(3.46142,2.49202e-08)(3.50515,2.39885e-08)(3.54889,2.25155e-08)(3.59263,2.27391e-08)(3.63636,2.1759e-08)(3.6801,2.06663e-08)(3.72384,2.06544e-08)(3.76757,2.05903e-08)(3.81131,2.3061e-08)(3.85505,2.57328e-08)(3.89878,2.72493e-08)(3.94252,2.87574e-08)(3.98625,2.88237e-08)(4.02999,2.84593e-08)(4.07373,2.76746e-08)(4.11746,2.65683e-08)(4.1612,2.56745e-08)(4.20494,2.45629e-08)(4.24867,2.43523e-08)(4.29241,2.35488e-08)(4.33614,2.24155e-08)(4.37988,2.12819e-08)(4.42362,2.00313e-08)(4.46735,1.84941e-08)(4.51109,1.6667e-08)(4.55483,1.49902e-08)(4.59856,1.42444e-08)(4.6423,1.30539e-08)(4.68604,1.17151e-08)(4.72977,1.10038e-08)(4.77351,1.00834e-08)(4.81724,8.26513e-09)(4.86098,6.57791e-09)(4.90472,5.3892e-09)(4.94845,4.65155e-09)(4.99219,4.14356e-09)(5.03593,2.94385e-09)(5.07966,2.84368e-09)(5.1234,2.11751e-09)(5.16714,1.6903e-09)(5.21087,1.44444e-09)(5.25461,1.35602e-09)(5.29834,2.90889e-10)(5.34208,-1.37863e-09)(5.38582,-2.28586e-09)(5.42955,-3.39655e-09)(5.47329,-5.3107e-09)(5.51703,-6.78537e-09)(5.56076,-7.42485e-09)(5.6045,-9.04905e-09)(5.64823,-1.06757e-08)(5.69197,-1.13842e-08)(5.73571,-1.17967e-08)(5.77944,-1.30714e-08)(5.82318,-1.46886e-08)(5.86692,-1.62677e-08)(5.91065,-1.70661e-08)(5.95439,-1.89814e-08)(5.99813,-1.98452e-08)(6.04186,-2.14826e-08)(6.0856,-2.37628e-08)(6.12933,-2.49191e-08)(6.17307,-2.52989e-08)(6.21681,-2.51691e-08)(6.26054,-2.60865e-08)(6.30428,-2.71662e-08)(6.34802,-2.73813e-08)(6.39175,-2.75091e-08)(6.43549,-2.80128e-08)(6.47923,-2.86358e-08)(6.52296,-3.00474e-08)(6.5667,-3.16028e-08)(6.61043,-3.24199e-08)(6.65417,-3.31652e-08)(6.69791,-3.35865e-08)(6.74164,-3.44548e-08)(6.78538,-3.53176e-08)(6.82912,-3.54693e-08)(6.87285,-3.57781e-08)(6.91659,-3.66918e-08)(6.96032,-3.83037e-08)(7.00406,-3.96064e-08)(7.0478,-4.0954e-08)(7.09153,-4.29645e-08)(7.13527,-4.47955e-08)(7.17901,-4.669e-08)(7.22274,-4.87219e-08)(7.26648,-5.07412e-08)(7.31022,-5.20228e-08)(7.35395,-5.31846e-08)(7.39769,-5.40127e-08)(7.44142,-5.52447e-08)(7.48516,-5.66301e-08)(7.5289,-5.84773e-08)(7.57263,-5.99872e-08)(7.61637,-6.11387e-08)(7.66011,-6.24051e-08)(7.70384,-6.30648e-08)(7.74758,-6.35239e-08)(7.79132,-6.43198e-08)(7.83505,-6.54076e-08)(7.87879,-6.66389e-08)(7.92252,-6.82018e-08)(7.96626,-7.01565e-08)(8.01,-7.1955e-08)(8.05373,-7.30281e-08)(8.09747,-7.41271e-08)(8.14121,-7.52431e-08)(8.18494,-7.55822e-08)(8.22868,-7.59917e-08)(8.27241,-7.69124e-08)(8.31615,-7.739e-08)(8.35989,-7.74046e-08)(8.40362,-7.79214e-08)(8.44736,-7.89306e-08)(8.4911,-7.98664e-08)(8.53483,-8.10523e-08)(8.57857,-8.29924e-08)(8.62231,-8.46339e-08)(8.66604,-8.67237e-08)(8.70978,-8.81372e-08)(8.75351,-8.93285e-08)(8.79725,-9.06938e-08)(8.84099,-9.19019e-08)(8.88472,-9.32049e-08)(8.92846,-9.43302e-08)(8.9722,-9.56419e-08)(9.01593,-9.6963e-08)(9.05967,-9.82278e-08)(9.10341,-9.89182e-08)(9.14714,-9.91399e-08)(9.19088,-9.95626e-08)(9.23461,-1.00956e-07)(9.27835,-1.02832e-07)(9.32209,-1.04554e-07)(9.36582,-1.04899e-07)(9.40956,-1.0486e-07)(9.4533,-1.04754e-07)(9.49703,-1.05146e-07)(9.54077,-1.06358e-07)(9.5845,-1.07663e-07)(9.62824,-1.08046e-07)(9.67198,-1.08266e-07)(9.71571,-1.08011e-07)(9.75945,-1.08009e-07)(9.80319,-1.08323e-07)(9.84692,-1.08249e-07)(9.89066,-1.07826e-07)(9.9344,-1.07992e-07)(9.97813,-1.08103e-07)};
\addplot+[semithick, mark options={solid, fill=markercolor, scale=1.5},line width=1.5pt]
coordinates{(0.00312402,0)(0.0499844,-2.99208e-10)(0.0968447,1.17198e-09)(0.143705,1.63079e-09)(0.190565,3.07836e-09)(0.237426,4.06707e-09)(0.284286,5.05992e-09)(0.331147,6.77992e-09)(0.378007,7.20693e-09)(0.424867,7.23481e-09)(0.471728,7.63065e-09)(0.518588,8.48032e-09)(0.565448,7.88693e-09)(0.612309,8.70059e-09)(0.659169,9.42339e-09)(0.706029,1.02816e-08)(0.75289,1.1981e-08)(0.79975,1.25283e-08)(0.84661,1.45007e-08)(0.893471,1.63133e-08)(0.940331,1.72912e-08)(0.987192,1.73582e-08)(1.03405,1.80253e-08)(1.08091,1.94106e-08)(1.12777,2.01078e-08)(1.17463,1.97929e-08)(1.22149,1.9839e-08)(1.26835,1.91927e-08)(1.31521,1.83874e-08)(1.36207,1.75907e-08)(1.40893,1.78968e-08)(1.4558,1.78635e-08)(1.50266,1.76663e-08)(1.54952,1.83358e-08)(1.59638,1.78862e-08)(1.64324,1.82915e-08)(1.6901,1.83774e-08)(1.73696,1.74978e-08)(1.78382,1.73731e-08)(1.83068,1.71178e-08)(1.87754,1.61394e-08)(1.9244,1.5117e-08)(1.97126,1.43712e-08)(2.01812,1.44139e-08)(2.06498,1.34719e-08)(2.11184,1.38099e-08)(2.1587,1.37431e-08)(2.20556,1.31569e-08)(2.25242,1.31188e-08)(2.29928,1.28291e-08)(2.34614,1.22808e-08)(2.393,1.2585e-08)(2.43986,1.34816e-08)(2.48672,1.34394e-08)(2.53358,1.38907e-08)(2.58044,1.40671e-08)(2.6273,1.30231e-08)(2.67416,1.27697e-08)(2.72102,1.26865e-08)(2.76789,1.27166e-08)(2.81475,1.21876e-08)(2.86161,1.1074e-08)(2.90847,1.08202e-08)(2.95533,1.02187e-08)(3.00219,1.01039e-08)(3.04905,9.69258e-09)(3.09591,9.30729e-09)(3.14277,8.9385e-09)(3.18963,8.68756e-09)(3.23649,8.59984e-09)(3.28335,8.28531e-09)(3.33021,7.88736e-09)(3.37707,7.64743e-09)(3.42393,7.76282e-09)(3.47079,7.94433e-09)(3.51765,7.35893e-09)(3.56451,7.29083e-09)(3.61137,7.12729e-09)(3.65823,6.93779e-09)(3.70509,6.86794e-09)(3.75195,6.7119e-09)(3.79881,7.17171e-09)(3.84567,7.76091e-09)(3.89253,8.18843e-09)(3.93939,8.36416e-09)(3.98625,8.2383e-09)(4.03311,8.19582e-09)(4.07998,7.91801e-09)(4.12684,7.69219e-09)(4.1737,7.38021e-09)(4.22056,7.24244e-09)(4.26742,7.03125e-09)(4.31428,6.71834e-09)(4.36114,6.35674e-09)(4.408,5.97646e-09)(4.45486,5.61762e-09)(4.50172,5.13061e-09)(4.54858,4.80554e-09)(4.59544,4.61019e-09)(4.6423,4.2028e-09)(4.68916,3.93831e-09)(4.73602,3.82308e-09)(4.78288,3.53172e-09)(4.82974,3.1203e-09)(4.8766,2.86377e-09)(4.92346,2.78603e-09)(4.97032,2.65636e-09)(5.01718,2.39738e-09)(5.06404,2.26832e-09)(5.1109,1.99756e-09)(5.15776,1.85203e-09)(5.20462,1.75551e-09)(5.25148,1.60317e-09)(5.29834,1.10494e-09)(5.3452,7.44371e-10)(5.39206,5.13407e-10)(5.43893,1.58082e-10)(5.48579,-1.10372e-10)(5.53265,-2.20147e-10)(5.57951,-5.75952e-10)(5.62637,-8.41266e-10)(5.67323,-9.27491e-10)(5.72009,-9.65309e-10)(5.76695,-1.24092e-09)(5.81381,-1.60651e-09)(5.86067,-1.92343e-09)(5.90753,-2.27332e-09)(5.95439,-2.61466e-09)(6.00125,-2.9566e-09)(6.04811,-3.53229e-09)(6.09497,-3.81074e-09)(6.14183,-3.87305e-09)(6.18869,-3.77177e-09)(6.23555,-3.90247e-09)(6.28241,-4.21772e-09)(6.32927,-4.39412e-09)(6.37613,-4.5647e-09)(6.42299,-4.78912e-09)(6.46985,-5.03355e-09)(6.51671,-5.43149e-09)(6.56357,-5.79548e-09)(6.61043,-5.90581e-09)(6.65729,-5.87551e-09)(6.70415,-6.02585e-09)(6.75102,-6.13069e-09)(6.79788,-6.24696e-09)(6.84474,-6.3828e-09)(6.8916,-6.77627e-09)(6.93846,-7.34254e-09)(6.98532,-7.78258e-09)(7.03218,-8.24913e-09)(7.07904,-8.88717e-09)(7.1259,-9.44568e-09)(7.17276,-9.95071e-09)(7.21962,-1.05444e-08)(7.26648,-1.10876e-08)(7.31334,-1.1531e-08)(7.3602,-1.17709e-08)(7.40706,-1.20955e-08)(7.45392,-1.24424e-08)(7.50078,-1.2839e-08)(7.54764,-1.30667e-08)(7.5945,-1.33161e-08)(7.64136,-1.35358e-08)(7.68822,-1.36743e-08)(7.73508,-1.38576e-08)(7.78194,-1.41345e-08)(7.8288,-1.4406e-08)(7.87566,-1.47245e-08)(7.92252,-1.50866e-08)(7.96938,-1.54498e-08)(8.01624,-1.55766e-08)(8.06311,-1.58598e-08)(8.10997,-1.61261e-08)(8.15683,-1.62623e-08)(8.20369,-1.64666e-08)(8.25055,-1.66967e-08)(8.29741,-1.68577e-08)(8.34427,-1.71167e-08)(8.39113,-1.74259e-08)(8.43799,-1.7669e-08)(8.48485,-1.79491e-08)(8.53171,-1.8289e-08)(8.57857,-1.87072e-08)(8.62543,-1.92366e-08)(8.67229,-1.97203e-08)(8.71915,-2.01144e-08)(8.76601,-2.05831e-08)(8.81287,-2.09835e-08)(8.85973,-2.13913e-08)(8.90659,-2.17191e-08)(8.95345,-2.20523e-08)(9.00031,-2.23577e-08)(9.04717,-2.26378e-08)(9.09403,-2.28626e-08)(9.14089,-2.297e-08)(9.18775,-2.32556e-08)(9.23461,-2.3654e-08)(9.28147,-2.40687e-08)(9.32833,-2.42578e-08)(9.3752,-2.43351e-08)(9.42206,-2.4376e-08)(9.46892,-2.45643e-08)(9.51578,-2.48449e-08)(9.56264,-2.52732e-08)(9.6095,-2.54822e-08)(9.65636,-2.54807e-08)(9.70322,-2.54233e-08)(9.75008,-2.55252e-08)(9.79694,-2.55877e-08)(9.8438,-2.55184e-08)(9.89066,-2.53737e-08)(9.93752,-2.53551e-08)(9.98438,-2.52858e-08)};
\end{axis}
\end{tikzpicture}
}
\caption{Change in entropy over time for entropy conservative formulations of the compressible Euler equations and the implicit midpoint method. }
\label{fig:euler_ec}
\end{figure}
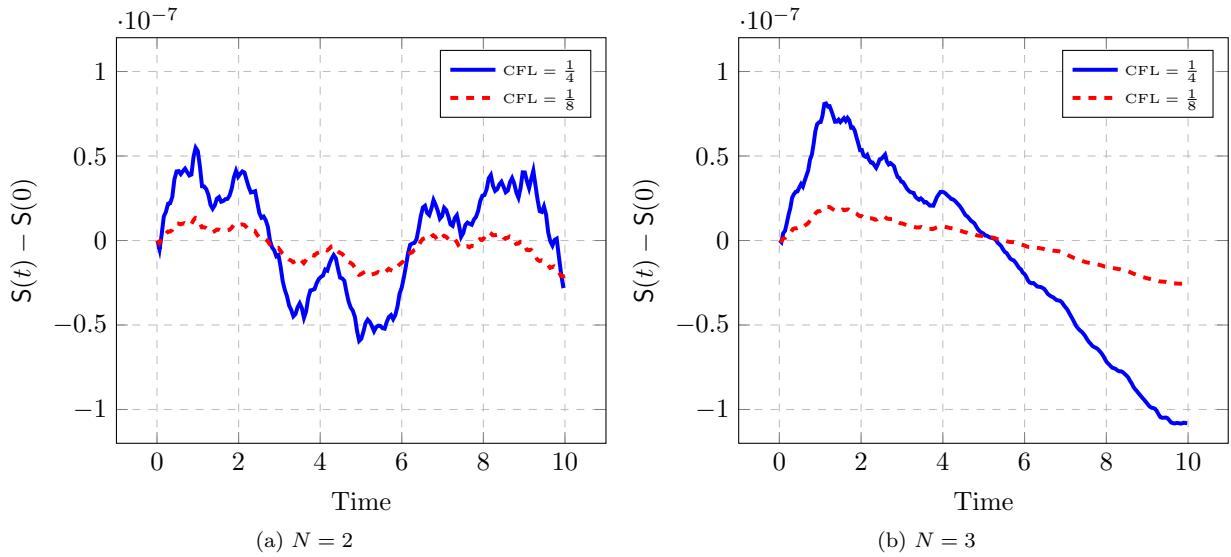
Let $\mathsf{S}(t) = \int_{\Omega}S(\bm{u}(\bm{x},t))$ denote the total entropy in the domain $\Omega$, where the integral is approximated using the same quadrature rule used to construct the DG mass matrix over each element. We begin by checking the change in entropy $\mathsf{S}(t)-\mathsf{S}(0)$ for an entropy conservative formulation. We utilize a discontinuous initial condition
\[
\rho = \begin{cases}
1.1 & -.5 \leq x,y \leq .5\\
1 & \text{otherwise}
\end{cases}, \qquad u,v = 0, \qquad E = \rho^\gamma.
\]
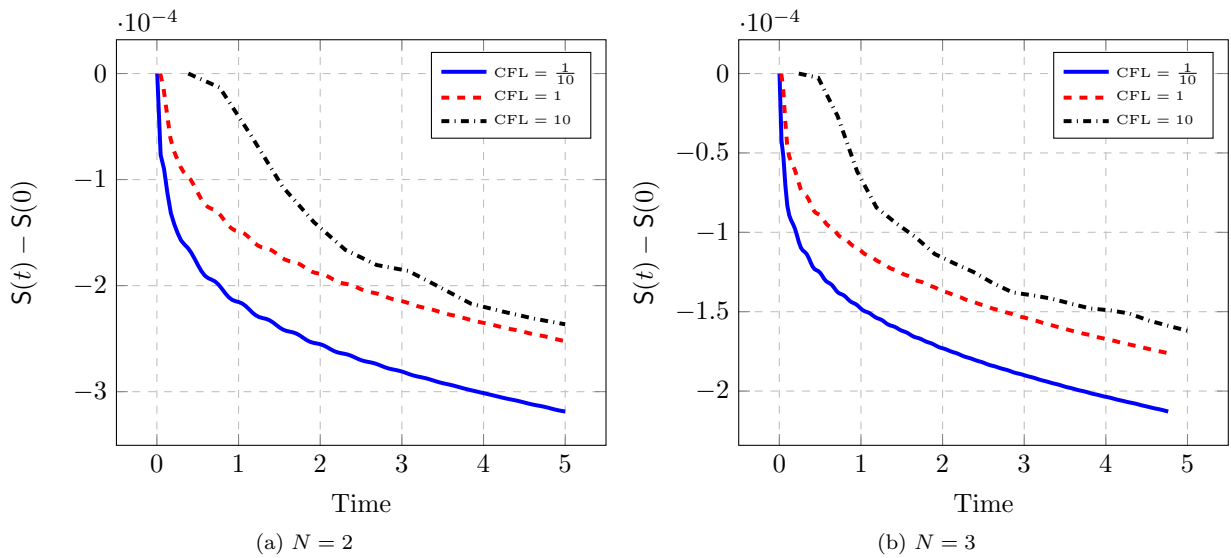
\begin{figure}
\centering
\subfloat[$N=2$]{
\begin{tikzpicture}
\begin{axis}[
    width=.49\textwidth,
    xlabel={Time },
    ylabel={$\mathsf{S}(t)-\mathsf{S}(0)$}, 
    %xmax=3.5,
%    ymin=-1.2e-7, ymax=1.2e-7,
    legend pos=north east, legend cell align=left, legend style={font=\tiny},	
    xmajorgrids=true, ymajorgrids=true, grid style=dashed, 
    legend entries={$\text{CFL} = \frac{1}{10}$, $\text{CFL} = 1$, $\text{CFL} = 10$}
]
\pgfplotsset{
cycle list={{blue}, {red, dashed},{black,dashdotted}}
}
\addplot+[semithick, mark options={solid, fill=markercolor, scale=1.5},line width=1.5pt]
coordinates{(0.0041632,0)(0.0457952,-7.66627e-05)(0.0874271,-8.99814e-05)(0.129059,-0.00011274)(0.170691,-0.000131392)(0.212323,-0.000141949)(0.253955,-0.000150471)(0.295587,-0.000157265)(0.337219,-0.000161203)(0.378851,-0.000164431)(0.420483,-0.000168926)(0.462115,-0.000174629)(0.503747,-0.000180738)(0.545379,-0.000186371)(0.587011,-0.000190571)(0.628643,-0.00019297)(0.670275,-0.000194342)(0.711907,-0.000196032)(0.753539,-0.000198899)(0.795171,-0.000202852)(0.836803,-0.000207134)(0.878435,-0.000210834)(0.920067,-0.000213336)(0.961699,-0.000214665)(1.00333,-0.000215522)(1.04496,-0.000216849)(1.08659,-0.000219162)(1.12823,-0.000222253)(1.16986,-0.000225439)(1.21149,-0.000228026)(1.25312,-0.000229676)(1.29475,-0.000230576)(1.33639,-0.000231335)(1.37802,-0.000232589)(1.41965,-0.000234584)(1.46128,-0.000237075)(1.50291,-0.000239537)(1.54455,-0.000241475)(1.58618,-0.000242691)(1.62781,-0.000243402)(1.66944,-0.000244115)(1.71107,-0.000245308)(1.75271,-0.000247134)(1.79434,-0.000249347)(1.83597,-0.000251462)(1.8776,-0.000253046)(1.91923,-0.000253988)(1.96087,-0.000254571)(2.0025,-0.000255277)(2.04413,-0.000256459)(2.08576,-0.000258132)(2.12739,-0.000260007)(2.16903,-0.000261688)(2.21066,-0.000262907)(2.25229,-0.000263666)(2.29392,-0.000264234)(2.33555,-0.000264967)(2.37719,-0.000266077)(2.41882,-0.000267517)(2.46045,-0.000269029)(2.50208,-0.000270325)(2.54371,-0.000271261)(2.58535,-0.000271914)(2.62698,-0.000272527)(2.66861,-0.000273344)(2.71024,-0.000274455)(2.75187,-0.000275758)(2.79351,-0.00027703)(2.83514,-0.000278077)(2.87677,-0.000278851)(2.9184,-0.000279472)(2.96003,-0.000280151)(3.00167,-0.000281044)(3.0433,-0.000282151)(3.08493,-0.000283334)(3.12656,-0.000284414)(3.16819,-0.000285289)(3.20983,-0.000285984)(3.25146,-0.000286623)(3.29309,-0.000287345)(3.33472,-0.000288225)(3.37635,-0.000289232)(3.41799,-0.00029026)(3.45962,-0.000291197)(3.50125,-0.000291996)(3.54288,-0.000292694)(3.58451,-0.000293385)(3.62614,-0.000294152)(3.66778,-0.000295011)(3.70941,-0.000295914)(3.75104,-0.000296794)(3.79267,-0.000297614)(3.8343,-0.000298385)(3.87594,-0.000299134)(3.91757,-0.000299885)(3.9592,-0.000300642)(4.00083,-0.000301393)(4.04246,-0.000302134)(4.0841,-0.000302878)(4.12573,-0.000303644)(4.16736,-0.000304435)(4.20899,-0.000305226)(4.25062,-0.000305982)(4.29226,-0.000306678)(4.33389,-0.000307326)(4.37552,-0.000307965)(4.41715,-0.000308647)(4.45878,-0.000309406)(4.50042,-0.000310229)(4.54205,-0.000311062)(4.58368,-0.000311837)(4.62531,-0.000312512)(4.66694,-0.000313107)(4.70858,-0.000313695)(4.75021,-0.000314365)(4.79184,-0.000315152)(4.83347,-0.000316015)(4.8751,-0.000316857)(4.91674,-0.00031759)(4.95837,-0.000318192)(5,-0.000318719)};
\addplot+[semithick, mark options={solid, fill=markercolor, scale=1.5},line width=1.5pt]
coordinates{(0.0413223,0)(0.0826446,-1.44653e-05)(0.123967,-3.80211e-05)(0.165289,-6.06282e-05)(0.206612,-7.26946e-05)(0.247934,-8.04246e-05)(0.289256,-8.8016e-05)(0.330579,-9.31887e-05)(0.371901,-9.62262e-05)(0.413223,-0.000100033)(0.454545,-0.00010553)(0.495868,-0.000111704)(0.53719,-0.000117566)(0.578512,-0.000122397)(0.619835,-0.000125586)(0.661157,-0.00012726)(0.702479,-0.000128545)(0.743802,-0.000130702)(0.785124,-0.000134154)(0.826446,-0.000138393)(0.867769,-0.000142515)(0.909091,-0.000145734)(0.950413,-0.000147702)(0.991736,-0.000148702)(1.03306,-0.000149548)(1.07438,-0.000151097)(1.1157,-0.000153665)(1.15702,-0.00015687)(1.19835,-0.000159973)(1.23967,-0.000162341)(1.28099,-0.000163763)(1.32231,-0.000164534)(1.36364,-0.000165297)(1.40496,-0.000166634)(1.44628,-0.000168708)(1.4876,-0.000171214)(1.52893,-0.000173624)(1.57025,-0.000175486)(1.61157,-0.000176651)(1.65289,-0.000177351)(1.69421,-0.00017807)(1.73554,-0.000179255)(1.77686,-0.000181046)(1.81818,-0.000183221)(1.8595,-0.000185335)(1.90083,-0.000186976)(1.94215,-0.000187998)(1.98347,-0.000188614)(2.02479,-0.000189261)(2.06612,-0.000190319)(2.10744,-0.000191883)(2.14876,-0.000193733)(2.19008,-0.000195492)(2.2314,-0.000196848)(2.27273,-0.000197722)(2.31405,-0.000198306)(2.35537,-0.000198939)(2.39669,-0.000199893)(2.43802,-0.000201216)(2.47934,-0.000202723)(2.52066,-0.000204135)(2.56198,-0.000205242)(2.60331,-0.00020602)(2.64463,-0.000206629)(2.68595,-0.000207309)(2.72727,-0.000208232)(2.7686,-0.000209405)(2.80992,-0.000210684)(2.85124,-0.000211869)(2.89256,-0.000212823)(2.93388,-0.000213553)(2.97521,-0.000214196)(3.01653,-0.000214931)(3.05785,-0.000215864)(3.09917,-0.000216964)(3.1405,-0.000218103)(3.18182,-0.000219135)(3.22314,-0.00021999)(3.26446,-0.000220697)(3.30579,-0.000221361)(3.34711,-0.000222096)(3.38843,-0.000222963)(3.42975,-0.000223944)(3.47107,-0.00022496)(3.5124,-0.000225919)(3.55372,-0.000226766)(3.59504,-0.000227508)(3.63636,-0.000228205)(3.67769,-0.000228929)(3.71901,-0.000229721)(3.76033,-0.000230576)(3.80165,-0.000231455)(3.84298,-0.000232318)(3.8843,-0.000233135)(3.92562,-0.000233895)(3.96694,-0.000234603)(4.00826,-0.000235274)(4.04959,-0.000235938)(4.09091,-0.000236629)(4.13223,-0.000237378)(4.17355,-0.000238189)(4.21488,-0.000239023)(4.2562,-0.000239813)(4.29752,-0.000240505)(4.33884,-0.000241097)(4.38017,-0.000241649)(4.42149,-0.000242248)(4.46281,-0.000242956)(4.50413,-0.000243772)(4.54545,-0.00024463)(4.58678,-0.000245441)(4.6281,-0.000246143)(4.66942,-0.000246733)(4.71074,-0.000247269)(4.75207,-0.000247842)(4.79339,-0.000248527)(4.83471,-0.000249335)(4.87603,-0.000250201)(4.91736,-0.00025102)(4.95868,-0.000251708)(5,-0.000252254)};
\addplot+[semithick, mark options={solid, fill=markercolor, scale=1.5},line width=1.5pt]
coordinates{(0.384615,0)(0.769231,-1.31819e-05)(1.15385,-5.85394e-05)(1.53846,-0.000106661)(1.92308,-0.000140182)(2.30769,-0.000166244)(2.69231,-0.000180554)(3.07692,-0.000186091)(3.46154,-0.00020138)(3.84615,-0.000217047)(4.23077,-0.000224371)(4.61538,-0.000231118)(5,-0.000236457)};
\end{axis}
\end{tikzpicture}
}
\subfloat[$N=3$]{
\begin{tikzpicture}
\begin{axis}[
    width=.49\textwidth,
    xlabel={Time },
    ylabel={$\mathsf{S}(t)-\mathsf{S}(0)$}, 
    %xmax=3.5,
%    ymin=-1.2e-7, ymax=1.2e-7,
    legend pos=north east, legend cell align=left, legend style={font=\tiny},	
    xmajorgrids=true, ymajorgrids=true, grid style=dashed, 
    legend entries={$\text{CFL} = \frac{1}{10}$, $\text{CFL} = 1$, $\text{CFL} = 10$}
    ]
\pgfplotsset{
cycle list={{blue}, {red, dashed},{black,dashdotted}}
}
\addplot+[semithick, mark options={solid, fill=markercolor, scale=1.5},line width=1.5pt]
coordinates{(0.00238095,0)(0.0261905,-4.19766e-05)(0.05,-5.25527e-05)(0.0738095,-7.03108e-05)(0.097619,-8.30838e-05)(0.121429,-8.93286e-05)(0.145238,-9.25202e-05)(0.169048,-9.49233e-05)(0.192857,-9.78758e-05)(0.216667,-0.000101953)(0.240476,-0.000106356)(0.264286,-0.000109694)(0.288095,-0.000111448)(0.311905,-0.000112396)(0.335714,-0.000113732)(0.359524,-0.000116031)(0.383333,-0.000118928)(0.407143,-0.000121518)(0.430952,-0.00012313)(0.454762,-0.000123869)(0.478571,-0.000124469)(0.502381,-0.000125626)(0.52619,-0.000127425)(0.55,-0.000129353)(0.57381,-0.000130801)(0.597619,-0.000131625)(0.621429,-0.000132249)(0.645238,-0.000133236)(0.669048,-0.000134756)(0.692857,-0.000136456)(0.716667,-0.000137814)(0.740476,-0.000138607)(0.764286,-0.000139068)(0.788095,-0.000139641)(0.811905,-0.000140603)(0.835714,-0.000141856)(0.859524,-0.000143061)(0.883333,-0.000143943)(0.907143,-0.000144522)(0.930952,-0.000145084)(0.954762,-0.000145913)(0.978571,-0.000147041)(1.00238,-0.000148228)(1.02619,-0.000149179)(1.05,-0.000149791)(1.07381,-0.000150217)(1.09762,-0.000150721)(1.12143,-0.000151462)(1.14524,-0.000152383)(1.16905,-0.000153269)(1.19286,-0.000153938)(1.21667,-0.000154399)(1.24048,-0.000154853)(1.26429,-0.000155517)(1.2881,-0.000156424)(1.3119,-0.000157388)(1.33571,-0.000158167)(1.35952,-0.000158664)(1.38333,-0.000158996)(1.40714,-0.000159387)(1.43095,-0.000159976)(1.45476,-0.000160719)(1.47857,-0.000161446)(1.50238,-0.000162014)(1.52619,-0.000162434)(1.55,-0.000162867)(1.57381,-0.00016347)(1.59762,-0.000164254)(1.62143,-0.000165062)(1.64524,-0.000165711)(1.66905,-0.000166141)(1.69286,-0.000166458)(1.71667,-0.00016683)(1.74048,-0.000167356)(1.76429,-0.000167993)(1.7881,-0.000168607)(1.8119,-0.000169098)(1.83571,-0.00016948)(1.85952,-0.000169875)(1.88333,-0.000170399)(1.90714,-0.000171058)(1.93095,-0.000171743)(1.95476,-0.000172321)(1.97857,-0.000172742)(2.00238,-0.000173073)(2.02619,-0.000173434)(2.05,-0.000173893)(2.07381,-0.000174422)(2.09762,-0.000174933)(2.12143,-0.000175361)(2.14524,-0.000175726)(2.16905,-0.000176111)(2.19286,-0.000176588)(2.21667,-0.000177149)(2.24048,-0.000177717)(2.26429,-0.000178207)(2.2881,-0.000178592)(2.3119,-0.000178922)(2.33571,-0.000179278)(2.35952,-0.000179702)(2.38333,-0.000180172)(2.40714,-0.000180624)(2.43095,-0.000181012)(2.45476,-0.000181351)(2.47857,-0.000181709)(2.50238,-0.00018214)(2.52619,-0.000182638)(2.55,-0.000183137)(2.57381,-0.000183565)(2.59762,-0.000183904)(2.62143,-0.000184203)(2.64524,-0.000184533)(2.66905,-0.000184924)(2.69286,-0.000185347)(2.71667,-0.000185744)(2.74048,-0.000186077)(2.76429,-0.000186373)(2.7881,-0.000186695)(2.8119,-0.000187094)(2.83571,-0.000187561)(2.85952,-0.000188032)(2.88333,-0.000188436)(2.90714,-0.000188751)(2.93095,-0.000189016)(2.95476,-0.000189297)(2.97857,-0.000189635)(3.00238,-0.00019002)(3.02619,-0.000190404)(3.05,-0.000190744)(3.07381,-0.000191042)(3.09762,-0.000191338)(3.12143,-0.000191678)(3.14524,-0.000192075)(3.16905,-0.000192494)(3.19286,-0.000192887)(3.21667,-0.000193227)(3.24048,-0.000193526)(3.26429,-0.000193826)(3.2881,-0.00019416)(3.3119,-0.000194524)(3.33571,-0.000194883)(3.35952,-0.000195206)(3.38333,-0.000195491)(3.40714,-0.000195775)(3.43095,-0.000196099)(3.45476,-0.000196476)(3.47857,-0.000196878)(3.50238,-0.000197256)(3.52619,-0.000197579)(3.55,-0.000197857)(3.57381,-0.000198134)(3.59762,-0.000198446)(3.62143,-0.000198796)(3.64524,-0.000199151)(3.66905,-0.000199469)(3.69286,-0.000199737)(3.71667,-0.000199985)(3.74048,-0.00020026)(3.76429,-0.000200592)(3.7881,-0.000200965)(3.8119,-0.000201334)(3.83571,-0.000201657)(3.85952,-0.00020193)(3.88333,-0.000202188)(3.90714,-0.000202469)(3.93095,-0.000202783)(3.95476,-0.000203103)(3.97857,-0.000203393)(4.00238,-0.000203643)(4.02619,-0.000203879)(4.05,-0.000204144)(4.07381,-0.000204462)(4.09762,-0.000204816)(4.12143,-0.000205164)(4.14524,-0.000205471)(4.16905,-0.000205739)(4.19286,-0.000205998)(4.21667,-0.000206275)(4.24048,-0.000206574)(4.26429,-0.000206868)(4.2881,-0.000207133)(4.3119,-0.000207371)(4.33571,-0.000207607)(4.35952,-0.000207875)(4.38333,-0.000208187)(4.40714,-0.000208525)(4.43095,-0.000208855)(4.45476,-0.000209157)(4.47857,-0.000209437)(4.50238,-0.000209717)(4.52619,-0.000210012)(4.55,-0.000210317)(4.57381,-0.000210612)(4.59762,-0.00021088)(4.62143,-0.000211126)(4.64524,-0.000211371)(4.66905,-0.000211639)(4.69286,-0.000211938)(4.71667,-0.000212256)(4.74048,-0.000212574)(4.76429,-0.000212878)};
\addplot+[semithick, mark options={solid, fill=markercolor, scale=1.5},line width=1.5pt]
coordinates{(0.0238095,0)(0.047619,-9.30491e-06)(0.0714286,-2.9059e-05)(0.0952381,-4.44567e-05)(0.119048,-5.19045e-05)(0.142857,-5.53619e-05)(0.166667,-5.7953e-05)(0.190476,-6.0774e-05)(0.214286,-6.45016e-05)(0.238095,-6.89514e-05)(0.261905,-7.28532e-05)(0.285714,-7.51727e-05)(0.309524,-7.62008e-05)(0.333333,-7.71661e-05)(0.357143,-7.90175e-05)(0.380952,-8.17552e-05)(0.404762,-8.4633e-05)(0.428571,-8.68041e-05)(0.452381,-8.79493e-05)(0.47619,-8.84927e-05)(0.5,-8.92262e-05)(0.52381,-9.06472e-05)(0.547619,-9.25858e-05)(0.571429,-9.44214e-05)(0.595238,-9.56558e-05)(0.619048,-9.63335e-05)(0.642857,-9.6966e-05)(0.666667,-9.80515e-05)(0.690476,-9.96333e-05)(0.714286,-0.000101289)(0.738095,-0.000102528)(0.761905,-0.000103206)(0.785714,-0.000103614)(0.809524,-0.0001042)(0.833333,-0.0001052)(0.857143,-0.000106482)(0.880952,-0.000107686)(0.904762,-0.000108543)(0.928571,-0.00010909)(0.952381,-0.000109629)(0.97619,-0.000110449)(1,-0.000111576)(1.02381,-0.00011276)(1.04762,-0.000113697)(1.07143,-0.000114285)(1.09524,-0.00011468)(1.11905,-0.000115153)(1.14286,-0.00011588)(1.16667,-0.000116808)(1.19048,-0.000117722)(1.21429,-0.000118421)(1.2381,-0.000118893)(1.2619,-0.000119327)(1.28571,-0.000119952)(1.30952,-0.000120828)(1.33333,-0.000121792)(1.35714,-0.000122595)(1.38095,-0.00012311)(1.40476,-0.000123427)(1.42857,-0.000123769)(1.45238,-0.000124309)(1.47619,-0.000125042)(1.5,-0.000125811)(1.52381,-0.000126442)(1.54762,-0.000126896)(1.57143,-0.000127302)(1.59524,-0.000127837)(1.61905,-0.000128564)(1.64286,-0.000129373)(1.66667,-0.000130072)(1.69048,-0.000130551)(1.71429,-0.000130866)(1.7381,-0.000131183)(1.7619,-0.000131645)(1.78571,-0.000132261)(1.80952,-0.000132917)(1.83333,-0.000133477)(1.85714,-0.000133896)(1.88095,-0.000134261)(1.90476,-0.000134707)(1.92857,-0.000135303)(1.95238,-0.000135986)(1.97619,-0.000136616)(2,-0.000137091)(2.02381,-0.000137428)(2.04762,-0.000137736)(2.07143,-0.000138129)(2.09524,-0.000138633)(2.11905,-0.000139178)(2.14286,-0.000139665)(2.16667,-0.000140056)(2.19048,-0.000140405)(2.21429,-0.000140808)(2.2381,-0.000141317)(2.2619,-0.000141891)(2.28571,-0.000142434)(2.30952,-0.000142867)(2.33333,-0.000143196)(2.35714,-0.000143498)(2.38095,-0.000143857)(2.40476,-0.000144302)(2.42857,-0.000144787)(2.45238,-0.000145234)(2.47619,-0.000145605)(2.5,-0.00014593)(2.52381,-0.000146287)(2.54762,-0.000146725)(2.57143,-0.000147223)(2.59524,-0.000147703)(2.61905,-0.000148097)(2.64286,-0.0001484)(2.66667,-0.000148674)(2.69048,-0.000148994)(2.71429,-0.000149387)(2.7381,-0.000149816)(2.7619,-0.000150214)(2.78571,-0.000150546)(2.80952,-0.000150839)(2.83333,-0.000151161)(2.85714,-0.000151562)(2.88095,-0.000152028)(2.90476,-0.000152488)(2.92857,-0.000152871)(2.95238,-0.000153159)(2.97619,-0.0001534)(3,-0.000153669)(3.02381,-0.00015401)(3.04762,-0.000154406)(3.07143,-0.000154798)(3.09524,-0.000155137)(3.11905,-0.000155423)(3.14286,-0.000155703)(3.16667,-0.000156031)(3.19048,-0.000156422)(3.21429,-0.000156843)(3.2381,-0.000157236)(3.2619,-0.00015757)(3.28571,-0.000157857)(3.30952,-0.000158142)(3.33333,-0.000158465)(3.35714,-0.000158826)(3.38095,-0.00015919)(3.40476,-0.000159519)(3.42857,-0.000159804)(3.45238,-0.000160078)(3.47619,-0.000160387)(3.5,-0.000160752)(3.52381,-0.000161151)(3.54762,-0.000161533)(3.57143,-0.000161858)(3.59524,-0.000162129)(3.61905,-0.000162387)(3.64286,-0.000162677)(3.66667,-0.000163013)(3.69048,-0.000163366)(3.71429,-0.000163692)(3.7381,-0.000163967)(3.7619,-0.000164211)(3.78571,-0.00016447)(3.80952,-0.000164783)(3.83333,-0.000165145)(3.85714,-0.000165515)(3.88095,-0.000165844)(3.90476,-0.000166118)(3.92857,-0.000166363)(3.95238,-0.00016662)(3.97619,-0.000166912)(4,-0.000167224)(4.02381,-0.000167521)(4.04762,-0.000167782)(4.07143,-0.00016802)(4.09524,-0.000168272)(4.11905,-0.000168567)(4.14286,-0.000168905)(4.16667,-0.000169252)(4.19048,-0.000169569)(4.21429,-0.000169843)(4.2381,-0.000170092)(4.2619,-0.000170347)(4.28571,-0.000170621)(4.30952,-0.000170904)(4.33333,-0.000171174)(4.35714,-0.000171422)(4.38095,-0.000171661)(4.40476,-0.000171917)(4.42857,-0.00017221)(4.45238,-0.000172533)(4.47619,-0.000172861)(4.5,-0.000173167)(4.52381,-0.000173445)(4.54762,-0.000173712)(4.57143,-0.000173986)(4.59524,-0.000174272)(4.61905,-0.00017456)(4.64286,-0.00017483)(4.66667,-0.000175079)(4.69048,-0.00017532)(4.71429,-0.000175574)(4.7381,-0.000175855)(4.7619,-0.000176159)(4.78571,-0.000176471)};
\addplot+[semithick, mark options={solid, fill=markercolor, scale=1.5},line width=1.5pt]
coordinates{(0.238095,0)(0.47619,-2.61097e-06)(0.714286,-2.68316e-05)(0.952381,-6.15585e-05)(1.19048,-8.43893e-05)(1.42857,-9.41787e-05)(1.66667,-0.000102974)(1.90476,-0.000113686)(2.14286,-0.000119561)(2.38095,-0.000125068)(2.61905,-0.000132096)(2.85714,-0.000137849)(3.09524,-0.000139645)(3.33333,-0.000141763)(3.57143,-0.000144944)(3.80952,-0.000147999)(4.04762,-0.000149114)(4.28571,-0.000151115)(4.52381,-0.000155522)(4.7619,-0.000158857)(5,-0.000162031)};
\end{axis}
\end{tikzpicture}
}
\caption{Change in entropy over time for entropy stable formulations of the compressible Euler equations and the implicit midpoint method. }
\label{fig:euler_es}
\end{figure}
A triangular mesh is constructed by bisecting each element in a uniform mesh of $8\times 8$ quadrilaterals, and the solution is evolved until final time $T=10$. Figure~\ref{fig:euler_ec} shows the results for $N=2$ and $N=3$ for $\text{CFL} = \frac{1}{4}$ and $\text{CFL} = \frac{1}{8}$. We observe that halving the CFL reduces the change in entropy by a factor of $4$, which corresponds to the second order time accuracy of the implicit midpoint rule. We also check the entropy dissipation for different CFL numbers in Figure~\ref{fig:euler_es}. Both $N=2$ and $N=3$ display similar results, with dissipation decreasing as the CFL increases. We also note that the number of Newton iterations remains relatively constant for different time-step sizes: for $\text{CFL} = .1$, Newton converged in $3-4$ iterations, for $\text{CFL} = 1$, Newton converged in $4-5$ iterations, and for $\text{CFL} = 10$, Newton converged in $5-6$ iterations. We also tried $\text{CFL} = 100$ over a longer time period, for which Newton also converged in $5-6$ iterations. However, for initial conditions with sufficiently large variations, Newton did not converge for $\text{CFL} = 100$. 
\begin{figure}[!h]
\centering
\subfloat[Uniform mesh]{\includegraphics[width=.45\textwidth]{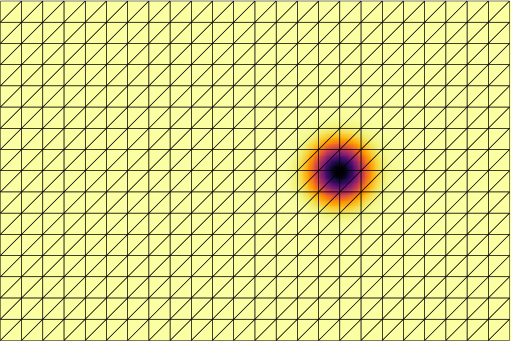}}
\hspace{2em}
\subfloat[Anisotropic mesh]{\includegraphics[width=.45\textwidth]{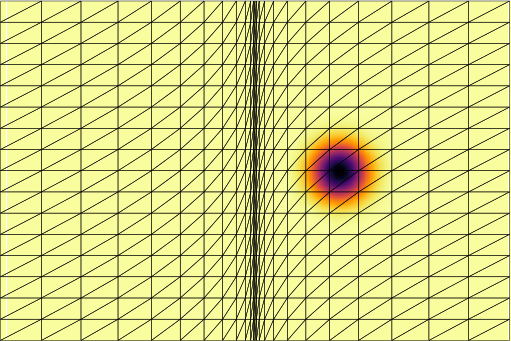}}
\caption{Isentropic vortex solutions at time $T=5$ for $N=3$ and $dt = .1$ on both uniform and ``squeezed'' anisotropic triangular meshes.}
\label{fig:vort}
\end{figure}

Finally, we examine the behavior of the implicit midpoint method with respect to variations in element size. We use the isentropic vortex analytic solution (centered at $x=0, y=5$) on the domain $[-5,5]\times [0,20]$.  We compute the $L^2$ error at time $T=5$ for a uniform and ``squeezed'' anisotropic triangular mesh, both of which are constructed by bisecting each element of a uniform $24\times 16$ quadrilateral mesh. Both cases use a degree $N=3$ approximation and time-step of $dt = .1$, and  Figure~\ref{fig:vort} shows both DG solutions with the mesh overlaid. The $L^2$ errors for the isentropic vortex are $0.0901$ and  $0.0935$ on the uniform and ``squeezed'' meshes, respectively, suggesting that implicit entropy stable formulations robustly handle settings where the maximum stable time-step for explicit methods is restricted by minimum element size.

\end{document}